\documentclass[11pt, reqno]{amsart}

\usepackage[utf8]{inputenc} 
\usepackage[left=2.65cm,right=2.65cm,top=2.7cm,bottom=2.7cm]{geometry}
\usepackage{graphicx} 
\usepackage{color}
\usepackage{mathtools}
\usepackage{dsfont}

\makeatletter

\renewcommand\subsubsection{\@secnumfont}{\bfseries}%
\renewcommand\subsubsection{\@startsection{subsubsection}{3}
  \z@{.5\linespacing\@plus.7\linespacing}{-.5em}%
  {\normalfont\bfseries}}
  
  \makeatother

\usepackage{array} 
\usepackage{paralist} 
\usepackage{verbatim} 
\usepackage{subfig} 

\usepackage{amsfonts}
\usepackage{amsmath}
\usepackage{amsthm}
\usepackage{amssymb}
\usepackage{mathtools}

\usepackage{MnSymbol}

\usepackage{mathtools}

\newcommand{\mel}{\MoveEqLeft}
\usepackage{amsthm}

\newcounter{dummy}

\newtheorem{claim}[dummy]{Claim}

\newtheorem{theorem}{Theorem}[section]

\newtheorem{remark}[theorem]{Remark} 

\newtheorem{example*}{Example\textsuperscript{*}}
\newtheorem{proposition*}{Proposition\textsuperscript{*}}
\newtheorem{corollary}[theorem]{Corollary}
\newtheorem{corollary*}{Corollary\textsuperscript{*}}
\newtheorem{proposition}[theorem]{Proposition}
\newtheorem{assumption}[theorem]{Assumption}

\newtheorem{lemma}[theorem]{Lemma}

\numberwithin{equation}{section}

\newtheorem{manualtheoreminner}{Assumption}
\newenvironment{manualtheorem}[1]{%
  \IfBlankTF{#1}
    {\renewcommand{\themanualtheoreminner}{\unskip}}
    {\renewcommand\themanualtheoreminner{#1}}%
  \manualtheoreminner
}{\endmanualtheoreminner}

\usepackage{bbm}

\def\Limes#1#2 {\lim\limits_{#1\rightarrow #2}}

\DeclareMathOperator{\sgn}{sgn}

\def\eps{\epsilon}

\DeclareMathOperator*{\essinf}{essinf}
\DeclareMathOperator*{\esssup}{esssup}
\DeclareMathOperator{\dist}{dist}

\DeclareMathOperator{\dom}{Dom}
\DeclareMathOperator{\ran}{Ran}
\def\R{\mathbb{R}}

\def\N{\mathbb{N}}

\def\XXint#1#2#3{{\setbox0=\hbox{$#1{#2#3}{\int}$ }
\vcenter{\hbox{$#2#3$ }}\kern-.59\wd0}}

\renewcommand{\div}{\grad\cdot}

\DeclareMathOperator{\diam}{Diam}

\DeclareMathOperator{\BV}{BV}

\def\norm#1{\left\lVert #1 \right\rVert}
\def\scalar#1#2{\langle #1,#2 \rangle}
\def\de{\partial}
\renewcommand{\div}{\operatorname{div}}

\def\dd{\,\mathrm{d}}
\def\dx{\,\mathrm{d}x}

\def\dy{\,\mathrm{d}y}

\def\dH{\,\mathrm{d}\mathcal{H}^{d-1}}

\def\hatz{\hat{z}_{x}}
\def\hatu{\hat{u}_{x}}

\newcommand{\red}[1]{{\textcolor{red}{#1}}}

\usepackage{hyperref}

\title[Boundary attainment in variational problems with linear growth]{On the attainment of boundary data in variational problems with linear growth}
\author{David Meyer}
\address{Instituto de Ciencias Matem\'aticas, Consejo Superior de Investigaciones Cient\'ificas, Calle Nicolás Cabrera 13-15, Campus de Cantoblanco, UAM, 28049,  Madrid, Spain}
\email{david.meyer@icmat.es}
\date{\today}

\keywords{}

\subjclass[2020]{}
\begin{document}

\begin{abstract}
It is well-known that convex variational problems with linear growth and Dirichlet boundary conditions might not have minimizers if the boundary condition is not suitably relaxed.

 We show that for a wide range of integrands, including the least gradient problem and the non-parametric Plateau problem, and under suitable mean-convexity conditions of the boundary, minimizers of the relaxed problem attain the boundary data in the trace sense if it lies in $BV$ or $W^{\alpha,p}$ with $\alpha p\geq 2$ without any kind of continuity assumption. Unlike previous works, our methods are also able to treat systems under a certain quasi-isotropy assumption on the integrand. We further show that without this quasi-isotropy assumption, smooth counterexamples on uniformly convex domains exist.
 
Further applications to the uniqueness of minimizers and to open problems about the ROF functional with Dirichlet boundary conditions, and to the trace space of functions of least gradient are given. 
\end{abstract}

\maketitle

\section{Introduction}
We are concerned with the boundary behaviour of convex variational problems with linear growth and Dirichlet boundary conditions, that is we consider integrands $f:\overline{\Omega}\times \R^{n\times d}\rightarrow \R$, where $\Omega\subset \R^{d}$ is some bounded domain with a sufficiently regular boundary, of which we assume that they are convex in the second variable and there are constants $C_0$ and $C_0'$ such that \begin{align*}
C_0|\xi|-C_0'\leq f(y,\xi)\leq C_0'(1+|\xi|)
\end{align*}
for all $y$ and $\xi$.

For these, we study the minimization problem (potentially also with lower order terms to be specified later on) \begin{align}
\min_{u\in W_{u_0}^{1,1}(\Omega, \R^m)} \int_\Omega f(x,\mathrm{D}u(x))\dx\label{org prob}
\end{align}
for some boundary datum $u_0$. It is classical that this problem will in general not have minimizers due to the failure of weak\textsuperscript{$\ast$}-closedness of $W_{u_0}^{1,1}$ and elementary counterexamples exist\footnote{\label{fn1}It is for instance quite easy to see by symmetry considerations that for $\Omega=B_2(0)\backslash B_1(0)$ and $f(y,\xi)=|\xi|$ with boundary values $1$ on $\de B_1(0)$ and $0$ on $\de B_2(0)$, the limit of any minimizing sequence is $0$.}. Instead, to obtain a solution, one has to replace the functional with its lower semicontinuous hull, yielding   the problem \begin{equation}\begin{aligned}
\mel\min_{u\in BV(\Omega, \R^n)} \int_\Omega f(x,\mathrm{D}^au(x))\dx+\int_\Omega f^\infty(x,\frac{\mathrm{d}\mathrm{D}^s u}{\mathrm{d}|\mathrm{D}^s u|}(x))\dd|\mathrm{D}^s u|(x)\\
&+\int_{\de \Omega} f^\infty(x,(u_0-u)(x)\otimes \nu_x)\dH(x),\label{rel prob}
\end{aligned}\end{equation}
where $\nu_x$ is the (outer) normal  of $\de \Omega$ at $x$ and \begin{align}f^\infty(x,\xi)=\lim_{x'\rightarrow x, t\rightarrow +\infty} \frac{1}{t}f(x',t\xi)\label{rec func}
\end{align}
is the recession function (see Section \ref{S2} for the other definitions), to which the direct method in the calculus of variations is applicable, and a (not necessarily unique) minimizer always exists. \smallskip


It is a natural question under which circumstances these modifications are actually necessary.
The issue of the prevalence of the space $BV$ is reasonably well understood, one can show that under extra ellipticity conditions on the integrand, interior Sobolev or even Hölder regularity does hold, see e.g.\  \cite{bildhauer2003convex,marcellini2006nonlinear,beck2015interior,fussangel2024singular,gmeineder2019partial}, while without ellipticity, easy counterexamples exist.

This work is concerned with the other modification, that is the replacement of the Dirichlet boundary condition with the penalty term $\int_{\de \Omega} f^\infty(x,(u_0-u)(x)\otimes \nu_x)\dH(x)$, whose necessity is a much less explored question and only understood in special cases and only in the scalar case $n=1$.\smallskip

Whether or not the boundary datum is attained largely hinges on curvature conditions on the boundary.

The two most classical cases in which the boundary behaviour is understood are that of the non-parametric Plateau problem, which consists of finding a function whose graph is a minimal surface with a given boundary, and which corresponds to $n=1$ and $f(y,\xi)=\sqrt{1+|\xi|^2}$ and the least gradient problem, which is the Dirichlet problem for the $1$-Laplacian, corresponding to $n=1$ and $f(y,\xi)=|\xi|$.

For the Plateau problem, the first existence results for $d=2$ were obtained by Bernstein over a century ago \cite{bernstein1910surfaces}, while Serrin and Jenkins \cite{serrin1968dirichlet} showed in 1968 that for $d\geq 2$, in domains whose boundary has non-negative mean curvature, a $C^2$ solution exists for any $C^2$ boundary datum, while for every domain with boundary parts of negative mean curvature counterexamples exist.

These classical results can also be shown to work with merely continuous (or a.e.\ continuous) boundary data \cite{miranda1971principio,miranda1974dirichlet} and can also be extended to the prescribed mean curvature equation \cite{giusti1976boundary} (that is the problem \eqref{org prob}/\eqref{rel prob} with $f(y,\xi)=\sqrt{1+|\xi|^2}$ and an additional term $\int_\Omega Hu\dx $, corresponding to a right-hand side term in the Euler-Lagrange equation) or anisotropic versions \cite{corsato2016dirichlet}. We refer the reader e.g.\ to the textbook \cite{giusti1984minimal} for further information. On the other hand, for boundary data which is merely $L^1$, Baldo and Modica constructed an example in \cite{baldo1991non} where the boundary data is not attained in a trace sense and in which the solution is also not unique.\smallskip

For the least gradient problem, which is relevant e.g.\ for the computational methods for minimal surfaces \cite{parks1977explicit}, optimal design  \cite{kohn1986constrained}, or image denoising and inpainting \cite{rudin1992nonlinear,chan2005variational}, existence of solutions for continuous boundary data in domains whose boundary has positive mean curvature was first shown by Sternberg, Williams and Ziemer in 1992 in \cite{sternberg1992existence}, who also established that under these conditions solutions are unique. Again, positive mean curvature is a necessary condition here (the function in the footnote \ref{fn1} earlier is a counterexample for negative curvature).
Spradlin and Tamasan showed in \cite{spradlin2014not} that having some kind of regularity assumption on the boundary data is also necessary, as there exists $L^1$-boundary data for which there is no minimizer attaining the boundary data in the trace sense. G\'orny, on the other hand, showed in \cite{gorny2021existence} that continuity of the boundary datum can be weakened to continuity a.e.

For discontinuous boundary data, uniqueness of minimizers is in general wrong (see e.g.\ \cite[Example 2.7]{mazon2014functions}), though quite a lot of structure on the set of minimizers is enforced \cite{moradifam2018existence}.

The generalisation to the anisotropic least gradient problem, in which $f$ is an $x$-dependent norm, was, motivated by applications in conductivity imaging, considered by Jerrard, Moradifam, and Nachman in \cite{jerrard2018existence}, where it is shown that solutions exist for continuous boundary data if the mean curvature condition is suitably adapted to $f$ and are unique under additional regularity assumptions on $f$.
 
We refer to \cite{gorny2017special,hakkarainen2015stability,dweik2019p,zuniga2019continuity, mercier2018continuity,mazon2016nonlocal} for some further recent developments and to the textbook \cite{gorny2024functions} for general background reading.

A few works on the boundary behaviour of other integrands with linear growth do exist, for instance in \cite{beck2018globally} Beck, Bul\'i\v{c}ek and Maringov\'a showed that under stronger ellipticity conditions on the integrand, which do not hold in the aformentioned special cases, solutions attaining the boundary values do exist even without a curvature condition on the boundary, see also \cite{bulivcek2015existence} for an earlier similar result. In \cite{lledos2025study}, plasticity-type problems were considered.\smallskip

The purpose of this work is to prove existence results in a more general setting, which covers both the minimal surface equation and the least gradient problem, as well as other integrands with linear growth, and extends the aforementioned results to boundary values that have (fractional) Sobolev or $BV$ regularity instead of continuity and to non-scalar-valued problems. 
To the best of the author's knowledge, these are the first results for boundary attainment both for fully discontinuous data and for systems.


\subsection{Precise Set-up and main results}
\subsubsection{The scalar case $n=1$}

Our approach is the following: We consider $u$ which are minimizers of the problem \begin{equation}\begin{aligned}
\mel\min_{\substack{u\in BV(\Omega),\\ \int_\Omega \lambda |u|^2\dx<\infty}}\int_\Omega f(x,\mathrm{D}^au(x))\dx+\int_\Omega f^\infty(x,\frac{\mathrm{d}\mathrm{D}^s u}{\mathrm{d}|\mathrm{D}^s u|}(x))\dd|\mathrm{D}^s u|(x)\\
&+\int_{\de \Omega}f^\infty(x,(u_0-u)\otimes \nu_x)\dH(x)+\int_\Omega gu+\frac{\lambda(x)}{2}|u-h|^2\dx\label{gen rel prob}
\end{aligned}\end{equation}
(definitions can be found in Section \ref{S2}), which is the relaxation of the problem \begin{equation}\begin{aligned}
\mel\min_{\substack{u\in W_{u_0}^{1,1}(\Omega),\\ \int_\Omega \lambda |u|^2\dx<\infty}}\int_\Omega f(x,\mathrm{D}u(x))\dx+\int_\Omega gu+\frac{\lambda(x)}{2}|u-h|^2\dx.\label{gen org prob0}
\end{aligned}\end{equation}
A more detailed account of the relation between these problems is given in Section \ref{S3}.

As already mentioned, the relaxed problem \eqref{gen rel prob} is amenable to the direct method, presuming $g$ is sufficiently small and $\lambda\geq 0$ (see Proposition \ref{ex min} for details), and therefore the existence of minimizers is quite straightforward, though they might not be unique.

Our results will then show that \textit{any} minimizer of the problem \eqref{gen rel prob} must have trace $u_0$ under suitable assumptions.

Regarding the domain, we assume the following:

\begin{manualtheorem}{A1}\label{A1} $\Omega\subset \R^d$ is open, bounded and has Lipschitz boundary. We let $U\subset\de \Omega$ be an open subset of the boundary, of which we assume that it is $C^2$ and that its (relative) boundary in $\de \Omega$ is $C^2$. 
\end{manualtheorem}

The assumption about the relative boundary is not restrictive because $u=u_0$ is a purely local property and therefore, if $\de U$ is not regular, one can cover $U$ with more regular subsets and apply the theorem to each of them.

Our assumptions on the integrand $f$ in the scalar setting $n=1$ are the following:

\begin{manualtheorem}{A2}\label{A2} $f:\overline{\Omega}\times \R^d\rightarrow \R$ is such that:
\begin{itemize}
\item $f(x,\cdot)$ is convex for each $x\in \overline{\Omega}$. 
\item There is a constant $C_1>0$, not depending on $x$ or $\xi$, such that for all $x,\xi$, it holds that \begin{align}
\frac{1}{C_1}|\xi|-C_1\leq f(x,\xi)\leq C_1(|\xi|+1).\label{bd gf1}
\end{align}
\item There is some $R>0$, such that for each $x\in\overline{\Omega}$, the function $f(x,\cdot)$ is in $C^2(\R^{d}\backslash B_R(0))$ and it holds that \begin{align}
\sup_{\xi:\, \xi\geq |a|}|\xi|\int_{|\xi|}^\infty \scalar{\mathrm{D}_\xi^2f(x,s\frac{\xi}{|\xi|})\frac{\xi}{|\xi|}}{\frac{\xi}{|\xi|}}\dd s\xrightarrow{|a|\rightarrow +\infty} 0\label{conv df}
\end{align}
uniformly in $x$ in a neighborhood of $U$.
\item $f$ is continuous in the joint variable $(x,\xi)\in \overline{\Omega}\times \R^d$. 
\item The limit \eqref{rec func} in the definition of $f^\infty$ exists for all $x\in \overline{\Omega}$ and all $\xi\in \R^d$.
\item Regarding the regularity of the recession function, defined in \eqref{rec func}, we assume that\footnote{Here $C_x^1C_\xi^1$ denotes the space of continuous functions $v$ for which the derivatives $\mathrm{D}_x v,\mathrm{D}_\xi v,\mathrm{D}_x\mathrm{D}_\xi v$ exist as classical derivatives and are continuous in the joint variable $(x,\xi)$} \begin{enumerate}
\item $f^\infty\in C_x^1C_\xi^1(\overline{\Omega},B_2(0)\backslash B_{\frac{1}{2}}(0))$ 
\item $f^\infty \in C_\xi^2(B_2(0)\backslash B_{\frac{1}{2}}(0))$ and it holds that $\sup_{x\in \overline{\Omega}} \norm{f(x,\cdot)}_{C^2(B_2(0)\backslash B_{\frac{1}{2}}(0))}<\infty$.
\end{enumerate}
\end{itemize}
\end{manualtheorem}

This includes the cases $f(x,\xi)=|\xi|, \sqrt{1+|\xi|^2}$ as well as $x$-dependent versions thereof (presuming the $x$-dependence is sufficiently regular), and the function $\min(|\xi|^2,|\xi|)$ used for Hencky plasticity as well as many others.\smallskip

Obviously, the mean curvature condition needs to be adapted to $f$. We let $\mathfrak{d}:\Omega\rightarrow \R_{>0}$ denote the distance to the boundary of $\de \Omega$, which is a $C^2$ function in a neighborhood of each $x\in U$ by the implicit function theorem. We then set \begin{align}
H_{\de\Omega,f}(x):=\min\left(-\div(\mathrm{D}_\xi f^\infty(x,\mathrm{D}\mathfrak{d}(x))),\,\div(\mathrm{D}_\xi f^\infty(x,-\mathrm{D}\mathfrak{d}(x)))\right)\label{def hf}
\end{align}
for $x\in U$.
If $f^\infty$ is even, both terms in the minimum agree, and this is essentially the same definition as used by Jerrard, Moradifam, and Nachman for the anisotropic least gradient problem in \cite{jerrard2018existence}. In this case, this generalizes the mean curvature (normalized to be $d-1$ for the unit sphere) as it is the first variation of the perimeter-type functional $\int_{\de \Omega} f^\infty(x,\nu_x)\dH(x)$. In particular, if $f^\infty=|\cdot|$, then this is the classical mean curvature.

On the other hand, if $f$ is not even, these terms might not be the same, and it is not difficult to construct one-dimensional examples showing that one needs to take both terms into account.
Let us also note that, as a consequence of \ref{A1} and \ref{A2}, $H_{\de\Omega,f}$ is continuous on $U$.

Regarding $g$ and the curvature we assume that: \begin{manualtheorem}{A3}\label{A3}
$g\in L^{\max(2,d)}(\Omega)$ is bounded in a neighborhood of $U$ and there is a $c>0$ such that it holds that\footnote{Here we define the essential supremum near a point as  $\essinf_{\Omega\ni y\rightarrow x}a(y):=\sup_{m\rightarrow \infty}\essinf_{B_{\frac{1}{m}}(x)\cap \Omega} a$ for any measurable function $a$} \begin{align}
\essinf_{\Omega\ni y\rightarrow x}  H_{\de\Omega,f}(x)-|g(y)|>c
\end{align}
for all $x\in U$.
\end{manualtheorem}

In particular this implies that \begin{align*}
\inf_{x\in U}H_{\de\Omega,f}(x)>0.
\end{align*}
Let us also point out that $g\in L^d(\Omega)$ ensures that the integral $\int gu\dx$ in the functional exists because $BV\hookrightarrow L^{\frac{d}{d-1}}$.

We make the following assumptions on $\lambda$ and $h$:
\begin{manualtheorem}{A4}\label{A4}
 $ \lambda\in L^\infty(\Omega)$ and $\lambda\geq 0$ and $h\in L^2(\Omega)$.
\end{manualtheorem}

\begin{manualtheorem}{A5}\label{A5}
\textbf{At least one}, but not necessarily more, of the following three statements holds: \begin{align}
&\lambda=0 \quad  \text{in a neighborhood of $U$.}\\
&h\in BV  \quad  \text{in a neighborhood of $U$.}\\ 
& h\in C^0\quad  \text{in a neighborhood of $U$.}
\end{align}
In the second or third case, we also require \begin{align}
h=u_0\text{ on $U$.}
\end{align}
\end{manualtheorem}

We remark that if the third condition holds, then $u_0$ must be automatically continuous.

Regarding the regularity of the boundary datum, we consider the following cases: \begin{manualtheorem}{A6}\label{A6}
\textbf{At least one} (but not necessarily more) of the following three conditions holds: \begin{itemize}
\item $u_0\in BV(U)\cap L^1(\de\Omega)$
\item $u_0\in W^{\alpha,p}(U)\cap L^1(\de\Omega)$ for some $\alpha\in (0,1)$ and $p\in (2,\infty)$ with $\alpha p\geq 2$
\item $u_0\in C^0(U)\cap L^1(\de\Omega)$ and ($f=f^\infty$ and ($\lambda=0$ or $h\in C^0$ in a neighborhood of $U$))
\end{itemize}
\end{manualtheorem}
The definitions of the spaces are also recalled in Sections \ref{S22} and \ref{S24}. Let us also stress that, unlike previous works, this does not require $u_0$ (or any of its representatives) to be continuous at any point.

 We also point out that the case $\lambda=g=h=0$ is of course included here, and in this case the assumptions \ref{A3}-\ref{A5} just reduce to $\inf H_{\de\Omega,f}>0$.

\begin{theorem}\label{T1}
Suppose that the assumptions \ref{A1}, \ref{A2}, \ref{A3}, \ref{A4}, \ref{A5} and \ref{A6} above hold and that $u\in BV(\Omega)\cap L^2(\Omega)$ is a minimizer of the problem \eqref{gen rel prob}. Then \begin{align*}
u=u_0 \quad \text{ in $U$}
\end{align*}
in the trace sense.
\end{theorem}

The extra assumption that $u\in L^2$ is a technical limitation due to the fact that the version of the Euler-Lagrange equation which is used in the proof (see Lemma \ref{sca char}) is in full generality only available for $u\in L^2$. We expect that using the ideas from the $L^1$-theory of the parabolic versions of these problems \cite{andreu2001minimizing,andreu2001dirichlet}, it is probably possible to remove this assumption, at least in some cases. In fact, in the simplest special case of the least gradient problem, we will show in Corollary \ref{cor sans l2} that this can be removed with an easy truncation argument.

 Nevertheless, it is usually not difficult to verify that $u\in L^2$, it is for instance always true if there is a maximum principle (see \cite[Appendix D]{beck2013dirichlet} for a detailed discussion of when this is true), if $\lambda> 0$, or if $d=2$ it follows from the Sobolev inequality since $u\in BV(\Omega)\hookrightarrow L^2(\Omega)$ if $d=2$.
 
\begin{corollary}\label{Cor1}
Assume that \ref{A1}-\ref{A6} hold with $U=\de \Omega$. Assume that $u\in L^2(\Omega)$ is a minimizer of \eqref{gen rel prob}, then $u$ is also a minimizer of the problem \begin{equation}\begin{aligned}
\min_{\substack{\bar{u}\in BV_{u_0}(\Omega),\\ \int_\Omega \lambda \bar{u}^2\dx<\infty}}\int_\Omega f(x,\mathrm{D}^a\bar{u}(x))\dx+\int_\Omega f^\infty(x,\frac{\mathrm{d}\mathrm{D}^s \bar{u}}{\mathrm{d}|\mathrm{D}^s \bar{u}|}(x))\dd|\mathrm{D}^s \bar{u}|(x)+\int_\Omega g\bar{u}+\frac{\lambda(x)}{2}|\bar{u}-h|^2\dx.\label{gen org prob}
\end{aligned}\end{equation}
If also $u\in W^{1,1}(\Omega)$, then $u$ is also a minimizer for the original problem \eqref{gen org prob0}.
\end{corollary} 
\begin{proof} 
 It is clear from the theorem that $u\in BV_{u_0}(\Omega)$. 
 It follows from the general theory for the relaxation of linear growth problems (see Proposition \ref{same inf} below) that the problems \eqref{gen org prob0}, \eqref{gen org prob}, and \eqref{gen rel prob} all have the same infimum and hence $u$ is a minimizer of \eqref{gen org prob}.
\end{proof}

\subsubsection{The vectorial case $n\geq 1$}
In the vectorial case, we do not include lower order terms and assume that $f$ is autonomous, i.e.\ we consider minimizers $u$ of the following problem: \begin{equation}\begin{aligned}
\mel\min_{u\in BV(\Omega, \R^n)} \int_\Omega f(\mathrm{D}^au(x))\dx+\int_\Omega f^\infty(\frac{\mathrm{d}\mathrm{D}^s u}{\mathrm{d}|\mathrm{D}^s u|}(x))\dd|\mathrm{D}^s u|(x)\\
&+\int_{\de \Omega} f^\infty((u_0-u)(x)\otimes \nu_x)\dH(x),\label{vproblem}
\end{aligned}\end{equation}
which is the relaxed version of the problem \begin{align}
\mel\min_{u\in W_{u_0}^{1,1}(\Omega, \R^n)} \int_\Omega f(\mathrm{D}u(x))\dx.\label{og vproblem}
\end{align}
Again, \eqref{vproblem} is the natural extension of Problem \eqref{og vproblem} and has (not necessarily unique) minimizers by the direct method (cf.\ Propositions \ref{same inf} and \ref{ex min}).

Regarding the domain, we shall assume that:
 \begin{manualtheorem}{B1}\label{B1}
$\Omega\subset \R^d$ is a bounded domain with Lipschitz boundary. $U\subset \de\Omega$ is a subset of the boundary which can be written as a graph of a uniformly convex $C^2$-function and whose relative boundary in $\de\Omega$ is $C^2$.
\end{manualtheorem}

Again, the assumption that $U$ is a graph with a regular boundary is not restrictive, as the theorem's statement is purely local.

Regarding the integrand $f$ we assume:

\begin{manualtheorem}{B2}\label{B2}
 \begin{itemize}
\item $f:\R^{n\times d}\rightarrow \R$ is convex. 
\item There is a constant $C_1>0$, such that \begin{align}
\frac{1}{C_1}|\xi|-C_1\leq f(\xi)\leq C_1(|\xi|+1).\label{bd gf2}
\end{align}
 \item There is an $R>0$, such that $f\in C^2(\R^{n \times d}\backslash B_R(0))\cap C^1(\R^{n\times d}\backslash \{0\})$, and it holds that \begin{align}
\sup_{\xi:\, \xi\geq |a|}|\xi|\int_{|\xi|}^\infty \scalar{\mathrm{D}_\xi^2f(s\frac{\xi}{|\xi|})\frac{\xi}{|\xi|}}{\frac{\xi}{|\xi|}}\dd s\xrightarrow{|a|\rightarrow +\infty} 0.\label{almost hom}
\end{align}
\item $f^\infty\in C^2(\R^{n\times d}\backslash\{0\})$.
\item There are positively $0$-homogeneous and odd functions $\mu_1$ and $\mu_2$ such that \begin{align}
\mathrm{D}_\xi f^\infty(v_1\otimes v_2)=\mu_1(v_1)\otimes \mu_2(v_2)\quad  \text{for all $v_1,v_2\neq 0$.}\label{r1 split}
\end{align}
\item There is a $C_2>0$ such that for all $\xi\neq 0$ it holds that\begin{align}
\mathrm{D}_\xi^2((f^\infty)^2)(\xi)\geq C_2 \mathrm{Id}\quad  \text{in the sense of positive definiteness}.\label{uni conv}
\end{align}
\item If for some $\xi\neq 0$, there is a $\xi'$ with $\mathrm{D}_\xi f^\infty(\xi')=\mathrm{D}_\xi f(\xi)$ then $\mathrm{D}_\xi f^\infty(\xi')=\mathrm{D}_\xi f^\infty(\xi)$.
\end{itemize}
\end{manualtheorem}

We remark that \eqref{r1 split} implies that $f^\infty$ is even.\smallskip

Regarding the boundary data we assume that: \begin{manualtheorem}{B3}\label{B3}
\begin{align}
u_0\in L^1(\de\Omega,\R^n)\cap BV(U,\R^n)\cap L^2(U,\R^n).\label{re u02}
\end{align}
\end{manualtheorem}

Some model integrands for which these assumptions hold are for instance $f(\xi)=|\xi|, \sqrt{1+|\xi|^2}$ or $\sqrt{\sum_{i=1}^n A\xi_{i,\bullet}\cdot \xi_{i,\bullet}}$ for any symmetric positive definite $A\in \R^{d\times d}$.
 Let us stress that this does \textit{not} contain the minimal surface problem in higher codimension as it is not convex (see e.g.\ \cite{acerbi1994new}).

\begin{remark}
Let us comment on the additional assumptions:
\begin{itemize}
\item \eqref{r1 split} is e.g.\ fulfilled when $f^\infty$ is isotropic, in the sense that $f^\infty=f^\infty(|\cdot|)$. For instance, in the context of the application to image restoration of color images, this is a completely natural assumption (see e.g.\ \cite{blomgren1998color}). We will also show that without this assumption, there are smooth counterexamples in Theorem \ref{T3}. 

\item \eqref{uni conv} takes the role of the curvature condition, together with the convexity assumption on the boundary. It can easily be shown that this condition is equivalent to $f^\infty$ being uniformly convex with a quadratic modulus of convexity. 

\item The last assumption on $f$ is presumably an artifact of the proof, it holds e.g.\ when $f$ is $1$-homogeneous or when $f$ is rotationally symmetric. Functions fulfilling the other assumptions but not this do however exist, an example is e.g.\ the integrand \begin{align*}
&f(\xi)=\begin{cases}
\tilde{\xi} &\text{ for $\tilde{\xi}\leq 1$}\\
2\tilde{\xi}^\frac{1}{2}-1&\text{ for $\tilde{\xi}\geq 1$.}
\end{cases}\\
\text{where }\:\:&\tilde{\xi}=\xi_{11}^2+\xi_{12}^2+\xi_{21}^2+(\xi_{22}-1)^2.
\end{align*}
One can check that this is $C^2$ for $|\xi|$ large and $C^1$ everywhere, has the recession function $|\xi|$, and that for $\tilde{\xi}\geq1$, the last condition is always violated.

\item The $L^2$-assumption in \eqref{re u02} comes from a technical limitation of the Anzelotti pairing, which is used in the proof, and it might be possible to remove it using more elaborate versions of the pairing (see e.g.\ \cite{crasta2019anzellotti}).

\item The existence of the recession function holds automatically for autonomous convex integrands and therefore does not need to be assumed.
\end{itemize}
\end{remark}

Our main theorem in this case is the following.

\begin{theorem}\label{T2}
Suppose that the assumptions \ref{B1}, \ref{B2} and \ref{B3} hold and let $u\in BV(\Omega,\R^n)\cap L^2(\Omega,\R^n)$ be a minimizer of the relaxed functional \eqref{vproblem}. Then \begin{align*}
u=u_0 \quad \text{on $U$}
\end{align*}
in the trace sense.
\end{theorem}

\begin{corollary}
Suppose \ref{B1}-\ref{B3} hold and suppose that $\de \Omega=U$. Then, if $u\in L^2(\Omega,\R^n)$ is a minimizer of \eqref{vproblem}, then $u$ is also a minimizer of \begin{align}
\min_{\bar{u}\in BV_{u_0}(\Omega,\R^n)}\int_\Omega f(\mathrm{D}^a\bar{u}(x))\dx+\int_\Omega f^\infty(\frac{\mathrm{d}\mathrm{D}^s \bar{u}}{\mathrm{d}|\mathrm{D}^s \bar{u}|}(x))\dd|\mathrm{D}^s \bar{u}|(x).\label{org vproblem}
\end{align}
If, furthermore, $u\in W^{1,1}(\Omega,\R^n)$, then $u$ is also a minimizer of the original problem \eqref{og vproblem}.
\end{corollary}
This follows from the theorem with the same argument as for Corollary \ref{Cor1}.\smallskip

On the other hand, we also have a smooth counterexample showing the necessity of \eqref{r1 split}.

\begin{theorem}\label{T3}
There is a positively $1$-homogeneous and even function $f:\R^{2\times 2}\rightarrow \R$ fulfilling every statement in the Assumption \ref{B2}, except \eqref{r1 split}, and which is additionally in $C^\infty(\R^{2\times 2}\backslash \{0\})$, such that there is a function $u_0\in C^\infty(\de B_1(0))$, which is not identically $0$, such that $u=0$ is a minimizer of the problem \eqref{vproblem} on $\Omega=B_1(0)\subset \R^2$ with this data.
In particular, for this data, it holds that $u\neq u_0$ on a subset of $\de B_1(0)$ with positive measure.
\end{theorem}

Of course, this data also fulfills \ref{B1} and \ref{B3}. This example illustrates how the situation here is quite different from the scalar setting, as for every scalar, even, uniformly convex and autonomous integrand, the ball would fulfill the mean curvature condition \footnote{this can, for instance, be seen from the equivalence with the barrier condition, as stated in \cite[Sec. 3]{jerrard2018existence}, which holds for all strictly convex sets (see e.g.\ \cite[Prop.\ 5.5]{gorny2024functions})}.

Roughly speaking, the mechanism for smooth non-attainment here is that there is an extension of $u_0$ such that a part of $\de \Omega$ is a minimal surface with respect to the anisotropic scalar integrand $f(u_0(y)\otimes \cdot)$.

\subsection{Further discussion}
\subsubsection{Uniqueness}
Our results also imply new uniqueness results for minimizers in a lot of cases.

\begin{corollary}
Assume that either the assumptions of Theorem \ref{T1} or of Theorem \ref{T2} hold and that $\Omega$ is connected and $U$ is not empty. Further assume that $f(x,\cdot)$ is strictly convex for every $x\in \Omega$.
Let $u,v\in W^{1,1}(\Omega,\R^n)\cap L^2(\Omega, \R^n)$ be two minimizers of \eqref{gen rel prob} (resp.\ \eqref{vproblem}) with the same boundary datum $u_0$ (and the same $g,h,\lambda$ for \eqref{gen rel prob}). Then $u=v$.
\end{corollary}
\begin{proof}
Since the singular part in \eqref{gen rel prob} (resp.\ \eqref{vproblem}) vanishes by assumption, it is immediate from the strict convexity that $\frac{u+v}{2}$ must have strictly less energy than $u$ and $v$ unless $\mathrm{D}(u-v)=0$, which, since $u$ and $v$ agree on a part of the boundary by the theorem, shows that $u=v$.
\end{proof}
There are many existing theorems about the regularity of these kinds of problems (including the vectorial case), showing that under certain ellipticity assumptions, which hold e.g.\ for the Plateau problem, any minimizer must automatically be in $W^{1,1}$, see e.g.\ \cite{bildhauer2003convex, marcellini2006nonlinear, beck2015interior, fussangel2024singular}. On the other hand, for the least gradient problem, strict convexity does not hold (and neither does interior regularity), and uniqueness might fail in this setting, see e.g.\ \cite[Example 2.7]{mazon2014functions}. A classical counterexample due to Santi \cite{santi1972sul} (see also \cite[Section 3.3]{beck2013dirichlet}) also shows that in almost nowhere mean-convex domains, uniqueness might fail even though interior regularity holds.

We also refer to \cite[Appendix B]{beck2013dirichlet} for a more detailed discussion of continuity in the boundary datum and similar issues. 

\subsubsection{The Dirichlet problem for the Rudin-Osher-Fatemi functional}\label{rofsec}

Let us comment on the special case in which $n=1$ and $\lambda=1$ and $f(x,\xi)=|\xi|$, in which the functional reduces (up to a constant) to the famous Rudin-Osher-Fatemi functional \cite{rudin1992nonlinear} given by \begin{align}
\int_\Omega |\mathrm{D} u|+\int_\Omega\frac{1}{2}(u-\bar{h})^2\dx \quad \text{(with $\bar{h}=h+g$)}\label{ROF}
\end{align}
equipped with Dirichlet boundary values $u_0$. The question when this Dirichlet problem (usually the functional is considered with a Neumann boundary condition) is solvable was raised by Brezis in \cite{brezis2019remarks} and was previously open for $d>1$.

In this case, our results give the following answer to this question:

\begin{corollary}
Suppose the following holds
\begin{itemize}
\item Assumption \ref{A1} holds with $U=\de \Omega$ 
\item $\bar{h}\in L^2(\Omega)$ is continuous in a neighborhood of $\de\Omega$ 
\item $u_0\in C^0(\de\Omega)$
\item and $\min_{\de \Omega}H_{\de \Omega}- |\bar{h}-u_0|>0$ where $H_{\de\Omega}=H_{\de\Omega,|\cdot|}$ is the mean curvature, normalized to be $d-1$ for a sphere, 
\end{itemize}
then the problem \eqref{ROF} has a minimizer $u\in BV(\Omega)$ for which the Dirichlet boundary condition is attained in the trace sense.
\end{corollary}
\begin{proof}
Take a continuous extensions $g$  of $\bar{h}|_{\de\Omega}-u_0$ to $\Omega$, then \ref{A3}-\ref{A5} hold for this $g$ and $h=\bar{h}-g$, while \ref{A1}, \ref{A2} and \ref{A6} hold by assumption already. The relaxed problem \eqref{gen rel prob} has a minimizer $u\in BV(\Omega)\cap L^2(\Omega)$ (see Proposition \ref{ex min} for details), which attains the boundary values $u_0$ by Theorem \ref{T1}.
\end{proof}




Similar statements can also be made for $u_0\in BV(\de\Omega)$ or $u_0\in W^{\alpha,p}(\de\Omega)$. Of course, the result also holds for other values of $\lambda$ if $\min H_{\de\Omega}-\lambda|u-\bar{h}|>0$.

Let us also mention that if $d<8$, then $u$ inherits the interior regularity of $\bar{h}$ in the sense that if $\bar{h}\in C^{\beta}(\Omega)$, then $u\in C^\beta(\Omega)$ for $\beta\in (0,1]$ as shown in \cite{caselles2011regularity}.

The bound of $H_{\de\Omega}$ on $|u_0-\bar{h}|$ is also sharp if $d\geq 3$, for $|u_0-\bar{h}|>H_{\de\Omega}$ a counterexample is given by \begin{align}
u_0=0;\quad \bar{h}=(1+t)\div\frac{x}{|x|};\quad u=t\div\frac{x}{|x|}\quad \text{ with $\Omega=B_1(0)$}
\end{align}
 for any $t\in \R_{>0}$, since $\div\frac{x}{|x|}=\frac{d-1}{|x|}$. To see that this is indeed a counterexample, one notes that on the one hand minimizers of both the original problem and the relaxed version \eqref{gen rel prob} must be unique due to the strictly convex $(u-\bar{h})^2$-term. On the other hand, one can check by direct calculation that this $u$ is indeed a minimizer of the relaxed problem by using the subdifferential characterisation in the Lemmata \ref{sca char} and \ref{sc equiv} with $z=-\frac{x}{|x|}$. 

We also stress that the necessity of non-negative mean curvature persists here, even if $\bar{h}=u_0$ on the boundary. A counterexample is for instance given by $\Omega=B_1(0)\backslash B_{\frac{1}{2}}(0)\subset \R^2$ with the data \begin{align}
u=0\quad \text{and}\quad \bar{h}=u_0=\frac{4}{3|x|}-\frac{4}{3},
\end{align}
 which does not attain the boundary data at $|x|=\frac{1}{2}$. To verify that this is indeed a counterexample, one can again use the uniqueness and Lemma \ref{sca char}, together with $z=\frac{2}{3}x-\frac{4}{3}\frac{x}{|x|}$.


On the other hand, optimality of this result is certainly an open-ended question, and if one takes more information on the structure of the data into account, then better results for $d=1$ are available, see e.g.\ \cite[Thm.\ 3]{brezis2019remarks} and \cite{rybka2025dirichlet}.

\subsubsection{The trace space of functions of least gradient}
Let us consider the special case of the least gradient problem, that is \eqref{gen rel prob} with $g=\lambda=h=0$ and $f(x,\xi)=|\xi|$. In \cite[Open Problem 6]{gorny2024functions} G\'orny and Maz\'on raised the question of whether for every $u_0\in BV(\de \Omega)$ there is a minimizer $u$ of \eqref{gen rel prob} for which $u=u_0$ in the trace sense (which was already shown in the special case $d=2$ by G\'orny in \cite{gorny2018planar}). 

Our result answers this question positively if $\de \Omega$ is $C^2$ and has positive mean curvature (if the mean curvature is negative, there are of course smooth counterexamples). In particular, in this context, it is also quite easy to remove the assumption that $u\in L^2$ in the theorem. \begin{corollary}\label{cor sans l2}
Suppose that the assumptions \ref{A1}, \ref{A3} and \ref{A6} hold with $g=\lambda=h=0$ and $f(x,\xi)=|\xi|$ and $U=\de \Omega$. Let $u$ be a minimizer of the relaxed problem \eqref{gen rel prob}, then \begin{align*}
u=u_0\quad \text{ on $\de \Omega$}
\end{align*}
in particular, the problem \eqref{gen org prob} is also solvable under these circumstances.
\end{corollary}
The (short) proof is given in Section \ref{S6}.



\subsubsection{Comments on the optimality of the assumptions}
\vphantom{a}

\textbf{The size bound on $g$:} We already gave a (somewhat implicit) example of the sharpness of the size bound on $g$ in \ref{A3} for $d\geq 3$ in the counterexample for the ROF functional in Subsection \ref{rofsec} above.

 If one allows other integrands, then there are also easier counterexamples in one dimension, for instance one can consider $\Omega=(0,1)$ and $f(x,\xi)=a(x)|\xi|$ for $a\in C^2$, then the generalized mean curvature at $1$ is $a'(1)$ and at $0$ it is $-a'(0)$. If both values are positive, then one can then check, using Lemma \ref{sca char} with $z=a(x)$ that $u=0$ is a minimizer for $g=-a'(x)$ and $\lambda=0$ and boundary datum $u_0(1)=1$ and $u_0(0)=-1$.

\textbf{The regularity of $u_0$:} Whether or not the spaces $BV$ and $W^{\alpha,p}$ (with $\alpha p\geq 2$) are optimal is not clear, the construction of a counterexample given by Spradlin and Tamasan in \cite{spradlin2014not} seems to be limited to spaces $W^{\alpha,1}$ with $\alpha <\frac{2}{3}$. 

It would be an interesting question whether or not the required regularity of the boundary datum improves for elliptic integrands (such as e.g.\ $\sqrt{1+|\xi|^2}$).\smallskip

\textbf{The structure of the lower order terms:}
There are plenty of works considering similar problems with different lower order terms, for instance, non-linear eigenvalue problems \cite{kawohl2007dirichlet}, saddle point problems \cite{figueiredo2018nehari} or parabolic versions \cite{andreu2004parabolic,andreu2005cauchy,gorny2022duality,Meyer}.

Some of these fall under our framework by using that their Euler-Lagrange equations have the same structure as the functional here (see Lemma \ref{sc equiv}) if one considers $g$ as the contribution of the additional nonlinearities and that fulfilling the Euler-Lagrange equation is equivalent to being a minimizer (cf.\ Lemma \ref{sc equiv}) (this e.g.\ works with the saddle points in \cite{figueiredo2018nehari} if the prefactor of the lower order nonlinearity is small enough). In full generality, results analogous to ours are simply not going to be true, e.g.\ for the Cheeger problem \cite{leonardi2015overview}, there are easy examples (such as balls) where the situation seems to be completely different.

Regarding the parabolic setting, we expect that the technique here should work under similar assumptions and should produce a bound of the type $\de_t|u-u_0|\leq -H_{\de\Omega,f}$ on $\de\Omega$ if $u\neq u_0$.\smallskip

Finally, we remark that of course not every assumption here has been optimized towards maximal generality, and we expect that some of them can be weakened (e.g.\ the regularity of $h$) without too much effort.

\subsection{Organisation of the paper}
In Section \ref{S2}, we introduce the necessary preliminaries about function spaces and convex analysis.
In Section \ref{S3}, we recall the general machinery for linear growth functionals in the calculus of variations and state the characterisations of their subdifferentials, which the proofs of the Theorems \ref{T1} and \ref{T2} in the subsequent Section \ref{S4} use. Most of the proofs of the technical lemmata used there are postponed to Section \ref{S5}. The short proof of Corollary \ref{cor sans l2} is given in Section \ref{S6}. The counterexample in Theorem \ref{T3} is constructed in Section \ref{S7}. Finally, the proof of a lemma from convex analysis, which we believe to be well known but could not find in a satisfying citable form, is given in the Appendix \ref{appendix}.

\section{Preliminaries and notation}\label{S2}
\subsection{Notational conventions}
We shall write $A\gtrsim B$ if there is some constant $C$, depending on $\Omega,n,d, f$ or $\kappa$, but not $x$ or $\eps$, such that $A\geq CB$, sometimes, we shall also explicitly write the constant, in this case it is allowed to change its value from line to line. If we specifically wish to highlight the dependence of the implicit constant on a certain parameter, we shall denote that by a subscript, as in e.g.\ $\gtrsim_{u_0}$. Sometimes we also write $C(t)$ to highlight the dependence on $t$ (or some other parameter). If such an inequality holds in both directions, we shall use the symbol $\approx$ for that.

The divergences of matrices are always taken row-wise. $S^{d-1}$ is the $d-1$-dimensional unit sphere. $\mathds{1}$ denotes indicator functions. $\mathcal{L}^d$ and $\mathcal{H}^d$ are the $d$-dimensional Lebesgue/Hausdorff measure. The spaces of smooth and compactly supported smooth test functions and the space of distributions are denoted by $\mathcal{D}$ and $\mathcal{D}'$. We denote the truncation by \begin{align}
T_b(a)=\min(b,\max(a,-b)) \quad \text{for $a\in \R$ and $b\in \R_{\geq 0}$}.\label{def trunc}
\end{align}

\subsection{The space \texorpdfstring{$BV$}{BV}}\label{S22}

Let $O\subset \R^d$ be a bounded set with Lipschitz boundary. We denote the outer normal at $x\in \de O$ with $\nu_x$ (sometimes also written as just $\nu$) and use the same notation for the normal on $\de \Omega$.

The space $BV(O,\R^n)$ of functions of bounded variation is defined as the the space of functions $w\in L^1(O,\R^n)$ for which the distributional derivative $\mathrm{D}w\in \mathcal{D}'(O,\R^n)$ is represented by integration against a bounded, signed, $\R^{n\times d}$-valued Radon measure, which is also denoted by $\mathrm{D} w$, i.e.\ the space of functions for which there is a measure $\mathrm{D}w$ with \begin{align*}
-\int_O w\div\phi \dx=\int_O \phi \mathrm{d}\mathrm{D}w \quad \forall\phi\in \mathcal{D}(O).
\end{align*}
The space $BV(O,\R^n)$ is a Banach space when endowed with the norm
\begin{align*}
\norm{w}_{BV(O,\R^n)} := \norm{w}_{L^1(O,\R^n)} + |\mathrm{D}w|(O),
\end{align*}
where $|\mathrm{D}w|$ denotes the total variation of $\mathrm{D}w$.

We say $w_m\rightarrow w$ strictly if $w_m\rightarrow w$ in $L^1(O,\R^n)$ and $|\mathrm{D}w_m|(O)\rightarrow |\mathrm{D}w|(O)$. One can show that $C^\infty(\overline{O})$ is dense in $BV(O,\R^n)$ with respect to strict convergence.

We shall write $\mathrm{D}^a w$ and $\mathrm{D}^sw$ for the absolutely continuous and singular parts of the measure, i.e.\ \begin{align*}
\mathrm{D}^a w(x)=\frac{\mathrm{d}\mathrm{D}^a w}{\mathrm{d}\mathcal{L}^d}(x) \quad \text{and}\quad \mathrm{D}^s w=\mathrm{D}w-\mathrm{D}^a w\mathcal{L}^d.
\end{align*}
It can be shown that there is a linear, bounded trace $w\rightarrow w|_{\de O}\in L^1(O,\R^n)$ on $BV(O,\R^n)$, which is continuous under strict convergence and agrees with the classical trace on $C^1$. We will omit the ``$|_{\de O}$'' for the rest of the paper.

Furthermore, one has the same Poincar\'e and Sobolev inequalities as for $W^{1,1}$, i.e.\ for every $q\in [1,\frac{d}{d-1}]$, there is a $C(O,q)>0$, such that it holds that \begin{align} 
\norm{w}_{L^q(O)}\leq C(O,q)\left(|\mathrm{D}w|(O)+\int_{\de O}|w|\dH\right)\quad \forall w\in BV(O),\label{poincare}
\end{align}
furthermore $BV(O)$ embeds into $L^q(O)$ and the embedding is compact for $q\in [1,\frac{d}{d-1})$.

Whenever $\mathfrak{g}:\R^n\rightarrow \R^m$ is Lipschitz and $w\in BV(O,\R^n)$, then $\mathfrak{g}\circ w\in BV(O,\R^m)$. A particularly important case for us is that if $n=m=1$, then it holds that $\max(a,w)\in BV(O)$ for any $w\in BV(O)$ and $a\in \R$, in particular it also holds that \begin{align}
|\mathrm{D}\max(a,w)|(\mathcal{U})\leq|\mathrm{D}w|(\mathcal{U})\label{measure trunc}
\end{align}
for every measurable $\mathcal{U}\subset O$ and of course this also holds for minima instead of maxima and any combination of the two (such as $T_b$).\smallskip

If $M\subset \R^{d}$ is a compact embedded $C^1$-manifold of dimension $d'<d$ (the only relevant case for us are subsets of $\de \Omega$), then we define $BV(M,\R^n)$ by first defining the norm \begin{align*}
\norm{w}_{BV(M,\R^n)}:=\norm{w}_{L^1(M,\R^n)}+\norm{\nabla_\tau w}_{L^1(M,\R^{n\times d'})},
\end{align*}
first for $w\in C^1(M,\R^n)$ where $\nabla_\tau$ is the tangential gradient, taken with respect to any orthonormal basis of the tangent space. We then define the norm for $w\notin C^1(M,\R^n)$ by \begin{align*}
\norm{w}_{BV(M,\R^n)}=\inf_{C^1\ni w_m\rightarrow w\in L^1}\norm{w_m}_{BV(M,\R^n)}.
\end{align*}
It can be shown that $\nabla_\tau w_m$ converges weakly\textsuperscript{$*$} to a measure which is independent of the minimizing sequence \cite[Sec.\ 2]{kreuml2019fractional}. This definition is also equivalent to using charts or distributional derivatives.

The reader may e.g.\ consult the textbook \cite{ambrosio2000functions} for further reading.

\subsubsection{The Anzellotti Pairing}\label{S23}
In order to develop a satisfying theory for the Euler-Lagrange equation, it is necessary to have some notion of product between $\mathrm{D}w$ and discontinuous functions, which is done via partial integration. We refer to Anzelotti's original paper \cite{anzellotti1983pairings} as a source\footnote{The reference only contains the case $n=1$, the extension to $n\geq1$ is however trivial by applying the theory for $n=1$ to each row of $z$}.
We define \begin{align*}
X_p(O,\R^{n\times d}):=\big\{z\in L^\infty(O,\R^{n\times d})\,\big|\, \div z\in L^p(O,\R^n)\big\}.
\end{align*}
Here, the divergence is the distributional one and taken row-wise. 
One can then define a distribution $(z,\mathrm{D}w)$ for $z\in X_p(O,\R^{n\times d})$ and $w\in (BV\cap L^q)(O,\R^n)$ with $\frac{1}{p}+\frac{1}{q}=1$ by \begin{align*}
(z,\mathrm{D}w)(\phi)= - \int_O \phi w\cdot\div(z)\dx - \int_O w \otimes\nabla \phi :z\dx \quad \text{for all } \phi \in \mathcal{D}(O).
\end{align*}
It can be shown that this defines a bounded Radon measure with $|(z,\mathrm{D}w)|<<|\mathrm{D}w|$ and that it holds that \begin{align}|(z,\mathrm{D}w)|(\mathcal{U})\leq \norm{z}_{L^\infty}|\mathrm{D}w|(\mathcal{U})\label{anz bas est}
\end{align}
 for every measureable $\mathcal{U}\subset O$.
It further generalizes the pointwise product in the sense that \begin{align}
\frac{\mathrm{d}(z,\mathrm{D}w)}{\mathrm{d}\mathcal{L}^d}=\scalar{z}{\frac{\mathrm{d}\mathrm{D}^aw}{\mathrm{d}\mathcal{L}^d}}\quad \text{$\mathcal{L}^d$-a.e.},
\end{align}
in particular, for $w\in W^{1,1}(O,\R^n)$ it is the pointwise product.
%
%
Regarding the singular part of the measure, one can show that if $z(x)\in \mathcal{C}$ for some closed convex set $\mathcal{C}$ in some open $\mathcal{O}\subset O$, then it holds that \begin{align}
\frac{\mathrm{d}(z,\mathrm{D}w)}{\mathrm{d}|\mathrm{D}^sw|}(x)\in \mathcal{C}\cdot \frac{\mathrm{d}\mathrm{D}^sw}{\mathrm{d}|\mathrm{D}^sw|}(x)\quad \text{for $|\mathrm{D}^sw|$-a.e.\ $x\in \mathcal{O}$}.\label{anz rep}
\end{align}
In fact, more precise expressions for the density as suitable weak limits exist \cite{anzellotti1983traces,crasta2019anzellotti}.

Furthermore $z\in X_p(O,\R^{n\times d})$ has a normal trace $[z,\nu]\in L^\infty(\de O,\R^n)$ on $\de O$, which is linear in $z$ and fulfills $\norm{[z,\nu]}_{L^\infty(\de O,\R^n)}\leq \norm{z}_{L^\infty(O,\R^{n\times d})}$. For $z\in C^1$ it agrees a.e.\ with the pointwise normal trace, $z\nu_x=\sum_{i=1}^d z_{mi}(x)\nu_x^i$, where $\nu_x^i$ are the components of $\nu_x$.
Similarly to \eqref{anz rep}, we have that whenever $z(x)\in \mathcal{C}$ for some closed convex set $\mathcal{C}$ in some open $\mathcal{O}\subset O$, then it holds that \begin{align}
[z,\nu](x)\in \mathcal{C} \nu_x\quad \text{for $\mathcal{H}^{d-1}$-a.e.\ $x\in \de(\mathcal{O}\cap  O)$}.\label{nt rep}
\end{align}
Furthermore, one has the following Gauss-Green type formula for $w\in (BV\cap L^q)(O,\R^n)$ (for $\frac{1}{p}+\frac{1}{q}=1$): \begin{align}
\int_O \scalar{w}{\div z}\dy+\int_O (z,\mathrm{D}w)=\int_{\de O}\scalar{w}{[z,\nu]}\dH.\label{gg form}
\end{align}
We will also use the following chain rule-type estimate for the pairing.\begin{lemma}\label{anz ch}
Let $w\in (BV\cap L^2)(O)$ and $z\in X_2(O,\R^d)$. Let $\mathfrak{g}:\R\rightarrow \R$ be a non-decreasing Lipschitz function, then it holds that \begin{align}
|(z,\mathrm{D}(\mathfrak{g}\circ w))_-|(O')\leq \norm{\mathfrak{g}}_{Lip}|(z,\mathrm{D}w)_-|(O')
\end{align}
for every measurable $O'\subset O$ where the ``$-$'' denotes the negative part of the measure.
\end{lemma}
\begin{proof}
In \cite[Prop. 4.5 (iii)]{crasta2019anzellotti} it is shown that composition with a non-decreasing Lipschitz function does not change the density $\frac{\mathrm{d}(z,\mathrm{D}w)}{\mathrm{d}|\mathrm{D}w|}$, which, together with the facts that $|\mathrm{D}(\mathfrak{g}\circ w)|\leq \norm{\mathfrak{g}}_{Lip}|\mathrm{D}w|$ by the chain rule for $BV$-functions \cite[Thm.\ 3.99]{ambrosio2000functions} and that $|(z,\mathrm{D}w)|<<|\mathrm{D}w|$, shows the lemma.
\end{proof}

\subsection{Elements of convex analysis}\label{Sec conv ana}
Let $Y$ be a (real) separable Hilbert space, and let $\mathcal{G}:Y\rightarrow \R\cup\{+\infty\}$ be convex and lower semicontinuous (with the usual conventions for $+\infty$), then we say that $v\in Y$ is in the subdifferential of $\mathcal{G}$ at $w\in Y$, written $v\in \de\mathcal{G}(w)$, if $\mathcal{G}(w)<\infty$ and \begin{align*}
\scalar{w'-w}{v}\leq \mathcal{G}(w')-\mathcal{G}(w) \quad \forall w'\in Y.
\end{align*}
If $\mathcal{G}$ is Gateaux-differentiable at $w$, then $\de\mathcal{G}(w)=\{\mathrm{D}\mathcal{G}(w)\}$.

One can show that if $Y$ is finite-dimensional and $\mathcal{G}$ is finite everywhere, then $\de\mathcal{G}(w)$ is nonempty for all $w$.
We also note that some $w\in Y$ is a minimizer of $\mathcal{G}$ if and only if $0\in \de\mathcal{G}(w)$.

We set \begin{align*}
\ran(\de\mathcal{G})=\{v\,\big|\, v\in \de\mathcal{G}(w)\text{ for some $w\in Y$}\}.
\end{align*}
For functions also depending on other parameters, we denote the dependence on the parameter by a subscript, e.g.\ for the subdifferential of $f(x,\xi)$ with respect to $\xi$, we write $\de_\xi f(x,\xi)$.

We now consider integrands $\bar{f}:\overline{\Omega}\times \R^{n\times d}$ which are convex in the second variable, continuous in the joint variable, and fulfill a  growth bound of the form \begin{align}
C_3^{-1}|\xi|-C_3\leq \bar{f}(x,\xi)\leq C_3(|\xi|-1),\label{gb barf}
\end{align}
with some $C_3>0$ not depending on $x$ or $\xi$.

We define the recession function $\bar{f}^\infty$ as \begin{align*} 
\bar{f}^\infty(x,\xi)=\lim_{x'\rightarrow x,\, t\rightarrow +\infty}\frac{1}{t}\bar{f}(x,t\xi),
\end{align*}
and it is easy to show that, if this limit exists, then \eqref{gb barf} implies \begin{align}
C_3^{-1}|\xi|\leq\bar{f}^\infty(x,\xi)\leq  C_3|\xi|\label{gb rec}
\end{align}
with the same constant $C_3$.

In particular, if \ref{A2} resp.\ \ref{B2} hold, then it holds that \begin{align}
C_1^{-1}|\xi|\leq f^\infty(x,\xi)\leq  C_1|\xi|,\label{gc finf}
\end{align}
with the same $C_1$ as in these assumptions.

The recession function is positively $1$-homogeneous in the second variable, and if $\bar{f}$ is positively $1$-homogeneous in the second variable, then $\bar{f}^\infty=\bar{f}$.

Regarding the subdifferentials of $\bar{f}$ and $\bar{f}^\infty$, we have the following relation \begin{align}
\overline{\ran{\de_\xi\bar{f}(x,\cdot)}}=\ran{\de_\xi\bar{f}^\infty(x,\cdot)}\subset \overline{B_{C_3}(0)}.\label{ball subdiff}
\end{align}
This also implies Lipschitz continuity of $\bar{f}$, more precisely \begin{align}
|\bar{f}(x,\xi_1)-\bar{f}(x,\xi_2)|\leq C_3|\xi_1-\xi_2|\label{lipsch}
\end{align}
with the same $C_3$.

The Legendre transform of $\bar{f}$ is defined as \begin{align*}
\bar{f}^*(x,\xi^*)=\sup_{\xi\in \R^{n\times d}}\scalar{\xi^*}{\xi}-\bar{f}(x,\xi).
\end{align*}
It can be shown that $\bar{f}^*(x,\cdot)=+\infty$ outside of $\overline{\ran{\de_\xi\bar{f}(x,\cdot)}}$ and that \begin{align*}
\bar{f}(x,\xi)=\sup_{\xi^*\in \R^{n\times d}}\scalar{\xi^*}{\xi}-\bar{f}^*(x,\xi^*),
\end{align*}
where the supremum is attained at $\xi^*$ if and only if $\xi^*\in \de_\xi f(x,\xi)$.

If $\bar{f}^\infty$ is differentiable in $\xi$ at $(x,\xi)$, we furthermore have  \begin{align}
\bar{f}^\infty(x,\xi)=\sup_{\xi^*\in\ran(\de_\xi\bar{f}^\infty(x,\cdot))} \scalar{\xi}{\xi^*}=\scalar{\xi}{\mathrm{D}_\xi \bar{f}^\infty(x,\xi)},\label{doal}
\end{align}
and it holds that $\de_\xi\bar{f}^\infty(x,0)=\ran(\de_\xi\bar{f}^\infty(x,\cdot))$.
%
%
%
%
We also note that our assumptions imply a quantitative version of this for $f^\infty$:

\begin{lemma}\label{fen lemma}
Assume \ref{A2} holds. There is a constant $C>0$, not depending on $x$, such that for all unit vectors $v$, all $v^*\in \overline{\ran(\de_\xi f(x,\cdot))}$ and all $x\in \overline{\Omega}$ it holds that \begin{align}
\scalar{\mathrm{D}_\xi f^\infty(x,v)-v^*}{v}\geq C|\mathrm{D}_\xi f^\infty(x,v)-v^*|^2.\label{quant fen}
\end{align}
The same holds if \ref{B2} holds.
\end{lemma}
We believe this Lemma to be quite well-known, but were not able to find this exact version in the literature; therefore, a proof can be found in the Appendix \ref{appendix}.

We also note for further reference that, if \ref{A2} resp.\ \ref{B2} hold, then \begin{align}
\lim_{t\rightarrow +\infty} \mathrm{D}_\xi f(x,t\xi)\cdot\frac{\xi}{|\xi|}=\mathrm{D}_\xi f^\infty(x,\xi)\cdot\frac{\xi}{|\xi|}=f^\infty(x,\frac{\xi}{|\xi|})\qquad \forall \xi\neq 0.\label{conv grad}
\end{align}
Indeed this follows from the fact that the expression on the left-hand side is bounded from below by $\frac{1}{2t|\xi|}(f(x,t\xi)-f(x,\frac{1}{2}t\xi))$ and from above by $\frac{1}{t|\xi|}(f(x,2t\xi)-f(x,t\xi))$ by convexity, which both converge to $f^\infty(x,\frac{\xi}{|\xi|})$ by definition.

The reader may refer e.g.\ to the books \cite{borwein2006convex,bauschke,boyd2004convex} for further background reading.

\subsection{Fractional Sobolev spaces}\label{S24}

Let $\mathcal{U}\subset \de\Omega$ be an open subset of the boundary which is $C^1$ and whose (relative) boundary is $C^1$, let $\alpha\in (0,1)$ and $p\in (1,\infty)$, then we define \begin{align*}
\norm{w}_{W^{\alpha,p}(\mathcal{U})}:=\norm{w}_{L^p(\mathcal{U})}+\left(\int_{\mathcal{U}^2} \frac{|w(x)-w(y)|^p}{|x-y|^{d-1+\alpha p}}\dx\dy\right)^\frac{1}{p}
\end{align*}
and define the space $W^{\alpha,p}(\mathcal{U})$ as the subset of $L^p(\mathcal{U})$ for which this norm is finite. This definition is equivalent to defining the space via a diffeomorphism to a subset of the $\R^{d-1}$ and using the same norm there. We refer e.g.\ to the monograph \cite{di2012hitchhikers} for background reading.

The main property which we will require is the following one: If $w\in W^{\alpha,p}(\mathcal{U})$, then there are functions $r_\delta\in L^p(\mathcal{U})$, such that for every a.e.\ $x,y\in \mathcal{U}$ with $|x-y|\leq \delta$ it holds that \begin{align}
|w(x)-w(y)|\leq |x-y|^\alpha(r_\delta(x)+r_\delta(y))\label{cald grad}
\end{align}
and that furthermore \begin{align}\lim_{\delta\searrow 0} \norm{r_\delta}_{L^p(\mathcal{U})}=0.\label{r lim}
\end{align}
A proof of this can be found e.g.\ in \cite{devore1984maximal}, where it is shown for open subsets of the $\R^d$, for subsets of $\de \Omega$ it follows by parametrizing the boundary.

\begin{remark}
This property is in fact a little weaker than being in $W^{\alpha,p}$, for instance \eqref{cald grad} is an equivalent characterisation of the Triebel-Lizorkin space $F_{p,\infty}^\alpha$ (which is bigger than $W^{\alpha,p}$), see e.g.\ \cite[Prop.\ 3.2]{bahouri1994equations} and since \eqref{r lim} holds for smooth functions, both properties hold for the closure of $C^\infty$ in the Triebel-Lizorkin space $F_{p,\infty}^\alpha$. We prefer not to use local Triebel-Lizorkin spaces here for the sake of simplicity, even though Theorem \ref{T1} still holds for $u_0\in L^p(U)\backslash W^{\alpha,p}(U)$ which fulfill \eqref{cald grad} and \eqref{r lim}.
\end{remark}

Let us also stress that we never use any kind of fractional derivative here, and $\mathrm{D}^a$ and $\mathrm{D}^s$ never mean anything other than the absolutely continuous and singular part of the gradient.

\section{Existence of minimizers and characterisation of the subdifferential}\label{S3}
In this section, we will recall some of the classical existence theory for variational problems with linear growth and characterisations of their subdifferential. Most of this is standard and well known to the expert, we nevertheless provide some of the proofs to show how to deal with the lower-order terms.

For the most part, we will treat the scalar and vector-valued setting simultaneously here, and statements where neither of the two is specified apply to both. In particular, we consider $f$ as a function on $\overline{\Omega}\times \R^{n \times d}$ here of which we assume that either $n=1$ or that it is independent of the first variable.\smallskip

The expression $\int f(x,\mathrm{D}w)\dx$ is not defined in a classical sense if $w\notin W^{1,1}$.
Instead one defines \begin{equation}\begin{aligned}
\mathcal{F}_{u_0}(w):=&\int_{\Omega} f(x,\mathrm{D}^aw(x))\dx+\int_{\Omega}f^\infty(x,\frac{\mathrm{d}\mathrm{D}^sw}{|\mathrm{d}\mathrm{D}^sw|}(x))\dd|\mathrm{D}^sw|(x)\label{def fu0}\\
&+\int_{\de\Omega} f^\infty(x,(u_0-w)(x)\otimes \nu_x)\dH(x)
\end{aligned}\end{equation}
for $u_0\in L^1(\de\Omega,\R^n)$ and $w\in BV(\Omega,\R^n)$. Clearly, this functional is $\int_\Omega f(x,\mathrm{D}w)\dx$ when restricted to $W_{u_0}^{1,1}(\Omega,\R^n)$.

This functional is coercive on $BV(\Omega,\R^n)$ in the sense that \begin{align}
\mathcal{F}_{u_0}(w)\geq C_1^{-1}\left(|\mathrm{D}w|(\Omega)+\int_{\de\Omega}|w|\dH\right)-C_1(\mathcal{L}^{d}(\Omega)+\norm{u_0}_{L^1(\de \Omega)})\label{f coer}
\end{align}
with the constant $C_1$ from the assumptions \eqref{bd gf1} resp.\ \eqref{bd gf2}, as one directly sees from using \eqref{gb rec} and the triangle inequality.

\begin{proposition}\label{rel func}
Suppose \ref{A1} and \ref{A2} (resp. \ref{B1} and \ref{B2}) hold and $u_0\in L^1(\de\Omega,\R^n)$ is fixed. Then we have the following:

\noindent\textbf{a)} The functional $\mathcal{F}_{u_0}$ is convex and lower-semicontinuous with respect to convergence in $\mathcal{D}'(\Omega,\R^n)$.

\noindent\textbf{b)} For every $w\in BV(\Omega,\R^n)$ there is a sequence $w_m\in (L^2\cap W_{u_0}^{1,1})(\Omega,\R^n)$, converging to $w$ in $L^1(\Omega,\R^n)$ and so that $\mathcal{F}_{u_0}(w_m)\rightarrow \mathcal{F}_{u_0}(w)$. 

If $w\in L^q(\Omega,\R^n)$ for $q\in[1,\infty)$, then the sequence can additionally be chosen so that $w_m\rightarrow w$ in $L^q(\Omega,\R^n)$.

If $w\in L^\infty(\Omega,\R^n)$ and $u_0\in L^\infty(\de\Omega,\R^n)$, then the sequence can also be chosen such that  $w_m\xrightarrow{*}w$ in $L^\infty(\Omega,\R^n)$.
\end{proposition}
This is a classical result from \cite{Giaquinta1979}, except for the point about $L^q$/$L^\infty$-convergence in b), which one can achieve with the standard construction, but which regrettably does not seem to be stated anywhere in the literature. We therefore give a short sketch here of why the construction in \cite[Appendix B.1]{bildhauer2003convex} still works with $L^q$-convergence.
\begin{proof}[Proof sketch for the $L^q/L^\infty$-convergence in b)]
We first extend $u_0$ as $u_0^{ext}$ to some bounded  Lipschitz domain $\hat{\Omega}$ with $\overline{\Omega}\subset \hat{\Omega}$. Given $q\in [1,\infty)$, it is possible by \cite[Thm.\ 1.4]{muller2016density} to choose this extension in $(W^{1,1}\cap L^q)(\hat{\Omega},\R^n)$. If $u_0\in L^\infty(\de\Omega,\R^n)$, then it is possible to choose the extension in $(W^{1,1}\cap L^\infty)(\hat{\Omega},\R^n)$ by e.g.\ Stampaccias Lemma and truncation.

It is shown in \cite{Giaquinta1979} that if some sequence $w_m\in BV(\Omega,\R^n)$ has the property that $(w_m-u_0^{ext})\mathds{1}_\Omega+u_0^{ext}$ converges area-strictly to $(w-u_0^{ext})\mathds{1}_\Omega+u_0^{ext}$ in $\hat{\Omega}$, meaning that \begin{align*}
\mel\int_{\hat{\Omega}}\sqrt{1+|\mathrm{D}^a((w_m-u_0^{ext})\mathds{1}_\Omega+u_0^{ext})|^2}\dx+\int_{\hat{\Omega}}\dd|\mathrm{D}^s((w_m-u_0^{ext})\mathds{1}_\Omega+u_0^{ext})|\\
&\xrightarrow{m\rightarrow \infty} \int_{\hat{\Omega}}\sqrt{1+|\mathrm{D}^a((w-u_0^{ext})\mathds{1}_\Omega+u_0^{ext})|^2}\dx+\int_{\hat{\Omega}}\dd|\mathrm{D}^s((w-u_0^{ext})\mathds{1}_\Omega+u_0^{ext})|,
\end{align*}
then it holds that $\mathcal{F}_{u_0}(w_m)\rightarrow \mathcal{F}_{u_0}(w)$. Such an approximating sequence $w_m\in W_{u_0}^{1,1}(\Omega,\R^n)$ is e.g.\ constructed in \cite[Appendix B]{bildhauer2003convex}. One can check directly from the construction there (which we omit here) that, as long as $(w-u_0^{ext})\mathds{1}_\Omega+u_0^{ext}\in L^q(\hat{\Omega},\R^n)$, the construction there also yields $L^q$-convergence, if $q<\infty$ and weak\textsuperscript{$\star$}-$L^\infty$ convergence if $q=\infty$.
\end{proof}

This implies the following two Propositions:

\begin{proposition}\label{same inf}
Suppose that \ref{A1}, \ref{A2} and \ref{A4} hold and that $g\in L^{d}(\Omega)$, as well as $u_0\in L^1(\de\Omega)$. Then the problems \eqref{gen org prob0}, \eqref{gen org prob}, and \eqref{gen rel prob} have the same infimum and, if one restricts these problems to functions in $L^2(\Omega)$, the infimum does not change.
Furthermore, any limit of an $L^1(\Omega)$-convergent minimizing sequence for \eqref{gen org prob0} or \eqref{gen org prob} is a minimizer for \eqref{gen rel prob}. 

Similarly, if \ref{B1} and \ref{B2} hold and $u_0\in L^1(\de\Omega,\R^d)$, then \eqref{og vproblem}, \eqref{org vproblem}, and \eqref{vproblem} have the same infimum, and if one restricts these problems to functions in $L^2(\Omega,\R^n)$, the infimum does not change.
Any limit of an $L^1(\Omega,\R^n)$-convergent minimizing sequence for \eqref{og vproblem} or \eqref{org vproblem} is a minimizer for \eqref{vproblem}.
\end{proposition}
\begin{proof}
We only show the statement for the scalar problem \eqref{gen org prob0}/\eqref{gen rel prob}, the vectorial case is easier due to the missing lower-order terms.

We trivially have \begin{align*}
\inf \text{of \eqref{gen org prob0} }\geq \inf\text{of \eqref{gen org prob} }\geq \inf \text{of \eqref{gen rel prob}}.
\end{align*}
Therefore, it suffices to show that $\inf \text{\eqref{gen org prob0} }\leq \inf \text{\eqref{gen rel prob}}$ to show that all three are the same.

Let $w\in BV(\Omega)$ be given. We first note that \begin{align}
\mathcal{F}_{u_0}(T_b(w))\rightarrow \mathcal{F}_{u_0}(w),\label{tr claim}
\end{align}
where the truncation $T_b$ was defined in \eqref{def trunc}. This can be seen from the fact that by \eqref{lipsch} we have \begin{align*}
|\mathcal{F}_{u_0}(T_b(w))-\mathcal{F}_{u_0}(w)|\lesssim \norm{\mathrm{D}(w-T_b(w))}_{L^1(\Omega,\R^d)}+\norm{(w-T_b(w))}_{L^1(\de\Omega)}\end{align*}
 for $w\in W^{1,1}(\Omega)$, and by smooth approximation, this also holds for $w\in BV$. Since $|\mathrm{D}(w-T_b(w))|(\Omega)+\norm{(w-T_b(w))}_{L^1(\de\Omega)}\rightarrow 0$ by e.g.\ the co-area formula and dominated convergence, this shows \eqref{tr claim}.

  By dominated convergence and the fact that $gu$ is integrable since $g\in L^d(\Omega)$ by assumption and that $w\in L^{\frac{d}{d-1}}(\Omega)$ by the Sobolev embedding, we hence have \begin{align*}
\mathcal{F}_{u_0}(T_b(w))+\int_\Omega gT_b(w)\dx+\int_\Omega \lambda |T_b(w)-h|^2\dx\xrightarrow{b\rightarrow \infty} \mathcal{F}_{u_0}(w)+\int_\Omega gw\dx+\int_\Omega \lambda |w-h|^2\dx.
\end{align*}
By the Proposition \ref{rel func} b), there are $w_{b,m}\in W_{u_0}^{1,1}(\Omega)$ converging to $T_b(w)$ in $L^2(\Omega,\R)$ and such that $\mathcal{F}_{u_0}(w_{b,m})\rightarrow \mathcal{F}_{u_0}(T_b(w))$ as $m\rightarrow \infty$. If $d=1$, then we trivially have $u_0\in L^\infty(\de \Omega)$ and can also ensure that this sequence converges weakly\textsuperscript{$*$} in $L^\infty(\Omega)$.

 Since this makes the lower order terms converge, we see that for a diagonal sequence $w_m$ among the $w_{b,m}$ we have \begin{align*}
\mathcal{F}_{u_0}(w_m)+\int_\Omega gw_m\dx+\int_\Omega \lambda |w_m-h|^2\dx\xrightarrow{m\rightarrow \infty} \mathcal{F}_{u_0}(w)+\int_\Omega gw\dx+\int_\Omega \lambda |w-h|^2\dx.
\end{align*}
This shows that the infimum of \eqref{gen org prob0}, \eqref{gen org prob}, and \eqref{gen rel prob} must be the same.

The fact that every $L^1(\Omega)$-convergent minimizing sequence $w_m\rightarrow w$ for \eqref{gen org prob0} (resp.\ \eqref{gen org prob}) converges to a minimizer of \eqref{gen rel prob} follows from the lower semicontinity of $\mathcal{F}_{u_0}$, the fact that $\int_\Omega (w_m-w)g\dx\rightarrow 0$ by the embedding $BV\hookrightarrow L^\frac{d}{d-1}(\Omega)$ and Banach-Alaoglu and because $\int_\Omega\lambda|\cdot-h|^2\dx$ is lower-semicontinuous by convexity.
\end{proof}

\begin{proposition}\label{ex min}
\textbf{a)} Suppose that \ref{A1}, \ref{A2}, \ref{A4} and $u_0\in L^1(\de\Omega)$ hold. Further suppose that $\norm{g}_{L^\frac{q}{q-1}(\Omega)}< C_1^{-1}C(\Omega,q)$ for the constant $C(\Omega,q)$ from the Poincar\'e inequality \eqref{poincare} for some $q\in [1,\frac{d}{d-1}]$ and the constant $C_1$ from \eqref{bd gf1}. Then there is a minimizer $u\in BV(\Omega)$ for the relaxed problem \eqref{gen rel prob}.

\noindent \textbf{b)} Suppose that \ref{A1}, \ref{A2}, \ref{A4} and $u_0\in L^1(\de\Omega)$ hold. Suppose that $\inf_{x\in \Omega} \lambda(x)>0$ and that $g\in L^2(\Omega)$, then there is a unique minimizer for the relaxed problem \eqref{gen rel prob}.

\noindent \textbf{c)} Suppose that \ref{B1}, \ref{B2} and $u_0\in L^1(\de\Omega,\R^n)$ hold. Then the relaxed problem \eqref{vproblem} has a minimizer $u\in BV(\Omega,\R^n)$. 
\end{proposition}
\begin{proof}
This is very standard, and we only outline \textbf{a)} and \textbf{b)}, \textbf{c)} is even easier.

\textbf{a)} We use the direct method in the calculus of variations. It follows from the bound \eqref{f coer} on $\mathcal{F}_{u_0}$, the Poincar\'e inequality \eqref{poincare}, and the assumed bound on $g$ that one has coercivity. This yields the existence of a minimizing sequence which is bounded in $BV(\Omega)$ and which also converges weakly\textsuperscript{$\ast$} in $L^{\frac{d}{d-1}}$ by the Sobolev embedding and Banach-Alaoglu. The functional in \eqref{gen rel prob} is lower semicontinuous with respect to this convergence by the point a) in Proposition \ref{rel func} and because the lower order terms are trivially lower semicontinuous with respect to weak\textsuperscript{$\ast$}-convergence in $L^{\frac{d}{d-1}}$. Hence, every minimizing sequence has a convergent subsequence going to a minimizer.
 
\textbf{b)} This also follows from the direct method, where one obtains from the $\lambda|u-h|^2$-term that each minimizing sequence is bounded in $L^2(\Omega)$ and hence that every minimizing sequence converges strongly in $L^2(\Omega)$ due to the uniform convexity of that term, which also yields the uniqueness.
\end{proof}

\allowdisplaybreaks[4]

\subsection{Characterisation of the subdifferential}
The functional $\mathcal{F}_{u_0}$ is in general not differentiable (e.g.\ when is $f$ not differentiable), instead the natural object to characterise minimizers is the subdifferential as defined in Subsection \ref{Sec conv ana}.

To define $\mathcal{F}_{u_0}$ on a Hilbert space, we will consider $\mathcal{F}_{u_0}$ as a functional on $L^2(\Omega,\R^n)$ by extending it as $+\infty$ to $L^2(\Omega,\R^n)\backslash BV(\Omega,\R^n)$, which is easily seen to be convex and lower semicontinuous, as a consequence of the coercivity bound \eqref{f coer}. By an abuse of notation, we still denote this by $\mathcal{F}_{u_0}$. Furthermore, the infimum of the functional remains the same thanks to part b) of Proposition \ref{rel func} above.

\begin{lemma}\label{sca char}
Suppose that \ref{A1} and \ref{A2} are fulfilled and $u_0\in L^1(\de\Omega)$. Then the following two statements are equivalent for $v,w\in L^2(\Omega)$:

\noindent\textbf{a)} It holds that $v\in \de \mathcal{F}_{u_0}(w)$.

\noindent\textbf{b)} It holds that $w\in BV(\Omega)$ and there is a $z\in X_2(\Omega,\R^d)$ such that \begin{align}
&v=-\div z&\label{cs1}\\
&z(x)\in \de_\xi f(x,\mathrm{D}^a w(x))&\mathcal{L}^d-\text{a.e.\ in $\Omega$}\label{cs2}\\
&\frac{\mathrm{d}(z,\mathrm{D}w)}{\mathrm{d}|\mathrm{D}^sw|}(x)=f^\infty(x,\frac{\mathrm{d}\mathrm{D}^sw}{\mathrm{d}|\mathrm{D}^sw|}(x))& |\mathrm{D}^sw|-\text{a.e.\ in $\Omega$}\label{cs3}\\
&z(x)\in \overline{\ran( \de_\xi f(x,\cdot))}&\mathcal{L}^d-\text{a.e.\ in $\Omega$}\label{cs4}\\
&[z,\nu_x](u_0-w)=f^\infty(x,(u_0-w)(x)\otimes \nu_x)&\mathcal{H}^{d-1}-\text{a.e.\ on $\de\Omega$}.\label{cs5}
\end{align}
\end{lemma}

\begin{lemma}\label{vec char}
Suppose that \ref{B1} and \ref{B2} hold and that $u_0\in L^1(\de\Omega,\R^n)$. Then the following two statements are equivalent for $v,w\in L^2(\Omega,\R^n)$:

\noindent\textbf{a)} It holds that $v\in \de \mathcal{F}_{u_0}(w)$.

\noindent \textbf{b)}  It holds that $w\in BV(\Omega,\R^n)$ and there is a $z\in X_2(\Omega,\R^{n\times d})$ such that \begin{align}
&v=-\div z&\label{cv1}\\
&z(x)\in \de_\xi f(\mathrm{D}^a w(x))&\mathcal{L}^d-\text{a.e.\ in $\Omega$}\label{cv2}\\
&\frac{\mathrm{d}(z,\mathrm{D}w)}{\mathrm{d}|\mathrm{D}^sw|}(x)=f^\infty(\frac{\mathrm{d}\mathrm{D}^sw}{\mathrm{d}|\mathrm{D}^sw|}(x))& |\mathrm{D}^sw|-\text{a.e.\ in $\Omega$}\label{cv3}\\
&z(x)\in \overline{\ran( \de_\xi f)}&\mathcal{L}^d-\text{a.e.\ in $\Omega$}\label{cv4}\\
&[z,\nu_x]\cdot(u_0-w)=f^\infty((u_0-w)(x)\otimes \nu_x)&\mathcal{H}^{d-1}-\text{a.e.\ on $\de\Omega$}\label{cv5}
\end{align}
\end{lemma}


We remark that, thanks to the uniform continuity of $\mathrm{D}_\xi f^\infty$ and the formula \eqref{nt rep}, for such a $z$ it also holds that \begin{align}
[z,\nu_x](x)\in \big\{\scalar{z_0}{\nu_x}\,\big|\, z_0\in \overline{\ran(\de_\xi f(x,\cdot))}\big\}\label{z tr pw}\end{align} pointwise a.e.\ on $\de\Omega$ in both the scalar and the vector setting (where one interprets the product as a matrix-vector product in the vectorial setting).

Thanks to \eqref{quant fen}, we also see that \eqref{cs4} (resp.\ \eqref{cv4}) imply that \begin{align}
[z,\nu_x]=\mathrm{D}_\xi f^\infty(x,(u_0-w)(x)\otimes \nu_x)\nu_x \quad \text{ if $w(x)\neq u_0(x)$, up to a $\mathcal{H}^{d-1}$-zero set}\label{good bd cond}
\end{align}
where the product on the right is understood as a matrix-vector product and the gradient is understood as a $n\times d$ matrix.

For future reference, let us also note that \eqref{cs4} resp.\ \eqref{cv4} together with \eqref{ball subdiff} also imply that \begin{align}
\norm{z}_{L^\infty(\Omega,\R^{n\times d})}\leq C_1\lesssim 1 \label{uni bd z}
\end{align}
with the $C_1$ from \ref{A2} resp.\ \ref{B2}.

There are many versions and special cases of this characterisation in the literature, the closest probably being \cite[Sec.\ 4]{gorny2022duality}, which additionally assumes that $f^\infty$ is even and $n=1$ and \cite{Meyer}, which considers $f$ which are rough in $x$ and only depend on a part of the gradient, which leads to the result being written in rather implicit way.

\begin{proof}[Proof of the Lemmata \ref{sca char} and \ref{vec char}]
We want to reduce this to the results from \cite{Meyer}. Let $\overline{\mathcal{F}}_{u_0}$ denote the functional defined on $L^2(\Omega, \R^n)$ by \begin{align*}
\overline{\mathcal{F}}_{u_0}(w):=\inf_{(W_{u_0}^{1,1}\cap L^2)(\Omega,\R^n)\ni w_m\xrightarrow{L^2} w}\int_\Omega f(x,\mathrm{D}w_m)\dx.
\end{align*}
This is convex and lower-semicontinuous as it is a lower-semicontinuous hull of convex functionals by definition.

In \cite[Thm.\ 6.4 and 6.5]{Meyer} (with $\mathcal{A}=\mathrm{D}$) it is shown that for $v\in L^2(\Omega, \R^n)$ it holds that $v\in \de\overline{\mathcal{F}}_{u_0}(w)$ if and only if $w\in BV(\Omega,\R^n)$ and if there is a $z\in X_2(\Omega,\R^{n\times d})$ with the following properties \begin{align}
&-\div z=v \quad\text{ in $\Omega$}\label{alt char1}\\
&f^*(x,z)\in L^1(\Omega)\label{alt char2}\\
&\int_\Omega (z,\mathrm{D}u)-\int_\Omega f^*(x,z)\dx+\int_{\Omega}[z,\nu_x]\cdot(u_0-u)\dH(x)=\overline{\mathcal{F}}_{u_0}(u).\label{alt char3}
\end{align}
We first note that we have $\overline{\mathcal{F}}_{u_0}=\mathcal{F}_{u_0}$ as a consequence of Proposition \ref{rel func} b).
Hence it merely needs to be checked that \eqref{alt char2}, \eqref{alt char3} are equivalent to \eqref{cs2}-\eqref{cs5} (resp.\ \eqref{cv2}-\eqref{cv5}).

First, assume \eqref{alt char2} and \eqref{alt char3} hold. As explained in the preliminary section $f^*(x,\cdot)=+\infty$ outside of $\overline{\ran(\de_\xi f(x,\cdot))}$ and therefore \eqref{alt char2} implies \eqref{cs4} (resp.\ \eqref{cv4}).

By Fenchel duality, we have that \begin{align*}
\scalar{z(x)}{\mathrm{D}^a w(x)}-f^*(x,z(x))\leq f(x,\mathrm{D}^a w(x))
\end{align*}
 a.e.\ and if equality holds at some point, then $z(x)\in \de_\xi f(x,\mathrm{D}^aw(x)))$. Similarly, thanks to \eqref{cs4} (resp.\ \eqref{cv4}), as well as \eqref{anz rep} and \eqref{doal}, we have \begin{align*}
\frac{\mathrm{d}(z,\mathrm{D}w)}{\mathrm{d}|\mathrm{D}^s w|}(x)\leq f^\infty(x,\frac{\mathrm{d}\mathrm{D}^s w}{\mathrm{d}|\mathrm{D}^s w|})(x)\quad \text{ for $|\mathrm{D}^s w|$-a.e.\ $x$.}
\end{align*}
Similarly, since we have \eqref{cs4} (resp.\ \eqref{cv4}) we also have \eqref{z tr pw} and therefore by \eqref{doal} we also must have \begin{align*}
\scalar{[z,\nu_x]}{u_0-w}\leq f^\infty(x,(u_0-w)\otimes \nu_x)\quad \mathcal{H}^{d-1}-\text{a.e.}
\end{align*}
It follows from \eqref{alt char3} and the Definition \eqref{def fu0} of $\mathcal{F}_{u_0}$, that we must actually have equality a.e.\ in all three of these which shows \eqref{cs2}, \eqref{cs3} and \eqref{cs5} (resp.\ \eqref{cv2}, \eqref{cv3} and \eqref{cv5}).

If on the other hand \eqref{cs2}-\eqref{cs5} (resp.\ \eqref{cv2}-\eqref{cv5}) hold, then we have that \begin{align*}
\scalar{z}{\mathrm{D}^aw(x)}-f^*(x,z)=f(x,\mathrm{D}^a w(x))\quad \text{$\mathcal{L}^d$-a.e.,}
\end{align*}
in particular, this implies that $f^*(x,z)\in L^1$ (and therefore \eqref{alt char2}), as the integrals of the other two terms exist. 

Summing this up with \eqref{cs3} and \eqref{cs5} (resp.\ \eqref{cv3} and \eqref{cv5}) and using the decomposition of the Anzelotti pairing into absolutely continuous and singular parts yields \eqref{alt char3} by the definition \eqref{def fu0} of $\mathcal{F}_{u_0}$.
\end{proof}




We can now formulate a suitable form of the Euler-Lagrange equations for \eqref{gen rel prob} and \eqref{vproblem}.

\begin{lemma}\label{sc equiv}
Suppose that the Assumptions \ref{A1}-\ref{A6} of Theorem \ref{T1} hold and let $u\in L^2(\Omega)$ be given, then it is a minimizer of \eqref{gen rel prob} if and only if \begin{align}
-(\lambda(u-h)+g)\in \de \mathcal{F}_{u_0}(u).\label{subdiff equiv}
\end{align}
\end{lemma}

\begin{lemma}\label{ve equiv}
Suppose that the Assumptions \ref{B1}-\ref{B3} of Theorem \ref{T2} hold and let $u\in L^2(\Omega,\R^n)$ be given, then it is a minimizer of \eqref{vproblem} if and only if \begin{align*}
0\in \de \mathcal{F}_{u_0}(u).
\end{align*}
\end{lemma}


\begin{proof}[Proof of Lemma \ref{sc equiv} and Lemma \ref{ve equiv}]
We only show Lemma \ref{sc equiv}, the other one is easier due to the missing lower-order terms.

Suppose that \eqref{subdiff equiv} holds. It follows from Proposition \ref{rel func} b) that it is enough to consider competitors in $L^2(\Omega)$ to show that $u$ minimizes \eqref{gen rel prob}. Let $v\in L^2(\Omega)$ be one. Then it holds that \begin{align*}
&\mathcal{F}_{u_0}(v)-\mathcal{F}_{u_0}(u)\geq -\scalar{\lambda(u-h)+g}{v-u}\\
&\iff\Bigg\{\: \begin{aligned}&\mathcal{F}_{u_0}(v)+\scalar{g}{v}+\frac{1}{2}\scalar{\lambda (v-h)}{(v-h)}-\frac{1}{2}\scalar{\lambda(v-u)}{(v-u)}\\
&\geq  \mathcal{F}_{u_0}(u)+\scalar{g}{u}+\frac{1}{2}\scalar{\lambda (u-h)}{(u-h)}\end{aligned}\\
&\implies\mathcal{F}_{u_0}(v)+\scalar{g}{v}+\frac{1}{2}\scalar{\lambda (v-h)}{(v-h)}\geq  \mathcal{F}_{u_0}(u)+\scalar{g}{u}+\frac{1}{2}\scalar{\lambda (u-h)}{(u-h)},
\end{align*}
showing that $u$ is a minimizer.

Suppose that $u\in L^2(\Omega)$ is a minimizer and consider a competitor $v\in L^2(\Omega)$, then, by convexity, it holds for $t\in (0,1]$ that \begin{align*}
\mel\mathcal{F}_{u_0}(v)-\mathcal{F}_{u_0}(u)\geq \frac{1}{t}(\mathcal{F}_{u_0}(u+t(v-u))-\mathcal{F}_{u_0}(u))\\
&\geq \frac{1}{t}\left(\frac{-1}{2}\scalar{\lambda(u+t(v-u)-h)}{u+t(v-u)-h)}+\frac{1}{2}\scalar{\lambda(u-h)}{u-h)}-t\scalar{g}{v-u}\right),
\end{align*}
where the second step follows from $u$ being a minimizer. Letting $t\searrow 0$ shows that \begin{align*}
\mathcal{F}_{u_0}(v)-\mathcal{F}_{u_0}(u)\geq -\scalar{\lambda(u-h)+g}{v-u}
\end{align*}
which shows \eqref{subdiff equiv}.
\end{proof}

\section{Proof of the main theorems \ref{T1} and \ref{T2}}\label{S4}
\subsection{A few words about the strategy}
Unlike other works, our approach is not based on barriers. Instead, we will transform $\Omega$ to a ball in the scalar case and use the subdifferential characterisation to show integral estimates in small spherical caps (called $D_x^\eps$ below) for $u-\hat{u}_x$, where $\hat{u}_x$ is a suitable extension of $u_0$.
 The most important point in the $BV$-case is that if the boundary condition \eqref{good bd cond} for $z$ holds on most of some spherical cap $C_x^\eps=\de\Omega\cap \de D_x^\eps$, then, as a consequence of the uniform convexity of the set $\ran(\de_\xi f^\infty)$ of allowed values for $z$, one actually controls $z$ in $D_x^\eps$ too as shown in Lemma \ref{cont z1}, which allows for sufficiently strong estimates.
 
The fact that $u_0\in BV$ enters through estimates on the gradient for the extension $\hat{u}_x$ (see Lemma \ref{char1}), the crucial point about it is that we can keep control of the modulus of integrability of the absolutely continuous part of the gradient of the extension (see Lemma \ref{dlvp}). 

The proof then proceeds by using the Morse covering theorem to cover the bad set $\{u\neq u_0\}$ with caps $C_x^\eps$ for which one has this control over $z$ and on which $|\mathrm{D}^su_0|$ is small, and recovering the trace from the weighted integrals over the $D_x^\eps$, as these become infinitesimally close to the boundary in the limit.\smallskip

In the $W^{\alpha,p}$-case the basic scheme is the same, though instead of an exact extension of $u_0$, we use local averages, which makes a few gradients disappear, but comes at the cost of having to control how well $u_0$ is approximated by its local averages, which is the part that needs the fractional Sobolev regularity.

The $C^0$ case is easier and mostly the $p\rightarrow \infty$-limit of the estimates in the $W^{\alpha,p}$-case.
\smallskip

The vectorial case $n\geq 1$ uses the same scheme as the $BV$-case, where there are extra limitations on the structure of $f$ coming from the fact that splitting into the positive and negative parts, as it is done in the scalar case, does not work well anymore, and one hence has to work in a way that handles both parts at once. In particular, we work with the original $\Omega$, which is the part that requires uniform convexity of the boundary.

\subsection{Preparations for the proof of Theorem \ref{T1}}
We will assume in the entire proof that the Assumptions \ref{A1}-\ref{A6} hold and all of the lemmata use some of them, though we will not explicitly state this.

\subsubsection{Transformation of the domain} 

It will be much easier to carry out the proof if $U=\de\Omega$ is a subset of a sphere. We will therefore use a diffeomorphism and show that the assumptions are stable under this, as the following Lemma shows. Furthermore, if this sphere has very small curvature, we can also obtain a more useful version of the curvature condition.

\begin{lemma}\label{trans lemma}
Suppose \ref{A1}-\ref{A6} hold.

\noindent\textbf{a)} Let $\Phi:\Omega\rightarrow \Omega'$ be a $C^2$-diffeomorphism. If we set \begin{align*}
&f_\Phi(y,\xi):=|(\det \mathrm{D}\Phi^{-1})(y)|f(\Phi^{-1}(y),(\mathrm{D}\Phi^{-1}(y))^{-T}\xi)\\
&\lambda_\Phi(y):=|(\det \mathrm{D}\Phi^{-1})(y)|\lambda(\Phi^{-1}(y))\\
&g_\Phi(y):=|(\det \mathrm{D}\Phi^{-1})(y)|g(\Phi^{-1}(y))\\
&h_\Phi(y):=h(\Phi^{-1}(y))
\end{align*}
for $y\in \Omega'$ and $\xi\in \R^d$, the assumptions \ref{A1}-\ref{A6} hold with $f_\Phi,\lambda_\Phi, g_\Phi,h_\Phi, u_0\circ \Phi^{-1}, \Phi(U)$ in place of $f,g,\lambda,h,u_0, U$ (with different constants).
Furthermore, \eqref{subdiff equiv} holds if and only if \begin{align}
-\left(\lambda_{\Phi}(u\circ \Phi^{-1}-h_\Phi)+ g_{\Phi}\right)\in \de(\Phi^*\mathcal{F}_{u_0\circ \Phi^{-1}})(u\circ \Phi^{-1}),\label{sd pushforward}
\end{align}
where $\Phi^*\mathcal{F}_{u_0\circ \Phi^{-1}}$ refers to the functional $\mathcal{F}_{u_0\circ \Phi^{-1}}$ defined with respect to $f_\Phi$.
%

\noindent\textbf{b)} For every $x_0\in U$, there is an open $U'\subset U$ with $x_0\in U'$ such that there is an $\Omega'\subset \R^d$ and a $\Phi:\Omega\rightarrow \Omega'$ (depending on $U'$) such that: \begin{itemize}
\item $\Phi$ is a $C^2$-diffeomorphism,
\item \begin{equation}\Phi(U') \subset \de B_{\frac{1}{\kappa}}(0)\quad \text{for some $\kappa>0$},\label{ballbd}\end{equation}  
\item  for every open $V\subset \Phi(U')$ with a positive distance to the boundary of $\Phi(U')$, there are $\eps_0,\,c_0>0$ (depending on $V$ and $U'$), such that for all $x\in V$ and all $y\in \Omega'\cap B_{\eps_0}(x)$ it holds that \begin{align}
\pm \div_y \mathrm{D}_{\xi}f_{\Phi}^\infty(y,\pm  \nu_x^{\Omega'})\geq |g_{\Phi}(y)|+c_0,\label{impr cc}
\end{align}
where $\nu_x^{\Omega'}$ is the outer normal to $\Omega'$ at $x$.
\end{itemize}
\end{lemma}

The lemma is proven in Section \ref{S5}. 

Now it suffices to show that $u=u_0$ in every such $U'$ as in b). By the Lemma, it is not restrictive to assume that every point in b) was already true for $U'$ and that $\Phi$ is the identity (otherwise we use the diffeomorphism which the lemma yields, rename everything, and use that by a) all the assumptions remain true). We also consider $\kappa$ to be fixed of order $1$ and allow the implicit constants to depend on it.


Since the statement of the theorem is local, it is not restrictive to only show that $u-u_0$ vanishes on $U''$ for every (relatively) open set $U''\subset U'$ with positive distance to the boundary of $U'$. 
It is also not restrictive to assume that both $U''$ and $U'$ have a sufficiently regular (relative) boundary.

Let $x\in U''$ and $\eps>0$, we set \begin{align}
&C_x^\eps=B_\eps(x)\cap\de \Omega\label {def C1}\\
&D_x^\eps=\left\{y\in \Omega\cap B_\eps(x)\,\big|\,\scalar{y-x}{\nu_x}\geq-\frac{1}{2}\kappa\eps^2\right\}\label{def D1}\\
&E_x^\eps=\de D_x^\eps\backslash C_x^\eps=\left\{y\in \overline{\Omega}\cap \overline{B_\eps(x)}\,\big|\,\scalar{y-x}{\nu_x}=-\frac{1}{2}\kappa\eps^2\right\}.\label{def E1}
\end{align}
We only consider $\eps$ such that $\eps<\dist(U'',\de \Omega\backslash U')$, so that $C_x^\eps\subset U'$ and there is no intersection between $D_x^\eps$ and $\de \Omega\backslash U'$.
We furthermore also only consider $\eps$ smaller than the $\eps_0$ in part b) of Lemma \ref{trans lemma}, so that the statement of the Lemma is applicable in $D_x^\eps$ with a $c_0$ uniform in $x\in U''$ and small $\eps$.

An easy calculation shows that $E_x^\eps$ is a piece of a plane with normal $\nu_x$ and that the boundary of the spherical cap $D_x^\eps$ is precisely $E_x^\eps\cup C_x^\eps$.  See also Figure \ref{fig1} for a sketch.

\begin{figure}
\includegraphics[width=6cm]{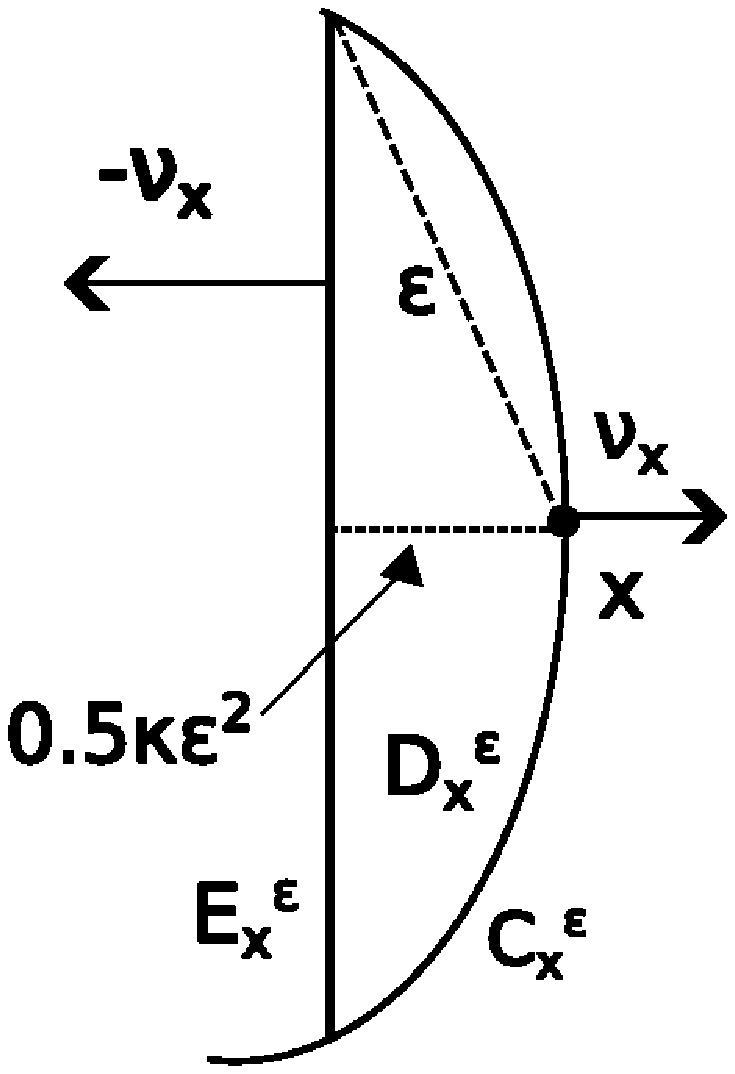}\vspace{-2cm}
\caption{Sketch of the geometry in the proof of Theorem \ref{T1}}\label{fig1}
\end{figure}

 We furthermore have the following elementary geometric properties: \begin{lemma}\label{geo lemma1}
\textbf{a)} It holds that \begin{align}
&\mathcal{H}^{d-1}(C_x^\eps)\approx \eps^{d-1}\label{est C1}\\
&\mathcal{L}^{d}(D_x^\eps)\approx \eps^{d+1}\label{est D1}\\
&\mathcal{H}^{d-1}(E_x^\eps)\approx \eps^{d-1}.\label{est E1}
\end{align}
uniformly in $x\in U''$ and small $\eps$.

\textbf{b)} It holds that \begin{align}\label{deri C}
\frac{\mathrm{d}}{\mathrm{d}\eps} \mathcal{H}^{d-1}(C_x^\eps)\approx \eps^{d-2}
\end{align}
uniformly in $x\in U''$ and small $\eps$ for $d>1$, for $d=1$ the derivative vanishes.

\textbf{c)} If $D_x^\eps$ and $D_{x'}^{\eps'}$ intersect for some $x,x'\in U''$, then $C_x^\eps$ and $C_{x'}^{\eps'}$ intersect too.
\end{lemma}

The Lemma is proven in Section \ref{S52}.

\subsection{Proof of Theorem \ref{T1} in the case \texorpdfstring{$u_0\in  BV$}{u0BV}}

We will only show that $u\leq u_0$ holds $\mathcal{H}^{d-1}$-a.e.\ on $U''$, the other inequality follows in the same way by appropriately changing the signs.

We will use a comparision vector field $\hatz$, which we define as \begin{align}
\hatz(y)=\mathrm{D}_\xi f^\infty(y,- \nu_x).\label{def hatz}
\end{align} 
%
%
%
It follows from the Assumption \ref{A2} that \begin{equation}
\hatz\in C^1(\Omega)\subset X_2(\Omega)\label{reg hatz}
\end{equation}
and by \eqref{ball subdiff}, we also have \begin{align}
\norm{\hatz}_{L^\infty(C_x^\eps)}\leq C_1\lesssim1.\label{bound hatz}
\end{align}
We will then define the function $\hatu$ as the extension of $u_0$ along the characteristic curves of $\hat{z}_{x}$, more precisely:
\begin{lemma}\label{char1}
There is a function $\hat{u}_x\in BV(D_x^\eps)$ such that \begin{align}
\hatu&=u_0\quad \text{ on $C_x^\eps$}\label{charu1}\\
(\hatz,\mathrm{D}\hatu)&=0 \qquad\text{in $D_x^\eps$.}\label{charu2}
\end{align}
Furthermore, it holds that \begin{align}
&\norm{\hatu}_{BV(D_x^\eps)}\lesssim \eps^2\norm{u_0}_{BV(C_x^\eps)}\label{est hatu1}
\end{align}
uniformly in $\eps$ and $x\in U''$.  
\end{lemma}
The Lemma is proven in Section \ref{S53}.

Crucially, we can also lift de-La-Vall\'ee-Poussin-type estimates for the absolutely continuous part of $\mathrm{D}u_0$ to the absolutely continuous part of $\mathrm{D}\hatu$.
 \begin{lemma}\label{dlvp}
There is a non-decreasing function $i:\R_{\geq 0}\rightarrow \R_{\geq 0}$ depending on $u_0$, but not on $\eps$ or $x$, such that
\begin{align*}
\lim_{t\rightarrow +\infty} \frac{i(t)}{t}=+\infty,
\end{align*}
it holds that \begin{align}
 \int_{U'} i(|\mathrm{D}^au_0|(y))\dH(y)<\infty\label{ineq i1}
\end{align}
and \begin{align}
\int_{D_x^\eps} i(|\mathrm{D}^a\hatu|(y))\dy\lesssim_{u_0} \eps^2\int_{C_x^\eps} i(|\mathrm{D}^au_0|(y))\dH(y),\label{ineq i2}
\end{align}
uniformly in $x\in U''$ and $\eps>0$.
Furthermore it holds that \begin{align}
|\mathrm{D}^s\hatu|(D_x^\eps)\lesssim \eps^2|\mathrm{D}^su_0|(C_x^\eps),\label{ineq i3}
\end{align}
uniformly in $x\in U''$ and $\eps$.
\end{lemma}
This lemma is also proven in Section \ref{S53}.

%
%

We now take $z\in X_2(\Omega,\R^d)$ such that the conditions \eqref{cs2}-\eqref{cs4} in the characterisation of the subdifferential hold with $v=-\lambda(u-h)-g$ and $w=u$, this is possible because of the characterisation of minimizers in the Lemmata \ref{sca char} and \ref{sc equiv}.

%
We use the truncations $T_b$, defined in \eqref{def trunc}, where $b\in \R_{\geq 0}$ is some arbitrary number.
The space $BV$ is closed under the application of $T_b$, by e.g.\ the chain rule as described in Section \ref{S22}.
Therefore, we can use the Gauss-Green formula \eqref{gg form} for the domain $D_x^\eps$ and $(T_b(u)-T_b(\hatu))_{+}$, where we denote the positive part by $a_+=\max(a,0)$, as well as $(z-\hatz)$, to the effect that \begin{equation}\begin{aligned}\label{part int}
\mel\int_{C_x^\eps} [z-\hatz,\nu_y](T_b(u)-T_b(u_0))_{+}\dH(y)+\int_{E_x^\eps} [z-\hatz,-\nu_x](T_b(u)-T_b(\hatu))_{+}\dH(y)\\
&=\int_{D_x^\eps}(\lambda (u-h) +g-\div\hatz) (T_b(u)-T_b(\hatu))_+\dy+\int_{D_x^\eps}(z-\hat{z},\mathrm{D}(T_b(u)-T_b(\hatu))_{+}),
\end{aligned}\end{equation}
where we have used that the outer normal on $E_x^\eps$ is $-\nu_x$.
We claim that the left-hand side is non-positive: Thanks to \eqref{cv5}, we have \begin{align}
[z,\nu_y]=-f^\infty(y,-\nu_y)\quad \text{on $C_x^\eps$} \label{znu 1}
\end{align}
whenever $u-u_0>0$ at $y$, due to the positive $1$-homogeneity of $f^\infty$. By the regularity and definition of $\hatz$ and \eqref{doal}, we have \begin{align}
[\hatz,\nu_y] =\scalar{\mathrm{D}_\xi f^\infty(y,-\nu_x)}{\nu_y}\geq -f^\infty(y,-\nu_y).\label{znu 2}
\end{align}
Since the integral over $C_x^\eps$ is automatically zero at the points where $u\leq u_0$, this shows that \begin{align*}
\int_{C_x^\eps} [z-\hatz,\nu_y](T_b(u)-T_b(u_0))_{+}\dH(y)\leq 0.
\end{align*}
Regarding the other integral, we have on $E_x^\eps$ \begin{align*}
    [\hat{z}_x,-\nu_x](y)=\scalar{\mathrm{D}_\xi f^\infty(y,-\nu_x)}{-\nu_x}= f^\infty(y,-\nu_x),
\end{align*}  
by \eqref{doal} as well as the regularity and definition of $\hatz$. Since we also have \begin{align*}
[z,-\nu_x]\in \left\{z_0\cdot(- \nu_x)\,\big|\, z_0\in \ran(\de_\xi f^\infty(y,\cdot)\right\}
\end{align*}
thanks to \eqref{anz rep}, \eqref{cs4} and the uniform continuity of $\mathrm{D}f^\infty$, we see from \eqref{doal} that \begin{align*}
[z-\hatz,-\nu_x]\leq 0,
\end{align*}
yielding \begin{align*}
\int_{E_x^\eps} [z-\hatz,-\nu_x](T_b(u)-T_b(\hatu))_{+}\dH(y)\leq 0.
\end{align*}
We hence obtain that \begin{align}\label{part int2}
\int_{D_x^\eps}(\lambda (u-h) +g-\div\hatz) (T_b(u)-T_b(\hatu))_+\dy\leq -\int_{D_x^\eps}(z-\hat{z},\mathrm{D}(T_b(u)-T_b(\hatu))_{+}).
\end{align}
We can further estimate the right-hand side, using Lemma \ref{anz ch} and the triangle inequality as \begin{align}\label{anz tr}
\mel-\int_{D_x^\eps}(z-\hat{z},\mathrm{D}(T_b(u)-T_b(\hatu))_{+})
\leq \int_{D_x^\eps}|(z-\hat{z},\mathrm{D}T_b(\hatu))|+|(z-\hat{z},\mathrm{D}u)_-|,
\end{align}
where the lower ``$-$'' denotes the negative part of the measure.


 We would like to show that the right-hand side of \eqref{part int2} becomes suitably small if $z$ equals $\mathrm{D}_\xi f^\infty(y,- \nu_y)$ on most of $C_x^\eps$, which will follow from the next Lemma.

\begin{lemma}\label{cont z1}
\textbf{a)} For every $\delta>0$ and every $x\in U''$ and $\eps>0$ it holds that, \textbf{if} we have \begin{align}\label{del ass}
\mathcal{H}^{d-1}\left(\{y\in C_x^\eps\,\big|\,[z,\nu_y](y)\neq \scalar{\mathrm{D}_\xi f^\infty(y,- \nu_y)}{\nu_y}\}\right)\leq \delta \mathcal{H}^{d-1}(C_x^\eps),\end{align}
 \textbf{then} \begin{align}\label{del stat}
\int_{D_x^\eps}|z(y)-\hatz(y)|^2\dy\lesssim \eps^{d+1}(\eps+\delta)+\eps^2\int_{D_x^\eps}|\div z|\dy.
\end{align}
%
\textbf{b)} There is a non-decreasing, bounded function $k:\R_{\geq 0}\rightarrow \R_{\geq 0}$, with $\lim_{t\searrow 0} k(t)=0$, only depending on $f$, such that for all $\eps,x$ it holds that
\begin{align}
\int_{D_x^\eps} \left|(z-\hatz,\mathrm{D}u)_-\right|\leq \int_{D_x^\eps} k(|z-\hatz|)\dy,\label{k est}
\end{align}
where the lower ``$-$'' denotes the negative part of the measure. If $f$ is positively $1$-homogeneous, one can take $k=0$.
\end{lemma}
The lemma is proven in Sections \ref{S54} and \ref{S55}.

\noindent We fix an arbitrary (small) $\delta>0$ now.
By Radon-Nikodym \cite[Thm.\ 1.153]{fonseca2007modern}\footnote{\label{fn2}The reference only states the theorems for subsets of the Euclidean space, for subsets of $\de \Omega$ it easily follows by parametrizing the boundary.}, applied to the indicator function $\mathds{1}_{u>u_0}$ and the measure $|\mathrm{D}^su_0|$ with respect to the measure $\mathcal{H}^{d-1}$ on $U'$, we know that for $\mathcal{H}^{d-1}$-a.e.\ $x\in U''\cap  \{u> u_0\}$ there exist an $\eps_{0,x}>0$, which can be chosen arbitrarily small, such that for all $\eps\leq \eps_{0,x}$ we have \begin{align}
&\mathcal{H}^{d-1}(\{y\in C_x^{\eps}\,\big|\, u(y)\leq u_0(y)\})\leq \delta\mathcal{H}^{d-1}(C_x^{\eps})\label{ball est1}\\
&|\mathrm{D}^su_0|(C_x^{\eps})\leq \delta\mathcal{H}^{d-1}(C_x^{\eps}).\label{ball est2}
\end{align}
%
%
In particular, by the Morse covering theorem \cite[Thm.\ 1.147]{fonseca2007modern}\textsuperscript{\ref{fn2}}, there exist a (countable) family of $C_{x_j}^{\eps_j}$ (which depends on $\delta$), such that \begin{align*}\mathcal{H}^{d-1}\left((U''\cap \{u>u_0\})\backslash \bigcup_j C_{x_j}^{\eps_j}\right)=0\end{align*}
and such that \eqref{ball est1} and \eqref{ball est2} hold for each of these spherical caps and such that the elements of this family are \textit{pairwise disjoint}. Furthermore, we may pick $\max_j \eps_j$ arbitrarily small.

By the boundary condition \eqref{cs5}, and its equivalent form \eqref{good bd cond}, we see that \eqref{ball est1} implies \eqref{del ass}, and in particular, on each such element of the covering, the assertion of the conditional statement of Lemma \ref{cont z1} a) holds.

For further reference, we note that, as a consequence of the disjointness and of Lemma \ref{geo lemma1}, it holds that \begin{align}
\sum_j \eps_j^{d-1}\approx\sum_j \mathcal{H}^{d-1}(C_{x_j}^{\eps_j}) \leq \mathcal{H}^{d-1}(U')\lesssim 1.\label{sum eps}
\end{align}
Now if we divide by $\eps_j^2$, and sum \eqref{part int2} together with \eqref{anz tr} over the covering, it holds that

\begin{equation} \begin{aligned}\label{eff est1}
\mel\sum_j\eps_j^{-2}\int_{D_{x_j}^{\eps_j}}(\lambda(u-h) +g-\div\hat{z}_{x_j}) (T_b(u)-T_b(\hat{u}_{x_j}))_+\dy\\
&\leq \sum_j\eps_j^{-2}\left( \int_{D_{x_j}^{\eps_j}} |(z-\hat{z}_{x_j},\mathrm{D}T_b(\hat{u}_{x_j}))|+\int_{D_{x_j}^{\eps_j}} |(z-\hat{z}_{x_j},\mathrm{D}u)_-|\right).
\end{aligned}\end{equation}
The rest of the proof in this case now consists of proving the following two claims.
\begin{claim}\label{claim1}
The right-hand side in \eqref{eff est1} goes to $0$ as $\delta,\max_j\eps_j\searrow 0$.
\end{claim}
\begin{claim}\label{claim2}
If the limsup of the left-hand side in \eqref{eff est1} is $\leq 0$ as $\delta,\max_j\eps_j\searrow 0$ then $u\leq u_0$ in $U''$.
\end{claim}

We start with proving the first claim. It holds that \begin{align}
\mel\sum_j\eps_j^{-2}\left( \int_{D_{x_j}^{\eps_j}} |(z-\hat{z}_{x_j},\mathrm{D}T_b(\hat{u}_{x_j}))|+\int_{D_{x_j}^{\eps_j}} |(z-\hat{z}_{x_j},\mathrm{D}u)_-|\right)\nonumber\\
&\lesssim\sum_j\eps_{j}^{-2}\bigg(\int_{D_{x_j}^{\eps_j}}|z-\hat{z}_{x_j}|(y)|\mathrm{D}^aT_b(\hat{u}_{x_j})|(y)\dy+\left(\norm{z}_{L^\infty}+\norm{\hat{z}_{x_j}}_{L^\infty}\right)|\mathrm{D}^sT_b(\hat{u}_{x_j})|(D_{x_j}^{\eps_j})\nonumber\\
&\quad+\int_{D_{x_j}^{\eps_j}}k(|z-\hat{z}_{x_j}|)\dy\bigg)\nonumber\\
&\lesssim\sum_j\eps_{j}^{-2}\left(\int_{D_{x_j}^{\eps_j}}|z-\hat{z}_{x_j}|(y)|\mathrm{D}^a\hat{u}_{x_j}|(y)\dy+|\mathrm{D}^s\hat{u}_{x_j}|(D_{x_j}^{\eps_j})+\int_{D_{x_j}^{\eps_j}}k(|z-\hat{z}_{x_j}|)\dy\right),\label{c1}
\end{align}
where in the first step we have split the first Anzelotti pairing into the absolutely continuous and singular part, which we estimated with its total variation and we have used \eqref{k est} for the second pairing, while the last step follows directly from \eqref{measure trunc}, as well as the bounds \eqref{bound hatz} and \eqref{uni bd z} on $z$ and $\hatz$.

It follows directly from \eqref{ineq i3}, \eqref{ball est2} and \eqref{sum eps} that \begin{align}
\sum_j\eps_{j}^{-2}|\mathrm{D}^s\hat{u}_{x_j}|(D_{x_j}^{\eps_j})\lesssim \delta \sum_j \mathcal{H}^{d-1}(C_{x_j}^{\eps_j})\lesssim \delta\rightarrow 0. \label{p1}
\end{align}
Regarding the last summand in \eqref{c1}, we split into the parts where $|z-\hat{z}_{x_j}|\geq \sigma$ and where $|z-\hat{z}_{x_j}|\leq \sigma$ for some real $\sigma$ to be chosen later, to see that \begin{align}
\mel\sum_j\eps_{j}^{-2}\int_{D_{x_j}^{\eps_j}}k(|z-\hat{z}_{x_j}|)\dy \leq k(\sigma)\sum_j\eps_{j}^{-2}\mathcal{L}^d(D_{x_j}^{\eps_j})+\norm{k}_{\sup}\sum_j\eps_{j}^{-2}\mathcal{L}^d(D_{x_j}^{\eps_j}\cap\{|z-\hat{z}_{x_j}|\geq \sigma\})\nonumber\\
&\lesssim k(\sigma)\sum_j \eps_{j}^{d-1}+\sum_j \frac{1}{\sigma^2}\left(\int_{D_{x_j}^{\eps_j}}|\div z|\dy+\eps_j^{d-1}(\eps_j+\delta)\right),\label{est kpart}
\end{align}
here the last step used \eqref{est D1} for the first summand and Lemma \ref{cont z1} a), as well as \eqref{est D1} and Chebyshevs' inequality for the second summand.

By Lemma \ref{geo lemma1} c) and the assumption on the $C_{x_j}^{\eps_j}$, the $D_{x_j}^{\eps_j}$ do not intersect and are also contained in a neighborhood of size $\approx\max_j\eps_j^2$ of the boundary of $\Omega$ by definition. Therefore, since $\div z=-\lambda(u-h)-g$ is integrable by the assumptions on the functions, we have by e.g.\ dominated convergence \begin{align}
\sum_j\int_{D_{x_j}^{\eps_j}}|\div z|\dy\xrightarrow{\max_j\eps_j}0.\label{dom conv}
\end{align}
We can furthermore let $\sigma=\sigma(\delta,\max_j\eps_j)\geq \max(\delta^\frac{1}{3},(\max_j \eps_j)^\frac{1}{3})$ go to $0$ so slowly that \begin{align}
\frac{1}{\sigma^2}\left(\sum_j\int_{D_{x_j}^{\eps_j}}|\div z|\dy+\eps_j^{d-1}(\eps_j+\delta)\right)\xrightarrow{\max_j\eps_j,\delta\rightarrow 0}0,\label{dom conv2}
\end{align}
since the second part of the sum goes to $0$ by \eqref{sum eps}.
Hence, using \eqref{sum eps} for the first summand in \eqref{est kpart}, we see that \begin{align}\label{p2}
\sum_j\eps_{j}^{-2}\int_{D_{x_j}^{\eps_j}}k(|z-\hat{z}_{x_j}|)\dy\lesssim k(\sigma)+\frac{1}{\sigma^2}\left(\sum_j\int_{D_{x_j}^{\eps_j}}|\div z|\dy+\eps_j^{d-1}(\eps_j+\delta)\right)\xrightarrow{\max_j\eps_j,\delta\rightarrow 0} 0.
\end{align}
The first part of the sum in \eqref{c1} uses a similar argument, we split $|\mathrm{D}^a \hat{u}_{x_j}|$ into the small and big parts to see that \begin{equation}\begin{aligned}
\mel\sum_j\eps_{j}^{-2}\int_{D_{x_j}^{\eps_j}}|z-\hat{z}_{x_j}|(y)|\mathrm{D}^a\hat{u}_{x_j}|(y)\dy\\
&\lesssim \sigma'\sum_j\eps_{j}^{-2} \int_{D_{x_j}^{\eps_j}}|z-\hat{z}_{x_j}|\dy
+\sum_j \eps_j^{-2}(\norm{z}_{L^\infty}+\norm{\hat{z}_{x_j}}_{L^\infty})\int_{D_{x_j}^{\eps_j}}\mathds{1}_{|\mathrm{D}^a\hat{u}_{x_j}|\geq \sigma'}|\mathrm{D}^a\hat{u}_{x_j}|(y)\dy\label{1st part}
\end{aligned}\end{equation}
with a real number $\sigma'>0$ to be fixed later.
For the first sum, we can use Cauchy-Schwarz on $\cup_j D_{x_j}^{\eps_j}$ together with \eqref{del stat} and \eqref{est D1} to see that \begin{align}
\sum_j\eps_{j}^{-2} \int_{D_{x_j}^{\eps_j}}|z-\hat{z}_{x_j}|\dy\leq \left(\sum_j \eps_j^{-2}\int_{D_{x_j}^{\eps_j}}1\dy\right)^\frac{1}{2}\left(\sum_j\eps_{j}^{-2} \int_{D_{x_j}^{\eps_j}}|z-\hat{z}_{x_j}|^2\dy\right)^\frac{1}{2}\nonumber\\
\lesssim \left(\sum_j \eps_j^{d-1}\right)^\frac{1}{2}\left(\sum_j\eps_{j}^{d-1}(\eps_j+\delta)+ \int_{D_{x_j}^{\eps_j}}|\div z|\dy\right)^\frac{1}{2}.\label{551}
\end{align}
By \eqref{sum eps} and \eqref{dom conv}, this goes to $0$.

The other summand in \eqref{1st part} on the other hand is estimated, using the bounds \eqref{bound hatz} and \eqref{uni bd z} on $z$ and $\hatz$, as well as Lemma \ref{dlvp}, as

\begin{align}
\mel\sum_j \eps_j^{-2}(\norm{z}_{L^\infty}+\norm{\hat{z}_{x_j}}_{L^\infty})\int_{D_{x_j}^{\eps_j}}\mathds{1}_{|\mathrm{D}^a\hat{u}_{x_j}|\geq \sigma'}|\mathrm{D}^a\hat{u}_{x_j}|(y)\dy\lesssim \frac{\sigma'}{i(\sigma')}\sum_j \eps_j^{-2}\int_{D_{x_j}^{\eps_j}}i(|\mathrm{D}^a\hat{u}_{x_j}|(y))\dy\nonumber\\
&\lesssim_{u_0} \frac{\sigma'}{i(\sigma')}\int_{U'} i(|\mathrm{D}^a u_0|(y))\dH(y)\xrightarrow{\sigma'\rightarrow +\infty}0,\label{last part}
\end{align}
where we used \eqref{ineq i2} and the disjointness of the sets in the last step.

Therefore, by \eqref{551} and \eqref{last part}, we see that, if we let $\sigma'\rightarrow \infty$ sufficiently slow, that \begin{align}
\sum_j\eps_{j}^{-2}\int_{D_{x_j}^{\eps_j}}|z-\hat{z}_{x_j}|(y)|\mathrm{D}^a\hat{u}_{x_j}|(y)\dy\xrightarrow{\max_j\eps_j,\delta\rightarrow 0}0.\label{p3}
\end{align}
Combining \eqref{c1}, \eqref{p1}, \eqref{p2} and \eqref{p3}, we see the Claim \ref{claim1}, and hence that \begin{align*}
\limsup_{\max_j\eps_j,\delta\rightarrow 0}\sum_j\eps_j^{-2}\int_{D_{x_j}^{\eps_j}}(\lambda(u-h) +g-\div\hat{z}_{x_j}) (T_b(u)-T_b(\hat{u}_{x_j}))_+\dy\leq 0.
\end{align*}

\subsubsection{Proof of the Claim \ref{claim2}}
We recall that the term $g-\div \hat{z}_{x_j}$ is $\gtrsim 1$ by the definition of $\hat{z}_{x_j}$ and \eqref{impr cc}.

Hence, using the monotonicity of the truncation, we see that in every $D_x^\eps$ it holds that \begin{align*}
\mel(\lambda(u-h)+g-\div \hatz)(T_b(u)-T_b(\hatu))_+\\
&\geq C(T_b(u)-T_b(\hatu))_++\lambda\big((u-\hatu)+(\hatu-h)\big)(T_b(u)-T_b(\hatu))_+\\
&\gtrsim (T_b(u)-T_b(\hatu))_+-Cb\lambda |\hatu-h|
\end{align*}
yielding, if we use \eqref{eff est1} and Claim \ref{claim1}, that \begin{align}
\mel\limsup_{\delta,\max_j\eps_j\searrow 0}\sum_j\eps_j^{-2}\int_{D_{x_j}^{\eps_j}}(T_b(u)-T_b(\hat{u}_{x_j}))_+\dy\lesssim_b\limsup_{\delta,\max_j\eps_j\searrow 0}\sum_j\eps_j^{-2}\int_{D_{x_j}^{\eps_j}}\lambda|\hat{u}_{x_j}-h|\dy.\label{lim est}
\end{align}


%

\noindent By the following Lemma, these weighted integrals converge to the traces of the corresponding functions in a uniform way up to a small cutoff of the radius. 
\begin{lemma}\label{trace lemma}
Let $\rho\in (0,1)$, then there is some constant $C$, depending on $\rho$, but not on $x$ or $\eps$, such that for every $x\in U''$ and $\eps>0$ and $w\in BV(D_x^\eps)$ it holds that \begin{align}\label{tr1}
\eps^{-2}\int_{D_x^\eps}|w|\dy\gtrsim_{\rho} \int_{C_x^{\rho\eps}}|w|\dH-C|\mathrm{D}w|(D_x^\eps).\end{align}
Furthermore
\begin{align}\label{tr2}
\eps^{-2}\int_{D_x^\eps}|w|\dy\lesssim |\mathrm{D}w|(D_x^\eps)+\int_{C_x^\eps}|w|\dH.
\end{align}
If additionally $h\in C^0(D_x^\eps)$, then 
\begin{align}\label{tr3}
\eps^{-2}\int_{D_x^\eps}|w-h|\dy\lesssim|\mathrm{D}w|(D_x^\eps)+\omega(2\eps)\eps^{d-1}+\int_{C_x^\eps}|w-h|\dH,
\end{align}
where $\omega(s)=\sup_{|y-y'|\leq s}|h(y)-h(y')|$ is the modulus of continuity of $h$.
\end{lemma}

The proof of the lemma is found in Section \ref{S56}.

Now by e.g.\ the coarea formula, we have $|\mathrm{D}\left((T_b(u)-T_b(\hatu))_+\right)|\leq |\mathrm{D}u|+|\mathrm{D}\hatu|$ for all $x$ and $b$.
Using \eqref{lim est}, we hence see that for each $\rho$ we have \begin{align*}
\mel\limsup_{\delta,\max_j\eps_j\searrow 0}\int_{\bigcup_j C_{x_j}^{\rho\eps_j}} (T_b(u)-T_b(u_0))_+\dH\\
&\lesssim_{\rho,b}  \limsup_{\delta,\max_j\eps_j\searrow 0}\sum_j\eps_j^{-2}\int_{D_{x_j}^{\eps_j}}\lambda|\hat{u}_{x_j}-h|\dy+\sum_j |\mathrm{D}u|(D_{x_j}^{\eps_j})+|\mathrm{D}\hat{u}_{x_j}|(D_{x_j}^{\eps_j}).
\end{align*}
Since the $D_{x_j}^{\eps_j}$ are disjoint by Lemma \ref{geo lemma1} c) and their union vanishes in the limit, we have \begin{align*}
\sum_j |\mathrm{D}u|(D_{x_j}^{\eps_j})\rightarrow 0,
\end{align*}
and the same holds for $h$ if $h\in BV$.
Similarly, by \eqref{est hatu1} and the disjointness it holds that \begin{align*}
\sum_j|\mathrm{D}\hat{u}_{x_j}|(D_{x_j}^{\eps_j})\lesssim \big(\max_j\eps_j^2\big)\sum_j |\mathrm{D}u_0|(C_{x_j}^{\eps_j})\lesssim\max_j\eps_j^2|\mathrm{D}u_0|(U')\rightarrow 0.
\end{align*}
Now applying \eqref{tr2} with $w=h-\hat{u}_{x_j}$ if $h\in BV$ or \eqref{tr3} with $w=\hat{u}_{x_j}$ if $h$ is continuous, we see from the assumption that $h=u_0$ on the boundary that \begin{align*}
\limsup_{\delta,\max_j\eps_j\searrow 0}\sum_j\eps_j^{-2}\int_{D_{x_j}^{\eps_j}}\lambda|\hat{u}_{x_j}-h|\dy\lesssim\begin{cases} \sum_j|\mathrm{D}\hat{u}_{x_j}|(D_{x_j}^{\eps_j})+|\mathrm{D}h|(D_{x_j}^{\eps_j})&\text{ if $h\in BV$}\\
\sum_j|\mathrm{D}\hat{u}_{x_j}|(D_{x_j}^{\eps_j})+\omega(2\max_j \eps_j)&\text{ if $h\in C^0$}.
\end{cases}
\end{align*}
In either case, this goes to $0$ as $\max_j \eps_j\rightarrow 0$. This shows that \begin{align}\label{fin lim}
\limsup_{\delta,\max_j\eps_j\searrow 0}\int_{\bigcup_j C_{x_j}^{\rho\eps_j}} (T_b(u)-T_b(u_0))_+\dH=0.
\end{align}
Finally, to deal with the $\rho$, we note that by \eqref{deri C} and \eqref{sum eps} it holds that\begin{align}\label{est def}
\sum_j\mathcal{H}^{d-1}(C_{x_j}^{\eps_j}\backslash C_{x_j}^{\rho\eps_j})\lesssim \sum_j (1-\rho)\eps_j^{d-1}\lesssim 1-\rho
\end{align}
and combining \eqref{fin lim} with  the fact that the $C_{x_j}^{\eps_j}$ cover $U''\cap \{u>u_0\}$ up to a $\mathcal{H}^{d-1}$-zero set by construction yields that for every $\rho\in (0,1)$ it holds that \begin{align*}
\limsup_{V\subset U',\, \mathcal{H}^{d-1}(V)\leq C(1-\rho)} \int_{U''\backslash V} (T_b(u)-T_b(u_0))_+\dH=0.
\end{align*}
Letting first $\rho\nearrow 1$ and then $b\rightarrow +\infty$ yields that $u=u_0$ in $U''$, showing this case of the theorem.\hfill\qedsymbol

\subsection{Proof of Theorem \ref{T1} if \texorpdfstring{$u_0\in W^{\alpha,p}$}{u0Wap}}

We use the same geometric setup and the same field $\hatz$ as defined in \eqref{def hatz}. We will again only show that $u\leq u_0$ a.e.\ on the boundary, the other case again follows from symmetry.
Instead of $u-\hatu$ we use $T_1(u-a)_{+}^{p}$, where $T_b(y)=\min(b,\max(y,-b))$ is again the trunctation and $a$ is some real number to be determined later.
This function lies in $BV$ by the chain rule for $BV$-functions as explained in Section \ref{S22}. We again use $z$ fulfilling the conditions in the characterisation of the subdifferential in Lemma \ref{sca char} and the same $\hat{z}_x$, defined in \eqref{def hatz}.
We see by the Gauss-Green formula \eqref{gg form} that \begin{align*}
\mel\int_{C_x^\eps} [z-\hatz,\nu_y]T_1(u-a)_{+}^{p}\dH(y)+\int_{E_x^\eps} [z-\hatz,-\nu_x]T_1(u-a)_{+}^{p}\dH\\
&=\int_{D_x^\eps}(\lambda(u-h)+g-\div\hatz)T_1(u-a)_{+}^{p}\dy+\int_{D_x^\eps}(z-\hat{z}_x,\mathrm{D}T_1(u-a)_{+}^{p})
\end{align*}
where we have again used that $-\nu_x$ is the outer normal on $E_x^\eps$.
The integral over $E_x^\eps$ is non-positive as in the $BV$-case, because $[\hat{z}_x,\nu_x]\geq [z,\nu_x]$ thanks to \eqref{quant fen}, \eqref{z tr pw} and the definition \eqref{def hatz} of $\hat{z}_x$, and hence we have the estimate \begin{equation}\begin{aligned}\label{part int 2}
\mel\int_{D_x^\eps}(\lambda(u-h)+g-\div\hatz)T_1(u-a)_{+}^{p}\dy\\
&\leq \int_{C_x^\eps} [z-\hatz,\nu]T_1(u-a)_{+}^{p}\dH(y)-\int_{D_x^\eps}(z-\hat{z},\mathrm{D}T_1(u-a)_{+}^{p}).
\end{aligned}\end{equation}
Observe that the integrand in the integral over $C_x^\eps$ can only be positive when $a\leq u$. Further observe that $[z-\hatz,\nu_y]\leq 0$ if $u\geq u_0$, as already established in \eqref{znu 1} and \eqref{znu 2} earlier. Hence, the integral over $C_x^\eps$ is only positive at points with $a\leq u\leq u_0$, yielding that 
 \begin{align}
[z-\hatz,\nu_y]T_1(u-a)_{+}^{p}\lesssim |u_0-a|^p \quad\text{$\mathcal{H}^{d-1}$-a.e.\ on $C_x^\eps$}\label{zu est}
\end{align}
since we also have that $z$ and $\hatz$ are uniformly bounded by \eqref{bound hatz} and \eqref{uni bd z}.
It follows from the Lemma \ref{anz ch} that \begin{align}
 \mel\int_{D_x^\eps} |(z-\hat{z},\mathrm{D}T_1(u-a)^p)_-|\leq p \int_{D_x^\eps} |(z-\hat{z},\mathrm{D}u)_-|\lesssim \int_{D_x^\eps} k(|z-\hatz|)\dy\label{ka part},\end{align}
where we also used Lemma \ref{cont z1} b).


We again take $\delta>0$ fixed.
We can again use the Morse covering theorem and Radon-Nikodym \cite[Thm.\ 1.147 and 1.153]{fonseca2007modern} to find $C_{x_j}^{\eps_j}$ (depending on $\delta$) which are disjoint, cover $U''\cap \{u>u_0\}$ up to a $\mathcal{H}^{d-1}$-zero set and are such that \eqref{del ass} (and therefore also \eqref{del stat}) holds on each one of them.\footnote{This is done with the sole intent of dealing with the $k$-term. In the special case in which $f$ is positively $1$-homogeneous and $k=0$, one can simplify the proof by dropping $\delta$ and taking an arbitrary covering.}

Then we again have, as in \eqref{sum eps}, that\begin{align}
\sum_j\eps_j^{d-1}\lesssim 1.\label{sum eps2}
\end{align}
Summing up and combining \eqref{zu est} and \eqref{ka part}, we see that for each fixed $p$ we have \begin{equation}\begin{aligned}
\mel\sum_j\eps_j^{-2}\int_{D_{x_j}^{\eps_j}}(\lambda (u-h)+g-\div\hat{z}_{x_j})T_{1}(u-a_j)_+^p\dy\\
&\lesssim \sum_j \eps_j^{-2} \int_{C_{x_j}^{\eps_j}}|u_0-a_j|^p\dH +\sum_j  \eps_j^{-2}\int_{D_{x_j}^{\eps_j}} k(|z-\hat{z}_{x_j}|)\dy\label{eff est 2}
\end{aligned}\end{equation}
for numbers $a_j$ which we define as the average values \begin{align}\label{a avg}
a_j:=\strokedint_{C_{x_j}^{\eps_j}} u_0\dH.
\end{align}
We use the same basic approach as in the $BV$-case, that is, showing that the right-hand side of \eqref{eff est 2} vanishes in the limit, while the left-hand side controls the trace.

\subsubsection{Proof that the right-hand side in \eqref{eff est 2} vanishes in the limit} Jensen's inequality lets us estimate the first integral on the right-hand side in \eqref{eff est 2} as \begin{align*}
\mel\eps_j^{-2} \int_{C_{x_j}^{\eps_j}}|u_0-a_j|^p\dH=\eps_j^{-2} \int_{C_{x_j}^{\eps_j}}\left|u_0(y)-\strokedint_{C_{x_j}^{\eps_j}} u_0(y')\dH(y')\right|^p\dH(y)\\
&\leq \eps_j^{-2}\int_{C_{x_j}^{\eps_j}}\strokedint_{C_{x_j}^{\eps_j}} \left|u_0(y)-u_0(y')\right|^p\dH(y)\dH(y').
\end{align*}
Using the property \eqref{cald grad}, with a function $r$ (depending on $\max_j\eps_j$) fulfilling \eqref{cald grad} with $\eps_j$ in place of $\delta$, we see that, since $\alpha p\geq 2$ it holds that  \begin{align}
\mel\eps_j^{-2} \int_{C_{x_j}^{\eps_j}}|u_0-a_j|^p\dH(y)\lesssim \int_{C_{x_j}^{\eps_j}}\strokedint_{C_{x_j}^{\eps_j}} r(y)^p+r(y')^p\dH(y)\dH(y')\nonumber\\
&=2\int_{C_{x_j}^{\eps_j}}r(y)^p\dH(y),\label{est ua}
\end{align}
since $|y-y'|\leq 2\eps_j$.
As we only need \eqref{cald grad} for points with distance $\leq 2\max_j\eps_j$, we see from \eqref{r lim} and summing up that in the limit \begin{align}
\sum_j\eps_j^{-2} \int_{C_{x_j}^{\eps_j}}|u_0-a_j|^p\dH\lesssim \int_{U'} r(y)^p\dH(y)\xrightarrow{\max_j\eps_j\rightarrow 0}0.\label{lim 1st part}
\end{align}
In the other sum in \eqref{eff est 2}, we split the sum into the small and big parts of $|z-\hat{z}_{x_j}|$  to see that \begin{align*}
\sum_j  \eps_j^{-2}\int_{D_{x_j}^{\eps_j}} k(|z-\hat{z}_{x_j}|)\dy\leq \sum_j k(\sigma)\eps_j^{-2}\mathcal{L}^{d}(D_{x_j}^{\eps_j})+\sum_j \eps_j^{-2}\int_{D_{x_j}^{\eps_j}}\mathds{1}_{|z-\hat{z}_{x_j}|\geq \sigma}\dy
\end{align*}
where $\sigma$ is a real number.
By \eqref{est D1}, \eqref{sum eps2} and the fact that $\lim_{t\searrow 0} k(t)=0$ the first summand goes to $0$ if $\sigma\rightarrow 0$.
The second summand is estimated by Chebyshev's inequality and \eqref{del stat} as \begin{align*}
\sum_j \eps_j^{-2}\int_{D_{x_j}^{\eps_j}}\mathds{1}_{|z-\hat{z}_{x_j}|\geq \sigma}\dy\lesssim \sum_j\frac{1}{\sigma^2}\left(\eps_j^{d-1}(\eps_j+\delta)+\int_{D_{x_j}^{\eps_j}}|\div z|\dy\right),
\end{align*}
if we let $\frac{1}{\sigma^2}\rightarrow \infty$ slow enough (depending on $\max_j\eps_j$ and $\delta$), then this goes to $0$ as already established in \eqref{dom conv}, \eqref{dom conv2}.

Together we see that \begin{align*}
\sum_j  \eps_j^{-2}\int_{D_{x_j}^{\eps_j}} k(|z-\hat{z}_{x_j}|)\dy\xrightarrow{\delta,\max_j\eps_j\rightarrow 0}0,
\end{align*}
showing together with \eqref{lim 1st part} that the right hand side of \eqref{eff est 2} vanishes in the limit.

\subsubsection{Proof that the left-hand side of \eqref{eff est 2} controlls the trace}
Again, it holds that $g-\div\hat{z}_{x_j}$ is $\gtrsim 1$ by \eqref{def hatz} and  \eqref{impr cc},  and using the pointwise inequality $cT_1(c)_+^{p}\geq 0$, we see that \begin{align*}
&\sum_j\eps_j^{-2}\int_{D_{x_j}^{\eps_j}}(\lambda (u-h)+g-\div\hat{z}_{x_j})T_1(u-a_j)_+^p\dy\\
&=\sum_j\eps_j^{-2}\int_{D_{x_j}^{\eps_j}}\left(\lambda ((u-a_j)+(a_j-h))+g-\div\hat{z}_{x_j}\right)T_{1}(u-a_j)_{+}^{p}\dy \nonumber\\
&\geq C\sum_j\eps_j^{-2} \int_{D_{x_j}^{\eps_j}}T_{1}(u-a_j)_{+}^{p}\dy-C'\sum_j\eps_j^{-2} \int_{D_{x_j}^{\eps_j}}\lambda(y)|h-a_j|\dy,
\end{align*}
where we have also used that $T_1(u-a)_+\leq  1$.

The first sum of integrals in the last line can be regarded as a single integral on the set $\cup D_{x_j}^{\eps_j}$ with respect to the measure whose density with respect to the Lebesgue measure in $D_{x_j}^{\eps_j}$ is $\eps_j^{-2}$, in particular we can apply Hölder with respect to this measure to see that \begin{align*}
\mel\left(\sum_j\eps_j^{-2}\int_{D_{x_j}^{\eps_j}}T_{1}(u-a_j)_+\dy\right)^p\\ &\lesssim\left(\sum_j \eps_j^{-2}\mathcal{L}^d(D_{x_j}^{\eps_j})\right)^\frac{p^2}{p-1}\times\Bigg(\sum_j\eps_j^{-2}\int_{D_{x_j}^{\eps_j}}\big(\lambda(y) (u-h)+g-\div\hat{z}_{x_j}\big)T_{1}(u-a_j)_{+}^{p}\dy\\
&\quad+\sum_j\eps_j^{-2} \int_{D_{x_j}^{\eps_j}}\lambda(y)|h-a_j|\dy\Bigg).
\end{align*}
Here, the first sum can be estimated as \begin{align*}
\sum_j \eps_j^{-2}\mathcal{L}^d(D_{x_j}^{\eps_j})\lesssim \sum_j \eps_j^{d-1}\lesssim 1,
\end{align*}
by \eqref{est D1} and \eqref{sum eps2}, and hence this factor may be disregarded.

Together, we have obtained \begin{align*}
\limsup_{\delta,\max_j\eps_j\rightarrow 0}\sum_j\eps_j^{-2}\int_{D_{x_j}^{\eps_j}}T_{1}(u-a_j)_+\dy\lesssim&\limsup_{\delta,\max_j\eps_j\rightarrow 0} \left(\sum_j \eps_j^{-2} \int_{D_{x_j}^{\eps_j}}\lambda(y)|h-a_j|\dy\right)^\frac{1}{p}.
\end{align*} 
We may then apply the trace Lemma \ref{trace lemma} for each $\rho\in (0,1)$, together with the fact that the total variation of $T_{1}(u-a_j)_+$ is less of equal than the one of $u$ by e.g.\ the coarea formula to conclude in the same way as in the $BV$-case that for every fixed $\rho\in (0,1)$ it holds that \begin{align*}
\limsup_{\delta,\,\max_j\eps_j\rightarrow 0}\sum_j \int_{C_{x_j}^{\rho\eps_j}}T_{1}(u-a_j)_+\dH\lesssim_{\rho} \limsup_{\delta,\,\max_j\eps_j\rightarrow 0}\left(\sum_j\eps_j^{-2} \int_{D_{x_j}^{\eps_j}}\lambda(y)|h-a_j|\dy\right)^{\frac{1}{p}}.
\end{align*}
%
%
%
Regarding the right-hand side, using Lemma \ref{trace lemma} and arguing as in the $BV$-case, we have \begin{align*}
\limsup_{\max_j\eps_j,\delta\rightarrow 0} \sum_j\eps_j^{-2} \int_{D_{x_j}^{\eps_j}}\lambda|h-a_j|\dy\lesssim \limsup_{\max_j\eps_j,\delta\rightarrow 0}\norm{\lambda}_{L^\infty}\sum_j\int_{C_{x_j}^{\eps_j}}|u_0-a_j|\dH,
\end{align*}
but we also have, by arguing as for \eqref{est ua}, that\begin{align*}
\sum_j\norm{\lambda}_{L^\infty}\int_{C_{x_j}^{\eps_j}}|u_0-a_j|\dH\lesssim \norm{\lambda}_{L^\infty}\max_j\eps_j^\alpha\int_{U'} r(y)\dH(y)\rightarrow 0,
\end{align*}
yielding \begin{align*}
\lim_{\max_j\eps_j\rightarrow 0}\sum_j\eps_j^{-2} \int_{D_{x_j}^{\eps_j}}\lambda(y)|h-a_j|\dy=0,
\end{align*}
and \begin{align}
\lim_{\delta,\,\max_j\eps_j\rightarrow 0}\sum_j \int_{C_{x_j}^{\rho\eps_j}}T_{1}(u-a_j)_+\dH= 0.\label{lim est2}
\end{align}
 It remains to estimate $\mathcal{H}^{d-1}(\{u> u_0\})$ with this sum. Observe that by Chebyshev and the triangle inequality we have, for any $a$, \begin{equation}\begin{aligned}
\mel\mathcal{H}^{d-1}(\{y\in C_x^\eps\,\big|\,u(y)\geq a+1\})= \mathcal{H}^{d-1}(\{y\in C_x^\eps\,\big|\,T_{1}(u(y)-a)_+\geq 1\})\\
&\leq \int_{C_x^\eps}T_{1}(u(y)-a)_+\dH.\label{tru est}
\end{aligned}\end{equation}
On the other hand, we also have \begin{align*}
(u(y)-u_0(y))_+\leq T_{1}(u(y)-a)_++|u_0(y)-a| \quad \text{if $u(y)\leq a+1$}
\end{align*}
hence, by Chebyshev \begin{equation}\begin{aligned}
&\mathcal{H}^{d-1}(\{y\in C_x^\eps\,\big|\,u(y)\leq a+1,\, u(y)-u_0(y)\geq c \})\leq\frac{1}{c}\int_{C_x^\eps} (T_{1}(u)-a)_++|u_0-a|\dH\label{tru est2}
\end{aligned}\end{equation}
for each $c> 0$.

Similarly as for  \eqref{est ua} above, we have that \begin{align}
\sum_j \int_{C_{x_j}^{\eps_j}} \left|u_0(y)-a_j\right|\dH(y)\lesssim \max_j\eps_j^\alpha\int_{U'} r\dH\rightarrow 0.\label{u0a}
\end{align}
Hence, using \eqref{tru est}, \eqref{tru est2} and \eqref{u0a}, we see that \begin{align*}
\mel\mathcal{H}^{d-1}\left(U''\cap\{u-u_0\geq c\}\cap \bigcup_j C_{x_j}^{\rho\eps_j}\right)\\
&\leq \sum_j \mathcal{H}^{d-1}\left(\{y\in C_{x_j}^{\rho\eps_j}\,\big|\,u(y)\geq a_j+1\}\right)+\mathcal{H}^{d-1}\left(\{y\in C_{x_j}^{\rho\eps_j}\,\big|\,u(y)\leq a_j+1,\, u(y)-u_0(y)\geq c \}\right)\\
& \leq\frac{1}{c} \sum_j \int_{C_{x_j}^{\rho\eps_j}}(T_{1}(u(y)-a_j)_+\dH+\frac{1}{c}\max_j\eps_j^\alpha\int_{U'} r\dH.
\end{align*}
This converges to $0$ for every fixed $c>0$ by \eqref{lim est2}. Since we have \begin{align*}
\mathcal{H}^{d-1}\left(\bigcup_j(C_{x_j}^{\eps_j}\backslash C_{x_j}^{\rho\eps_j})\right)\xrightarrow{\rho\nearrow 1} 0,
\end{align*}
as in \eqref{est def} and because the $C_{x_j}^{\eps_j}$ cover $U''\cap \{u>u_0\}$ up to a $\mathcal{H}^{d-1}$-zero set, we conclude the theorem in this case.\hfill\qedsymbol

\subsection{Proof of Thm.\ \ref{T1} in the case \texorpdfstring{$u_0\in C^0$}{u0C0}}
Yet again, it suffices to consider the positive part of $u-u_0$.

We start with \eqref{part int 2} (which did not use the regularity of $u_0$ at all and therefore still holds here), where we again take \begin{align*}
a=\strokedint_{C_x^\eps} u_0\dH.
\end{align*}
Then, by the extra assumption in \ref{A6} that in this case either $\lambda=0$ or $h$ is a continuous extension of $u_0$ (near $U$) we have $\norm{\lambda(a-h)}_{L^\infty(D_x^\eps)}\rightarrow 0$ as $\eps\searrow 0$ yielding, together with \eqref{impr cc}, that in $D_x^\eps$ it holds that \begin{align*}
\mel(\lambda (u-h)+g-\div \hatz)T_1(u-a)_+^{p}\\
&\geq \left(\lambda ((u-a)+(a-h))+g-\div \hatz\right)T_1(u-a)_+^{p}\gtrsim T_1(u-a)_+^{p},
\end{align*}
uniformly in $p$, where we have again used that $(u-a)T_1(u-a)_+\geq 0$.
We now take the limit $p\rightarrow \infty$ of the $p$-th root of \eqref{part int 2}, since the limit of the $L^p$-norm is the $L^\infty$-norm, we see that \begin{align*}
\mel\esssup_{y\in C_x^\eps} T_1(u(y)-a)_+\leq\esssup_{y\in D_x^\eps} T_1(u(y)-a)_+\\
&\lesssim \limsup_{p\rightarrow \infty}\left( \int_{D_x^\eps}(\lambda (u-h)+g-\div \hatz)T_1(u-a)_+^p\dy\right)^\frac{1}{p}\\
&\lesssim \limsup_{p\rightarrow \infty}\left(\int_{C_x^\eps} [z-\hatz,\nu]T_1(u-a)_+^p\dH\right)^\frac{1}{p}+\left(\int_{D_x^\eps}|(z-\hatz,\mathrm{D}(T_1(u-a)_+^p))_-|\right)^\frac{1}{p}\nonumber\\
&\lesssim\sup_{y\in C_x^\eps} T_1(u_0-a)_++\limsup_{p\rightarrow \infty}\left(\int_{D_x^\eps}|(z-\hatz,\mathrm{D}T_1(u-a)_+^p)_-|\right)^\frac{1}{p}.
\end{align*}
Using Lemma \ref{anz ch} and that we can take $k=0$ in Lemma \ref{cont z1} b), due to the extra assumption of positive $1$-homogeneity of $f$ in \ref{A6}, we see that the second summand vanishes.

%
Now if we let $\eps\rightarrow 0$, then the right-hand side goes to $0$ by continuity, while $a$ converges to $u_0(x)$ by the continuity of $u_0$, showing that $u(x)\leq u_0(x)$ whenever $x$ is a Lebesgue point of $u|_{\de\Omega}$.\hfill\qedsymbol



\subsection{Proof of Theorem \ref{T2}}
Throughout the entire proof, we assume that \ref{B1}-\ref{B3} hold, in particular, the lemmata might require them even though we don't explicitly mention it.

We use a similar scheme as in the $BV$ case for Theorem \ref{T1}, but without transforming the domain. 
As the statement of the theorem is local, we again only need to show that for every point $x^0\in U$, we can find a neighborhood where $u=u_0$.

We shall consider $U'\subset U$ of the form $U'=\de\Omega\cap B$ for some (open) ball $B$.
 We can also assume that $\Omega\cap B$ is convex by choosing $B$ sufficiently small, thanks to the Assumption \ref{B1} on $U$.

Since $U$ is uniformly convex and $C^2$, we can also assume that $U'$ is uniformly convex in the sense that it can be written as the graph of a uniformly convex $C^2$ function. It is also not restrictive to assume that such a graph is taken over the plane perpendicular to $\nu_{x_0}$ for some $x_0\in U'$ after potentially making $U'$ smaller.

We use the following adaptions of the sets for $x\in U'$ and small $\eps$: \begin{align}
&\tilde{E}_x^\eps:=\{y\in \overline{\Omega}\cap B\,\big|\, \scalar{y-x}{\nu_x}=-\eps^2\}\\
&\tilde{D}_x^\eps:=\{y\in \Omega\cap B\,\big|\, \scalar{y-x}{\nu_x}>-\eps^2\}\\
&\tilde{C}_x^\eps=\de\tilde{D}_x^\eps\backslash\tilde{E}_x^\eps=\{y\in\de \Omega\cap B\,\big|\, \scalar{y-x}{\nu_x}>-\eps^2\}.
\end{align}
Figure \ref{figure2} shows a sketch. 
 We have the following geometric properties.
 
\begin{figure}
\includegraphics[width=5cm]{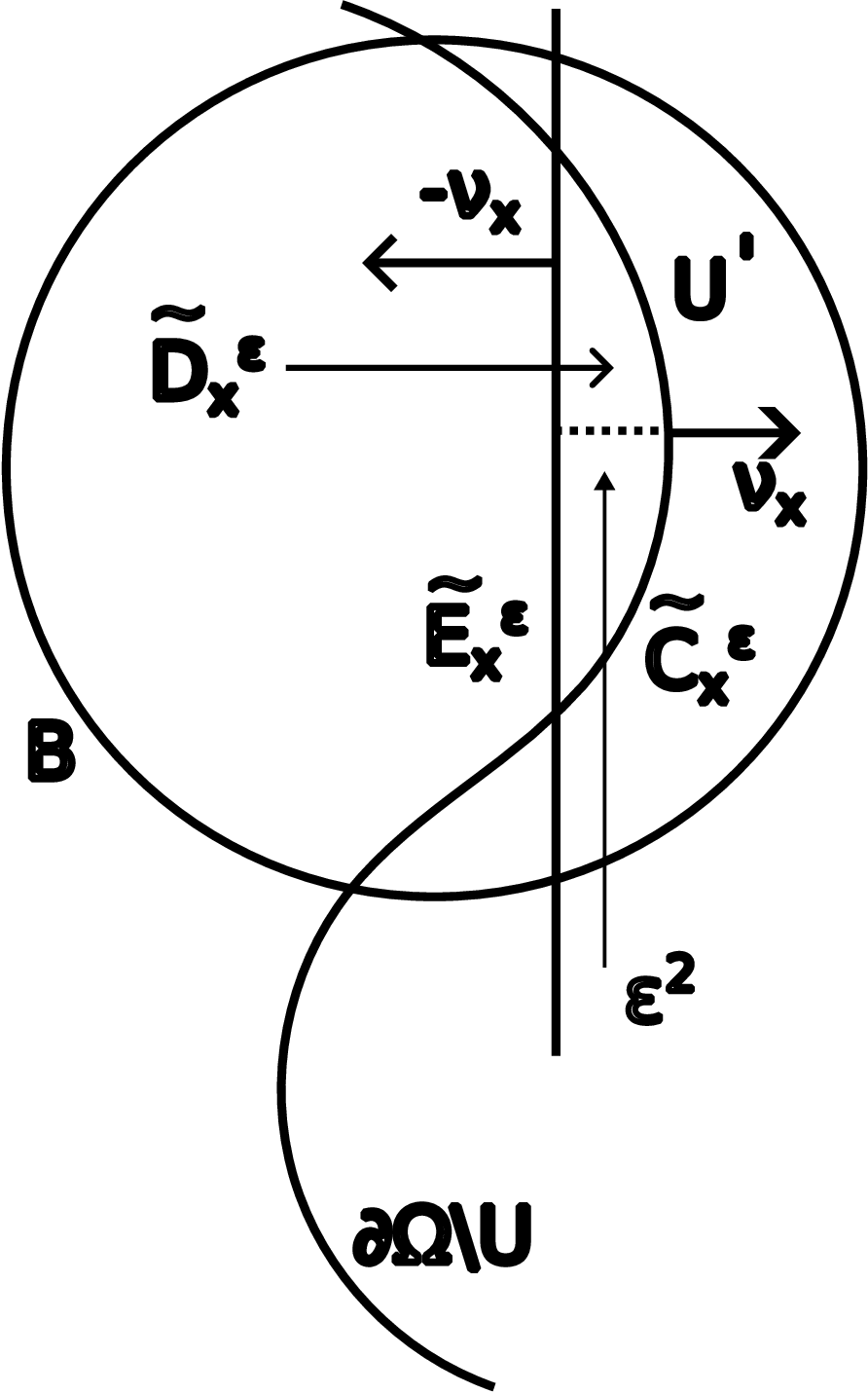} 
\caption{Sketch of the geometry in the proof of Theorem \ref{T2}}\label{figure2}
\end{figure}

\begin{lemma}\label{geo lemma 2}
For every $x^0\in U'$, we can find an open $U''\subset U$ with $x^0\in U''$, such that $U''$ has a positive distance to the relative boundary $\de U'$ and there is a constant $C>0$ such that we have the following geometric properties for all $x\in U''$ and $\eps<C\min(\dist(U'',\de B), \dist(U'',\de\Omega\backslash U'))$:

\noindent \textbf{a)} It holds that $\tilde{C}_x^\eps\cup \tilde{D}_x^\eps\cup \tilde{E}_x^\eps\subset B_{C'\eps}(x)$ for some constant $C'$, independent of $x$ and $\eps$, and we have the bounds \begin{align}
&\mathcal{H}^{d-1}(\tilde{C}_x^\eps)\approx \eps^{d-1}\label{est C2}\\
&\mathcal{L}^{d}(\tilde{D}_x^\eps)\approx \eps^{d+1}\label{est D2}\\
&\mathcal{H}^{d-1}(\tilde{E}_x^\eps)\approx \eps^{d-1}\label{est E2}
\end{align}
and $\mathcal{H}^{d-1}(\tilde{C}_x^\eps)$ is Lipschitz in $\eps$ with \begin{align}
\frac{\mathrm{d}}{\mathrm{d}\eps}\mathcal{H}^{d-1}(\tilde{C}_x^\eps)\approx \eps^{d-2}\label{C lip}
\end{align}
at all $\eps$ at which the derivative exists, all uniformly in $x\in U''$ and small $\eps$.

\noindent\textbf{b)} The boundary $\de \tilde{D}_x^\eps$ is Lipschitz.

\noindent\textbf{c)} If $\tilde{D}_x^\eps$ and $\tilde{D}_{x'}^{\eps'}$ intersect, then so do $\tilde{C}_x^\eps$ and $\tilde{C}_{x'}^{\eps'}$.

\noindent\textbf{d)} Let $\mathfrak{pr}$ denote the orthogonal projection onto the plane perpendicular to $\nu_{x^0}$, then each of the sets $\mathfrak{pr}(\tilde{C}_x^\eps)$ is convex and there is constant $C$, not depending on $x$ or $\eps$ such that \begin{align}
B_{\frac{1}{C}\eps}(\mathfrak{pr}(x))\subset \mathfrak{pr}(\tilde{C}_x^\eps)\subset B_{C\eps}(\mathfrak{pr}(x)).\label{morse prereq}
\end{align}

\noindent\textbf{e)} Let $\rho\in (0,1)$, then it holds that \begin{align}
\eps\lesssim_\rho|\nu_x-\nu_y|\lesssim \eps \quad\text{in $\tilde{C}_x^\eps\backslash \tilde{C}_x^{\rho\eps}$,}\label{normal change}
\end{align}
uniformly in $x$ and $\eps$.
\end{lemma}
The lemma is proven in Section \ref{S52}.

During the rest of the proof, we fix such an $U''$ and show that $u=u_0$ holds in $U''$, which then yields the theorem because $U$ can be covered with such sets.
We will only consider $\eps<<\min(\dist(U'',\de B), \dist(U'',\de\Omega\backslash U'))$ so that $\tilde{C}_x^\eps\subset U'$, there is no intersection between $\tilde{D}_x^\eps$ and $\de\Omega\backslash U'$, and the statements of the Lemma hold.

\smallskip



We use the functions $\mu_1,\mu_2$ from the Assumption \eqref{r1 split} on $f$ and define for $x\in U'$ \begin{align*}
\zeta_x:=\mu_2(\nu_x)
\end{align*}
Observe that, by the definition \eqref{r1 split} of $\mu_1$ and $\mu_2$, the property \eqref{doal} and the growth condition \eqref{gc finf}, it holds that \begin{align}
\scalar{\zeta_x}{\nu_x}=\frac{1}{\scalar{\mu_1(a)}{a}}f^\infty(a\otimes \nu_x)\gtrsim 1,\label{pos scalar}
\end{align}
where $a$ is arbitrary.
Therefore, by the implicit function theorem, we can write every point in $\Omega$ in a neighborhood of $x$ uniquely in the form \begin{align}
y=y_0^x+t\zeta_x,\label{imp norm}
\end{align}
where $t\leq 0$ and $y_0^x\in \de\Omega$ is some point near $x$. Since all possible values of $\zeta_x$ lie in a compact set, it is possible to pick $\eps$ so small that for every $x\in U''$, the set $\tilde{D}_x^\eps$ lies in a neighborhood of the boundary where such a unique representation is possible.
It follows from \eqref{pos scalar} and the definition of the sets that for $y\in \tilde{D}_x^\eps$, the point $y_0^x$ must lie in $\tilde{C}_x^\eps$.

We then set \begin{align*}
\bar{u}_x(y):=u_0(y_0^x).
\end{align*}

\begin{lemma}\label{char2}
The function $\bar{u}_x$ lies in $(BV\cap L^2)(\tilde{D}_x^\eps,\R^n)$, and it holds that \begin{align}
&a\otimes\zeta_x: \mathrm{D}\bar{u}_x=0\label{van}\\
&|\mathrm{D}\bar{u}_x|(\tilde{D}_x^\eps)\lesssim \eps^2|\mathrm{D}u_0|(\tilde{C}_x^\eps)\label{est baru1}\\
&|\mathrm{D}^s\bar{u}_x|(\tilde{D}_x^\eps)\lesssim \eps^2|\mathrm{D}^su_0|(\tilde{C}_x^\eps)\label{est baru2}
\end{align}
for every $a\in \R^n$.
Furthermore, there is a non-decreasing function $i:\R_{\geq 0}\rightarrow \R_{\geq 0}$ depending on $u_0$, but not on $\eps$ or $x$, such that
\begin{align*}
\lim_{t\rightarrow +\infty} \frac{i(t)}{t}=+\infty,
\end{align*}
it holds that \begin{align*}
\int_{U'} i(|\mathrm{D}^au_0|(y))\dH(y)<\infty
\end{align*}
and \begin{align*}
\int_{\tilde{D}_x^\eps} i(|\mathrm{D}^a\bar{u}_x|(y))\dy\lesssim_{u_0} \eps^2\int_{\tilde{C}_x^\eps} i(|\mathrm{D}^au_0|(y))\dH(y),
\end{align*}
uniformly in $x\in U''$ and $\eps>0$ small.
\end{lemma}
The lemma is proven in Section \ref{S53}.

To define the comparison tensor field $\bar{z}_x$, we first note that every point $y$ in $\tilde{D}_x^\eps$ can also uniquely be written as $y_1^x+s'\zeta_x$  for $y_1^x\in \tilde{E}_x^\eps$. Indeed this representation is unique as $\zeta_x$ is not tangential to $\tilde{E}_x^\eps$ by \eqref{pos scalar} and every point can be written in this way because the ray $y+s\zeta_x$ with $s\leq 0$ must intersect the boundary of $\tilde{D}_x^\eps$, yet can not hit $\tilde{C}_x^\eps$ because of the uniqueness of the representation \eqref{imp norm}.

We then set  \begin{align}
\bar{z}_x(y):=-\mu_1\left(u(y_1^x)-\bar{u}_x(y_1^x)\right)\otimes \zeta_x,\label{def barz}
\end{align}
where we set $\mu_1(0)=0$. If $u(y_1^x)-\bar{u}_x(y_1^x)\neq 0$, then by the definition \eqref{r1 split} of $\mu_1$ and $\mu_2$ it holds that \begin{align*}
\bar{z}_x(y)=-\mathrm{D}_\xi f^\infty\big((u(y_1^x)-\bar{u}_x(y_1^x))\otimes \nu_x\big).
\end{align*}
This field is (row-wise) divergence-free, since every row is a multiple of $\zeta_x$ and the field is constant along lines with direction $\zeta_x$, and it is furthermore in the class $X_2(\tilde{D}_x^\eps,\R^{n\times d})$.

By the following lemma, the normal trace of $\bar{z}_x$ is the one that one would expect from the definition. \begin{lemma}
Let $a\in L^\infty(\tilde{E}_x^\eps)$, then for the function $a(y_1^x)\otimes \zeta_x\in X_2(\tilde{D}_x^\eps,\R^{n\times d})$ it holds that \begin{align}
&[a(y_1^x)\otimes \zeta_x,\nu_{\tilde{E}_x^\eps}](y)=-a(y)\scalar{\zeta_x}{\nu_x}\quad \text{ on $\tilde{E}_x^\eps$}\label{nt1}\\
&[a(y_1^x)\otimes \zeta_x,\nu](y)=a(y_1^x)\scalar{\nu_y}{\zeta_x}\quad \text{ on $\tilde{C}_x^\eps$}\label{nt2}
\end{align}
where $\nu_{\tilde{E}_x^\eps}=-\nu_x$ is the outer normal on $\tilde{E}_x^\eps$.
\end{lemma}
\begin{proof}
For $a\in C^1$, this is trivial because the normal trace is then just the pointwise normal trace and $y=y_1^x$ on $\tilde{E}_x^\eps$. For $a\notin C^1$, we can use an approximating sequence $a_m\rightarrow a$ converging weakly\textsuperscript{$*$} in $L^\infty$ to $a$. Then the function $a_m(y_1^x)\otimes \zeta_x$ is trivially divergence-free and hence for every $\phi\in C^1(\tilde{D}_x^\eps,\R^n)$ we have \begin{align*}
\mel\int_{\tilde{C}_x^\eps\cup \tilde{E}_x^\eps}\scalar{[(a_m-a)(y_1^x)\otimes \zeta_x,\nu]}{\phi}\dH(y)=\int_{\tilde{D}_x^\eps}(a_m-a)(y_1^x)\otimes \zeta_x:\mathrm{D}\phi\dy\\
&=\int_{\tilde{D}_x^\eps}\scalar{(a_m-a)(y_1^x)}{\de_{\zeta_x}\phi}\dy\rightarrow 0,
\end{align*}
where the convergence easily follows from Fubini.
\end{proof}

We now fix a $z$ fulfilling the conditions in Lemma \ref{vec char} (with $v=0$, as it is a minimizer).

We then use the same approach as in the scalar-valued case and see from the Gauss-Green formula \eqref{gg form} that \begin{align}
\int_{\tilde{C}_x^\eps} [z-\bar{z}_x,\nu]\cdot(u-u_0)\dH+\int_{\tilde{E}_x^\eps} [z-\bar{z}_x,\nu_{\tilde{E}_x^\eps}]\cdot(u-\bar{u}_x)\dH= \int_{\tilde{D}_x^\eps} (z-\bar{z}_x,\mathrm{D}u-\mathrm{D}\bar{u}_x).\label{par int}
\end{align}
By \eqref{nt1}, the definition of $\bar{z}_x$, and \eqref{doal} we know that on $\tilde{E}_x^\eps$ it holds that \begin{align*}
[\bar{z}_x,\nu_{\tilde{E}_x^\eps}]\cdot(u-\bar{u}_x)=\big(\mu_1\big(u(y_1^x)-\bar{u}_x(y_1^x)\big)\otimes \mu_2(\nu_x)\big):\big((u-\bar{u}_x)\otimes \nu_x\big)=f^\infty((u-\bar{u}_x)\otimes \nu_x)
\end{align*}
since $y=y_1^x$ on $\tilde{E}_x^\eps.$
Together with \eqref{z tr pw}, \eqref{cv4} and \eqref{doal} this shows that \begin{align*}
[z-\bar{z}_x,\nu_{\tilde{E}_x^\eps}]\cdot(u-\bar{u}_x)\leq 0 \quad \text{ on $\tilde{E}_x^\eps$}.
\end{align*}
Hence, using \eqref{par int}, we see that \begin{align}
-\int_{\tilde{C}_x^\eps} [z-\bar{z}_x,\nu]\cdot(u-u_0)\dH(y)\leq \int_{\tilde{D}_x^\eps} (z-\bar{z}_x,\mathrm{D}\bar{u}_x)-(z-\bar{z}_x,\mathrm{D}u).\label{dadada}
\end{align}
It follows from the definition of $\bar{z}_x$, the equation \eqref{van} and the formulas for the density of the Anzelotti pairing in Section \ref{S23} that \begin{align*}
(\bar{z}_x,\mathrm{D}\bar{u}_x)=0.
\end{align*}

The analogue of Lemma \ref{cont z1} holds:

\begin{lemma}\label{cont z2}
\textbf{a)} For every $\delta>0$, every $v\in S^{d-1}$, every $x\in U''$ and $\eps>0$, we have the following conditional estimate: \textbf{If} we have \begin{align}
\mathcal{H}^{d-1}\left(\{y\in \tilde{C}_x^\eps\,\big|\,[z,\nu](y)\notin \scalar{\zeta_y}{\nu_y}B_\delta(\mu_1(v))\}\right)\leq \delta \mathcal{H}^{d-1}(\tilde{C}_x^\eps)\label{del ass2}
\end{align}
 \textbf{then} \begin{align}
\int_{\tilde{D}_x^\eps} |z(y)-\mu_1(v)\otimes \zeta_x|^2\dy\lesssim \eps^{d+1}(\delta+\eps).\label{del stat2}
\end{align}
\textbf{b)} There is a non-decreasing and bounded function $k:\R_{\geq 0}\rightarrow \R_{\geq 0}$, depending on $f$, but not on the other quantities, such that $\lim_{t\rightarrow 0} k(t)=0$ and \begin{align*}
\int_{\tilde{D}_x^\eps} \left|(z-\bar{z}_x,\mathrm{D}u)_-\right|\lesssim \int_{\tilde{D}_x^\eps}k(|z(y)-\mu_1(v)\otimes \zeta_x|)\dy,
\end{align*}
holds for every $v\neq 0$, where the "$-$" denotes the negative part of the measure.
\end{lemma}
The proof can be found in Sections \ref{S54} and \ref{S55}. 

We note that if $u(y)\neq u_0(y)$ at some $y\in \de \Omega$, then, if we plug \eqref{r1 split} into \eqref{good bd cond}, it holds that $[z,\nu](y)= \scalar{\zeta_y}{\nu_y}\mu_1(v)$ with  $v=\frac{u_0(y)- u(y)}{|u_0(y)- u(y)|}$. Since we also have $\scalar{\zeta_y}{\nu_y}\gtrsim 1$ by \eqref{pos scalar}, we see from Radon-Nikodym and the Lebesgue point theorem that for $\mathcal{H}^{d-1}$-a.e.\ $x$ in the set $\{u\neq u_0\}$, the condition \eqref{del ass2} holds for every sufficiently small $\eps$ (depending on $\delta$ and $x$).

Now we fix $\delta>0$ and use a covering argument as in the scalar case.
We can again find $\eps_{0,x}>0$ such that for $\mathcal{H}^{d-1}$-a.e.\ $x\in U''\cap \{u\neq u_0\}$ the property \eqref{del ass2} holds for $\eps<\eps_{0,x}$ and a $v$ depending on $\eps$, as explained in the previous paragraph, and furthermore such that $|\mathrm{D}^su_0|(\tilde{C}_x^\eps)\leq \delta \mathcal{H}^{d-1}(\tilde{C}_x^\eps)$. Indeed, this follows from Radon-Nikodym.

We can apply the Morse covering theorem \cite[Thm.\ 1.147]{fonseca2007modern} to the projections of the $\tilde{C}_x^\eps$ onto the plane perpendicular to $\nu_{x^0}$  since the assumptions of the theorem are true by Lemma \ref{geo lemma 2} d). Since $U'$ can be written as a graph over the plane, we have that two different $\tilde{C}_x^\eps$ are disjoint if their projections are disjoint. This yields that we can find again find a (countable) cover of $U''\cap \{u\neq u_0\}$, up to a $\mathcal{H}^{d-1}$-zero set with disjoint $\tilde{C}_{x_j}^{\eps_j}$ (depending on $\delta$) such that \eqref{del ass2} holds for each of these and such that \begin{align*}
|\mathrm{D}^su_0|(\tilde{C}_{x_j}^{\eps_j})\leq \delta \mathcal{H}^{d-1}(\tilde{C}_{x_j}^{\eps_j}).
\end{align*}
The analogue of \eqref{sum eps} does still hold here for the same reason, i.e.\ \begin{align}
\sum_j \eps_j^{d-1}\approx \sum_j \mathcal{H}^{d-1}(\tilde{C}_{x_j}^{\eps_j})\leq \mathcal{H}^{d-1}(U')\lesssim 1.\label{sum eps3}
\end{align}

\noindent We hence see that as in \eqref{eff est1}, \eqref{c1} it follows from \eqref{dadada} that \begin{align*}
\mel\sum_j-\eps_j^{-2}\int_{\tilde{C}_x^\eps} [z-\bar{z}_{x_j},\nu]\cdot(u-u_0)\dH(y)\leq\sum_j\eps_j^{-2}\left(\int_{\tilde{D}_{x_j}^{\eps_j}} (z,\mathrm{D}\bar{u}_{x_j})+|(z-\bar{z}_{x_j},\mathrm{D}u)_-|\right)\\
&\lesssim \sum_j \eps_j^{-2}\int_{\tilde{D}_{x_j}^{\eps_j}}|z-\mu_1(v_j)\otimes\zeta_{x_j}||\mathrm{D}^a \bar{u}_{x_j}|\dy\\
&\quad+\eps_j^{-2}|\mathrm{D}^s\bar{u}_{x_j}|(\tilde{D}_{x_j}^{\eps_j})+\eps_j^{-2}
\int_{\tilde{D}_{x_j}^{\eps_j}}k(|z-\mu_1(v_j)\otimes\zeta_{x_j}|)\dy,
\end{align*}
where we have used \eqref{van} to bring in the $\mu_1(v_j)\otimes\zeta_{x_j}$.

The same arguments as in the proof of Claim \ref{claim1} in the scalar case earlier show that this goes to $0$ as $\delta,\max_j\eps_j\searrow 0$, in fact, the argument even simplifies quite a bit due to the vanishing divergence of $z$.

It remains to control the trace with the left-hand side.

\subsubsection*{Proof that the left-hand side controls the trace}
Let us analyze the integrand on the left-hand side.  We distinguish the cases $u(y_1^x)-\bar{u}_x(y_1^x)=0$ and $u(y_1^x)-\bar{u}_x(y_1^x)\neq 0$.
At points $y\in \tilde{C}_x^\eps$ with  $u(y_1^x)-\bar{u}_x(y_1^x)=0$ it holds that $[\bar{z}_x,\nu](y)=0$ by \eqref{nt2}. Then it follows from \eqref{cv5} and \eqref{gc finf} that \begin{align}
-\left[z-\bar{z}_x,\nu\right]\cdot(u-u_0)=f^\infty((u-u_0)\otimes \nu_y)\gtrsim |u-u_0|.\label{str bd}
\end{align}
In the other case, it either holds that $u=u_0$ or, as one sees from \eqref{good bd cond} and the definition \eqref{r1 split} of $\mu_1,\mu_2$, it holds that \begin{equation*}[z,\nu]\cdot (u-u_0)(y)=-\left(\mu_1(u-u_0)\otimes \mu_2(\nu_y)\right):((u-u_0)\otimes \nu_y)\end{equation*} as $\mu_1$ is odd by definition.

Combining this with  \eqref{nt2} and the Definition \eqref{def barz} of $\bar{z}_x$, we see, using the oddness of $\mu_1$ again, that\begin{align*}
\mel -\left[z-\bar{z}_x,\nu\right]\cdot(u-u_0)(y)\\
&=\scalar{\mu_1(u-u_0)\otimes \mu_2(\nu_y)-\mu_1\left(u(y_1^x)-\bar{u}_x(y_1^x)\right)\otimes \mu_2(\nu_x)}{(u-u_0)\otimes \nu_y}\\
&\gtrsim|(u-u_0)\otimes \nu_y|\left|\mu_1(u-u_0)\otimes \mu_2(\nu_y)-\mu_1\left(u(y_1^x)-\bar{u}_x(y_1^x)\right)\otimes \mu_2(\nu_x)\right|^2,
\end{align*}
where the last step used Lemma \ref{fen lemma} and the positive $0$-homogeneity of $\mu_1$ and $\mu_2$.
Again, using the Definition \eqref{r1 split} of $\mu_1$ and $\mu_2$, we see that \begin{align}
\mel\left|\mu_1(u-u_0)\otimes \mu_2(\nu_y)-\mu_1\left(u(y_1^x)-\bar{u}_x(y_1^x)\right)\otimes \mu_2(\nu_x)\right|\nonumber\\
&=\left|\mu_1\left(\frac{u-u_0}{f^\infty((u-u_0)\otimes \nu_y)}\right)\otimes \mu_2(\nu_y)-\mu_1\left(\frac{u(y_1^x)-\bar{u}_x(y_1^x)}{f^\infty((u(y_1^x)-\bar{u}_x(y_1^x))\otimes \nu_x)}\right)\otimes \mu_2(\nu_x)\right|\nonumber\\
&=\frac{1}{2}\left|\mathrm{D}_\xi( (f^\infty)^2)\left(\frac{(u-u_0)\otimes \nu_y}{f^\infty((u-u_0)\otimes \nu_y)}\right)-\mathrm{D}_\xi ( (f^\infty)^2)\left(\frac{(u(y_1^x)-\bar{u}_x(y_1^x))\otimes \nu_x}{f^\infty((u(y_1^x)-\bar{u}_x(y_1^x))\otimes \nu_x)}\right)\right|.\label{rewrite grad}
\end{align}
Since $(f^\infty)^2$ is uniformly convex by assumption, we know that for all $\xi_1,\xi_2\in \R^{n\times d}\backslash \{0\}$ it holds that \begin{align*}
\mel\left|\mathrm{D}_\xi( (f^\infty)^2)(\xi_1)-\mathrm{D}_\xi( (f^\infty)^2)(\xi_2)\right|\\
&\geq \frac{1}{|\xi_1-\xi_2|}\scalar{\mathrm{D}_\xi( (f^\infty)^2)(\xi_1)-\mathrm{D}_\xi( (f^\infty)^2)(\xi_2)}{\xi_1-\xi_2}\\
&\gtrsim |\xi_1-\xi_2|.
\end{align*}
Combining this estimate with \eqref{rewrite grad}, we see that 
 \begin{align*}
\mel\big|\mu_1(u-u_0)\otimes \mu_2(\nu_y)-\mu_1\left(u(y_1^x)-\bar{u}_x(y_1^x)\right)\otimes \mu_2(\nu_x)\big|\\
&\gtrsim \left|\frac{(u-u_0)\otimes \nu_y}{f^\infty((u-u_0)\otimes \nu_y)}-\frac{(u(y_1^x)-\bar{u}_x(y_1^x))\otimes \nu_x}{f^\infty((u(y_1^x)-\bar{u}_x(y_1^x))\otimes \nu_x)}\right|.
\end{align*}
Using the growth bound \eqref{gc finf} for $f^\infty$ and that it holds $|a\otimes v_1-b\otimes v_2|\gtrsim \min(|a|,|b|)\min(|v_1-v_2|,|v_1+v_2|)$ for all unit vectors $v_1$ and $v_2$, we see that \begin{align*}
\left|\frac{(u-u_0)\otimes \nu_y}{f^\infty((u-u_0)\otimes \nu_y)}-\frac{(u(y_1^x)-\bar{u}_x(y_1^x))\otimes \nu_x}{f^\infty((u(y_1^x)-\bar{u}_x(y_1^x))\otimes \nu_x)}\right|\gtrsim \min(|\nu_x-\nu_y|,\,|\nu_x+\nu_y|).
\end{align*}  
In sum, we have obtained that \begin{align*}
-\left[z-\bar{z}_x,\nu\right]\cdot(u-u_0)\gtrsim |u-u_0| \min(|\nu_x-\nu_y|,\,|\nu_x+\nu_y|)^2.
\end{align*} 
Using that $|\nu_x-\nu_y|<<1$ as $|x-y|<<1$ and that the normal is continuous by the assumed regularity of the boundary, we see that $|\nu_x+\nu_y|\geq |\nu_x-\nu_y|$ for small $\eps$ and hence, thanks to Lemma \ref{geo lemma 2} e), we see that this is \begin{align}
\gtrsim_\rho\eps^2 |u-u_0|\quad \text{ in $\tilde{C}_x^\eps\backslash \tilde{C}_x^{\rho\eps}$}\label{norm est }
\end{align}
uniformly in $x$ and $\eps$ for $\rho\in (0,1)$ for $u(y_1^x)-\bar{u}_x(y_1^x)\neq 0$. This is also true in the case $u(y_1^x)-\bar{u}_x(y_1^x)= 0$  from earlier, as \eqref{str bd} is clearly much stronger and hence \eqref{norm est } holds on all of $\tilde{C}_x^\eps\backslash \tilde{C}_x^{\rho\eps}$.
This yields that \begin{align}
\limsup_{\delta,\max_j\eps_j\searrow 0} \int_{\bigcup_j\tilde{C}_x^\eps\backslash \tilde{C}_x^{\rho\eps}}|u-u_0|\dH=0.\label{lim est 3}
\end{align}
%
%
%
%
We note that, thanks to the disjointness of the sets, it holds that \begin{align*}
\mathcal{H}^{d-1}\left(\bigcup_j \tilde{C}_{x_j}^{\rho\eps_j}\right)\leq \sum_j\mathcal{H}^{d-1}(\tilde{C}_{x_j}^{\rho\eps_j})\approx \rho^{d-1}\sum_j\eps_j^{d-1}\lesssim \rho^{d-1},
\end{align*}
where we have used  \eqref{est C2} and \eqref{sum eps3}.

We therefore obtain that \begin{align*}
\sup_{V\subset U', \mathcal{H}^{d-1}(V)\leq C\rho^{d-1}}\int_{U''\backslash V} |u-u_0|\dH=0.
\end{align*}
Letting $\rho\searrow 0$ yields the theorem.\hfill\qedsymbol




\section{Proofs of the technical Lemmata}\label{S5}

\subsection{Proof of the diffeomorphism Lemma \ref{trans lemma}}
\vphantom{a}

\textbf{a)} 
We first check that the assumptions hold for the transformed data.
\ref{A1} is trivial. \ref{A6} is trivial if $u_0\in C^0$, if $u_0\in BV$ it follows from the chain rule for $BV$-functions, see e.g.\ \cite[Remark 3.18]{ambrosio2000functions}. For $u_0\in W^{\alpha,p}$ it follows e.g.\ from applying the transformation rule to the integral in the definition and using that $\Phi$ is Lipschitz.

Regarding \ref{A2}, we can check the different points as follows: \begin{itemize}
\item Convexity of $f_\Phi$ is trivial from the definition.
\item The growth bound \eqref{bd gf1} holds because $\Phi$ is $C^1$ and, since it is a diffeomorphism, the $C^1$ norm of the inverse is also bounded, yielding the bound directly from the definition.
\item $f_\Phi$ being $C^2$ in the $\xi$-variable outside a large ball follows from $\Phi$ being $C^2$. We also see directly from the definition and the chain rule that \begin{align*}
\scalar{\mathrm{D}_\xi^2 f_\Phi(y,\xi)\frac{\xi}{|\xi|}}{\frac{\xi}{|\xi|}}=|\det \mathrm{D}\Phi^{-1}|\scalar{\mathrm{D}_\xi^2 f\big(\Phi^{-1}(y),(\mathrm{D}\Phi^{-1})^{-T}\xi\big)\frac{(\mathrm{D}\Phi^{-1})^{-T}\xi}{|\xi|}}{\frac{(\mathrm{D}\Phi^{-1})^{-T}\xi}{|\xi|}}
\end{align*}
which yields \eqref{conv df} for $f_\Phi$ after applying \eqref{conv df} with $(\mathrm{D}\Phi^{-1})^T\xi$ in place of $\xi$ and using the boundedness of the derivatives of $\Phi^{-1}$.
\item Continuity of $f_\Phi$ follows from the regularity of $\Phi$.
\item The recession function of $f_\Phi$ is \begin{align*}
f_\Phi^\infty(y,\xi)=|(\det \mathrm{D}\Phi^{-1})y|f^\infty\big(\Phi^{-1}(y),(\mathrm{D}\Phi^{-1}(y))^{-T}\xi\big)\end{align*}
and the convergence of $\frac{1}{t}f_\Phi(y,t\cdot)$ to it follows from the continuity of $\Phi$ and the boundedness of the involved derivatives.
\item The regularity of the recession function follows from the chain rule, the homogeneity of the recession function, and $\Phi$ being $C^2$.
\end{itemize}



It follows directly from the fact that $\Phi$ is $C^2$ that $\lambda_\Phi$ and $h_\Phi$ have the regularity/integrability required by \ref{A4}-\ref{A5}.

It follows from the transformation rule that $g_\Phi\in L^{\max(2,d)}(\Omega')$. In order to establish the curvature condition in \ref{A3}, we write $H_{\de \Omega', f_\Phi}$ to denote the generalized curvature with respect to $\Omega'$ and $f_\Phi$, defined analogously to \eqref{def hf}. We claim that for $y\in U$ it holds that \begin{align}
H_{\de \Omega', f_\Phi}(\Phi(y))=|\det \mathrm{D}\Phi(y)|^{-1}H_{\de\Omega,f}(y),\label{cur ch}
\end{align}
which then implies \eqref{A3}.

To show this, we first claim that for any $C^1$ vector field $z_0$, defined in a neighborhood of $U$, for which it holds that $z_0(y)=c(y)\nu_y$ for some $c(y)\in \R_{<0}$ on $U$, it holds that \begin{align}
\div\mathrm{D}_\xi f^\infty(y,z_0(y))=\div\mathrm{D}_{\xi} f^\infty(y,\mathrm{D}\mathfrak{d}(y))\quad\text{ for $y\in U$.}\label{n field claim}
\end{align}
Indeed, to see this, we first note that we can assume that $z_0(y)=-\nu_y=\mathrm{D}\mathfrak{d}(y)$ on the boundary because $\mathrm{D}_\xi f^\infty$ is positively $0$-homogeneous.
Therefore, we only need to show that the part of the divergence that is not determined by the values on the boundary is the same for both fields.

We can rewrite the divergence at each point $x\in U$, thanks to its frame invariance, as $\div=\nu\cdot\de_\nu+ \mathfrak{pr}_{T_x}(\nabla-\de_\nu)$, where $\mathfrak{pr}_{T_x}$ is the orthogonal projection onto the tangent space of $\de\Omega$ at $x$ and $\de_\nu$ the normal derivative. The second summand is determined by the boundary values only.


 By the chain rule, it holds that \begin{align*}
\nu_y\cdot \de_\nu \mathrm{D}_\xi f^\infty(y,z_0(y))=\nu_y\otimes \nu_y:\de_y\mathrm{D}_\xi f^\infty(y,z_0(y))+\nu_y\cdot \mathrm{D}_\xi^2 f^\infty(y,-\nu_y)\cdot \de_\nu z_0(y).
\end{align*}
The first term is determined by the values of $z_0(y)$ at the boundary. To see that the second term vanishes, we note that the function \begin{align*}
-\nu_y\cdot \mathrm{D}_\xi f^\infty(y,v)
\end{align*}
has a local maximum in the variable $v$ at $v=-\nu_y$ by e.g.\ Lemma \ref{fen lemma} and therefore its derivative must vanish, yielding $\nu_y\cdot \mathrm{D}_\xi^2 f^\infty(y,-\nu_y)\cdot \de_\nu z_0(y)=0$. This implies \eqref{n field claim}.

To proceed with the proof of the equivalence, we let $\mathfrak{d}_{\Omega'}$ denote the distance to the boundary in $\Omega'$.
We note that, thanks to the identity $|\det \mathrm{D}\Phi|^{-1}\div(\psi\circ\Phi)=\div(|\det \mathrm{D}\Phi^{-1}|(\mathrm{D}\Phi^{-1}(\cdot))^{-1}\psi)\circ \Phi$ which holds for every vector field $\psi$, we have \begin{align*}
\mel\big(\div \mathrm{D}_\xi f_\Phi^\infty(\cdot,\mathrm{D}\mathfrak{d}_{\Omega'}(\cdot))\big)\circ \Phi\\
&=\div\Big(|\det \mathrm{D}\Phi^{-1}| (\mathrm{D}\Phi^{-1}(\cdot))^{-1}\mathrm{D}_\xi f^\infty\left(\Phi^{-1}(\cdot),(\mathrm{D}\Phi^{-1}(\cdot))^{-T}\mathrm{D}\mathfrak{d}_{\Omega'}(\cdot)\right)\Big)\circ \Phi\\
&=|\det \mathrm{D}\Phi|^{-1}|\div\left(f^\infty(\cdot,(\mathrm{D}\Phi^{-1})^{-T}\circ \Phi\mathrm{D}\mathfrak{d}_{\Omega'}\circ \Phi )\right).
\end{align*}
The claim \eqref{n field claim} is applicable to the field $(\mathrm{D}\Phi^{-1})^{-T}\circ \Phi\mathrm{D}\mathfrak{d}_{\Omega'}\circ \Phi=\mathrm{D}(\mathfrak{d}_{\Omega'}\circ \Phi)$. Indeed, this field is trivially normal to the boundary as $\Phi$ maps the boundary to the boundary, the interior to the interior, and is bilipschitz.
This yields \begin{align*}
\mel(\div \mathrm{D}_\xi f_\Phi^\infty(\cdot,\mathrm{D}\mathfrak{d}_{\Omega'}(x)))\circ \Phi=|\det \mathrm{D}\Phi|^{-1}\div \mathrm{D}_\xi f^\infty(\cdot, \mathrm{D}\mathfrak{d}(\cdot)).
\end{align*}
\eqref{cur ch} follows by applying the same argument to $-\mathfrak{d}_{\Omega'}$ and noting that this is the term in the definition \eqref{def hf} of the generalized curvature.

Regarding \eqref{sd pushforward}, we note that \begin{align*}
(\Phi^*\mathcal{F}_{u_0\circ \Phi^{-1}})(w\circ \Phi^{-1})=\mathcal{F}_{u_0}(w)
\end{align*}
by the chain rule whenever $w\in W_{u_0}^{1,1}$. By using an approximating sequence and Proposition \ref{rel func} both for $\Phi^*\mathcal{F}_{u_0\circ \Phi^{-1}}$ and $\mathcal{F}_{u_0}$, we see that this also holds for $w\in BV(\Omega)$. This implies by the transformation rule and the definition of the subdifferential that $v\in \de\mathcal{F}_{u_0}(w)$ if and only if $|\det \mathrm{D}\Phi^{-1}|v\circ\Phi^{-1}\in \de(\Phi^*\mathcal{F}_{u_0\circ \Phi^{-1}})(w\circ \Phi^{-1})$. This in turn implies \eqref{sd pushforward} by unraveling the definitions.
\smallskip

\textbf{b)} We first show the first two points by constructing $\Phi=\Phi_\kappa$ depending on the (small) parameter $\kappa>0$ and then show that the third point holds automatically if we pick $\kappa$ small enough.

By translation, it is not restrictive to assume that $x_0=e_d$. We let $U'$ be a neighborhood of $x_0$ in $U$ such that in this neighborhood $\de\Omega$ can be written as a graph of a strictly positive $C^2$-function $\phi$ over the unit ball in $\R^{d-1}$ (after a rotation and a rescaling) and that \begin{align*}
\left\{y\,\big|\,(y_1,\dots, y_{d-1})\in B_{1}^{\R^{d-1}}(0)\text{ and } y_d\in (0,\,\phi(y_1,\dots, y_{d-1})))\right\}\subset \Omega,
\end{align*}
as well as \begin{align*}
\left\{y\,\big|\,(y_1,\dots, y_{d-1})\in B_{1}^{\R^{d-1}}(0)\text{ and } y_d\in (\phi(y_1,\dots, y_{d-1}),3)\right\}\cap\Omega=\emptyset.
\end{align*}
 By rescaling and potentially making $U'$ smaller, it also not restrictive to assume that $\phi$ takes values in $(\frac{9}{10},\frac{11}{10})$.
 Pick $\delta>0$ so that $U'$ is contained in the graph over $B_{1-\delta}^{\R^{d-1}}(0)$ and let $\psi_1:\R^{d-1}\rightarrow \R_{\geq 0}$ denote a smooth, nonnegative cutoff function which equals $1$ on $B_{1-\delta}^{\R^{d-1}}(0)$ and is supported on $B_{1-\frac{1}{2}\delta}^{\R^{d-1}}(0)$. Let $\psi_2$ denote a smooth, non-negative cutoff function which equals $1$ on $(\frac{1}{2},\frac{3}{2})$ and is supported in $(\frac{1}{4},2)$.

We can then use the following diffeomorphism \begin{align*}
&(\Phi_\kappa(y)_1,\cdots, \Phi_\kappa(y)_{d-1}):=(y_1,\cdots, y_{d-1})=:y'\\
&\Phi_\kappa(y)_d:=\psi_1(y')\psi_2(y_d)\frac{y_d}{\phi(y')}\left(\sqrt{\frac{1}{\kappa^2}-|y'|^2}-\left(\frac{1}{\kappa}-1\right)\right)+\big(1-\psi_1(y')\psi_2(y_d)\big)y_d.
\end{align*}
We can extend this to the whole $\Omega$ as the identity, which does not affect the regularity due to the cut-off. We define $\Omega'$ (depending on $\kappa$) as the image of $\Omega$ under $\Phi_\kappa$.

It holds that \begin{align*}
\norm{\sqrt{\frac{1}{\kappa^2}-|y'|^2}-\frac{1}{\kappa}}_{C^k(B_1(0))}=\norm{\frac{\kappa|y'|^2}{\sqrt{1-\kappa^2|y'|^2}+1}}_{C^k(B_1(0))}\leq C(k,d) \kappa
\end{align*}
for sufficiently small $\kappa$ and $k\in \N_{\geq 0}$, where the constant $C(k,d)$ does not depend on $\kappa$.

Therefore, we see that this is indeed a $C^2$-diffeomorphism for sufficiently small $\kappa$ and that $U'$ is mapped to the graph of $\sqrt{\frac{1}{\kappa^2}-|y'|^2}-\frac{1}{\kappa}-1$, which a subset of the boundary of a ball of radius $\frac{1}{\kappa}$ and this yields \eqref{ballbd} after translating.

Furthermore, it is straightforward to verify that $\norm{\Phi_\kappa}_{C^2}$ and $\norm{\Phi_\kappa^{-1}}_{C^2}$ are bounded uniformly in small $\kappa$. This shows the first two points.\smallskip

We move on to the proof of the third property.
We set \begin{align*}
c:=\inf_{x\in U}\essinf_{y\rightarrow x} H_{\de\Omega,f}(x)-|g(y)|,
\end{align*}
 which is $>0$ by assumption. It follows from part a) and the uniform estimates on $\Phi_\kappa$ that we also have \begin{align}
\inf_{x\in \Phi_{\kappa}(U')}\essinf_{y\rightarrow x} H_{\de\Omega',f_{\Phi_\kappa}}(x)-|g_{\Phi_\kappa}(y)|\geq c|\det\mathrm{D} \Phi_\kappa^{-1}|\geq c' \label{essinf re}
\end{align}
for some $c'>0$, uniformly in small $\kappa$. This also implies that for every open subset $V\subset \Phi_{\kappa}(U')$ with a positive distance to the boundary of $\Phi_{\kappa}(U')$, there is an $\eps_1=\eps_1(\kappa)>0$ such that \begin{align}
\inf_{x\in V}\essinf_{y:\, |y-x|\leq \eps_1}H_{\de\Omega',f_{\Phi_\kappa}}(x)-|g_{\Phi_\kappa}(y)|\geq \frac{c'}{2},\label{act inf}
\end{align}
if not, we could find a sequence of $x_n$ for which this difference is $\leq \frac{c'}{2}$ on a set of positive measure with distance $\leq \frac{1}{n}$ to $x_n$, which by compactness and continuity of $H_{\de\Omega',f_{\Phi_\kappa}}$ would show that \eqref{essinf re} is wrong at any accumulation point of the $x_n$.


We now claim that
 \begin{align}
\left|\div_y \mathrm{D}_\xi f_{\Phi_\kappa}^\infty(y,\pm\mathrm{D}\mathfrak{d}_{\Omega'}(y))\big|_{y=x}- \div_y \mathrm{D}_\xi f_{\Phi_\kappa}^\infty(y,\mp\nu_x^{\Omega'})\big|_{y=x}\right|\xrightarrow{\kappa\searrow 0} 0,\label{curv claim}
\end{align}
uniformly in $x\in V$, where $\mathfrak{d}_{\Omega'}$ denotes the signed distance to the boundary in $\Phi_\kappa(\Omega)$ and $\nu_x^{\Omega'}$ is the outer normal of $\Omega'$ at $x$. To show the claim, we first note that it holds that \begin{align}
|\mathrm{D}^2\mathfrak{d}_{\Omega'}(x)|\leq C(d) \kappa,\label{kappa bd}
\end{align} 
where the constant $C(d)$ does not depend on $\kappa$, for $x\in \Phi(U')$, since $\Phi(U')$ is a subset of a sphere of curvature $\kappa$ by construction.
%

We can then estimate, using the chain rule and that $\nu_x^{\Omega'}=-\mathrm{D}\mathfrak{d}_{\Omega'}(x)$ does not depend on $y$ \begin{align*}
\mel\left|\div_y \mathrm{D}_\xi f_{\Phi_\kappa}^\infty(y,\pm\mathrm{D}\mathfrak{d}_{\Omega'}(y))\big|_{y=x}- \div_y \mathrm{D}_\xi f_{\Phi_\kappa}^\infty(y,\mp\nu_x^{\Omega'})\big|_{y=x}\right|\\
&\leq C(d)|\mathrm{D}_{\xi}^2f_{\Phi_\kappa}^\infty(x,\mathrm{D}\mathfrak{d}_{\Omega'}(x))||\mathrm{D}_y^2\mathfrak{d}_{\Omega'}(x))|\\
&\leq C(d)\norm{f_{\Phi_\kappa}^\infty(x,\cdot)}_{C_\xi^2(B_2(0)\backslash B_{\frac{1}{2}}(0))}|\mathrm{D}^2\mathfrak{d}_{\Omega'}(x)|\leq C(d)\kappa\xrightarrow{\kappa\rightarrow 0} 0,
\end{align*}
where the last step follows from \eqref{kappa bd} and the fact that, thanks to step a) and the uniform estimates on $\Phi_\kappa$, the norm $\norm{f_{\Phi_\kappa}^\infty(y,\cdot)}_{C_\xi^2(B_2(0)\backslash B_{\frac{1}{2}}(0))}$ is bounded uniformly in $\kappa$.

Now, we see from \eqref{curv claim} and the the definition \eqref{def hf} of the generalized curvature that, if we choose $\kappa$ sufficiently small, it holds that \begin{align*}
\pm\div_y \mathrm{D}_\xi f_{\Phi_\kappa}^\infty(y,\pm\nu_x^{\Omega'})\big|_{y=x}-H_{\de\Omega',f_{\Phi_\kappa}}(x)\geq-\frac{c'}{8}
\end{align*}
for all $x\in V$. It further follows from the fact that $f_{\Phi_\kappa}^\infty$ is $C_y^1C_\xi^1$, that, after potentially lowering $\eps_1$ (which can be done uniformly in $x\in V$), we also have \begin{align*}
\essinf_{y:\, |y-x|\leq \eps_1}\pm\div_y \mathrm{D}_\xi f_{\Phi_\kappa}^\infty(y,\pm\nu_x^{\Omega'})-H_{\de\Omega',f_{\Phi_\kappa}}(x)\geq-\frac{c'}{4}
\end{align*}
which, together with \eqref{act inf}, yields \eqref{impr cc}.\hfill\qedsymbol


\subsection{Proof of the geometric Lemmata \ref{geo lemma1} and \ref{geo lemma 2}}\label{S52}

We only prove Lemma \ref{geo lemma 2} in dimension $d\geq 2$. The Lemma \ref{geo lemma1} is merely a special case of this if $d\geq 2$ and trivial if $d=1$, while the Lemma \ref{geo lemma 2} is a void statement if $d=1$ as the boundary of a $1$-dimensional set can not be uniformly convex.

\begin{proof}[Proof of Lemma \ref{geo lemma 2} if $d\geq 2$]
It is not restrictive to assume that $x^0=0$ and that $\nu_{x^0}=-e_d$, otherwise we can rotate and translate the set. We can further assume that the boundary of $\Omega$ around $x^0$ is the graph of some $\varpi\in C^2$ such that \begin{align}
1\lesssim \mathrm{D}^2\varpi\lesssim 1\quad \text{ in the sense of positive definiteness,}\label{pi nice}
\end{align}
in a neighborhood of $x^0$.
 One can easily check that the implicit constant in \eqref{pi nice} can be chosen (locally) uniformly in $x^0$.

By Taylor's theorem and \eqref{pi nice} there is a constant $C_4$ (independent of $x$ and $\eps$) such that in a neighborhood of $x^0$, it holds that \begin{align}
&\frac{1}{C_4}|y'|^2\leq \varpi(y')\leq C_4|y'|^2.\label{pi bd}\\
&|\mathrm{D}\varpi(y')|\leq 1.\label{dpi bd}
\end{align}
We take $U''$ as the intersection of such an open neighborhood with $\de \Omega$. After potentially making $U''$ smaller, we can also assume that $U''$ is compactly contained in $U'$ and that these bounds on $\varpi$ still hold in a small neighborhood of $U''$.

\textbf{a)} Since all the bounds on $\varpi$ are uniform in $x^0$, it is not restrictive to only show the statements for $x=x^0=0$. As $\de\Omega\cap B$ is the graph of $\varpi$, we know that some point $y=(y',y_d)$ lies in $\Omega\cap B$ if and only if $y\in B$ and $y_d\geq \varpi(y')$.

Since we also have $-\scalar{y}{\nu_x}=y_d$, this shows that for every $y\in \tilde{D}_x^\eps$ we must have $\varpi(y)\leq y_d\leq \eps^2$ and therefore
 $|y'|< C_4^{\frac{1}{2}}\eps$, yielding the upper bound in \eqref{est D2} on the volume of $\tilde{D}_x^\eps$. This also shows that $\tilde{C}_x^\eps\cup\tilde{D}_x^\eps \cup \tilde{E}_x^\eps=\overline{\tilde{D}_x^\eps}$ is contained in a ball of radius of order $\eps$ around $x$, which in turn implies that if $\eps$ is small enough, then the $\tilde{C}_{x}^\eps$ lie in $U'$.
 
 Similarly, we see the upper bound on $\tilde{E}_x^\eps$ in \eqref{est E2} because every $y\in \tilde{E}_x^\eps$ is of the form $(y',\eps^2)$ with $|y'|\leq C_4^{\frac{1}{2}}\eps$.

To see the lower bounds in \eqref{est D2} and \eqref{est E2} on these sets, we note that, by the upper bound in \eqref{pi bd}, we must have \begin{align*}
B_{\frac{1}{2}C_4^{-\frac{1}{2}}\eps}(0)\times (\frac{1}{4}\eps^2,\eps^2)\subset \tilde{D}_x^\eps\quad \text{and} \quad B_{\frac{1}{2}C_4^{-\frac{1}{2}}\eps}(0)\times \{\eps^2\}\subset \tilde{E}_x^\eps.
\end{align*}
For notational convenience, we set \begin{align*}
P_\eps:=\{y'\in B\cap \R^{d-1}\,\big|\,\varpi(y')< \eps^2\},
\end{align*}
where we identify $\R^{d-1}$ with $\R^{d-1}\times \{0\}\subset \R^d$.
It follows from \eqref{pi bd} that \begin{align}
B_{C_4^{-\frac{1}{2}}\eps}(0)\subset P_\eps\subset B_{C_4^{\frac{1}{2}}\eps}(0),\label{bd P}
\end{align}
uniformly in $\eps$.

 To see the bound \eqref{est C1} on $\tilde{C}_x^\eps$, we note that $\tilde{C}_x^\eps$ is precisely the graph of $\varpi$ over $P_\eps$ and that $|\mathrm{D}\varpi|$ is bounded in a neighborhood of $x$ and hence we have \begin{align*}
\mathcal{H}^{d-1}(\tilde{C}_x^\eps)\approx \mathcal{H}^{d-1}(P_\eps),
\end{align*}
which yields the bound by \eqref{bd P}.

We move on to showing \eqref{C lip}. For $\eps'\in (\frac{1}{2}\eps,\eps)$, it holds that $\tilde{C}_x^{\eps'}\subset \tilde{C}_x^\eps$ since $\tilde{D}_x^{\eps'}\subset \tilde{D}_x^\eps$ by definition. We have that $\tilde{C}_x^\eps\backslash \tilde{C}_x^{\eps'}$ is the graph of $\varpi$ over $P_\eps\backslash P_{\eps'}$.
As the slope of $\varpi$ has a modulus $\leq 1$ by \eqref{dpi bd}, we hence have that \begin{align}
\mathcal{H}^{d-1}(\tilde{C}_x^\eps\backslash \tilde{C}_x^{\eps'})\approx \mathcal{H}^{d-1}(P_\eps\backslash P_{\eps'}).\label{eq meas}
\end{align}
We have $P_\eps\backslash P_{\eps'}=\varpi^{-1}([(\eps')^2,\eps^2))$ by definition.
We therefore have $P_\eps\backslash P_{\eps'}\subset B_{C_4^\frac{1}{2}\eps}(0)\backslash B_{C_4^{-\frac{1}{2}}\eps'}(0)$  by \eqref{pi bd}, in particular,  we also have $|y'|\approx \eps$ for every $y'\in P_\eps\backslash P_{\eps'}$. Hence by the convexity, it also holds that $\mathrm{D}\varpi(y')\cdot\frac{y'}{|y'|}\approx \eps$ since $\mathrm{D}\varpi(0)=0$.
This shows that \begin{align*}
\mathcal{H}^{1}(P_\eps\backslash P_{\eps'}\cap (\frac{y'}{|y'|}\R_{\geq 0}))\approx \eps-\eps'
\end{align*}
for each fixed direction $\frac{y'}{|y|}$.


 By Fubini, we can therefore write \begin{align*}
\mel\mathcal{H}^{d-1}(P_\eps\backslash P_{\eps'})=\int_{S^{d-2}}\int_{\R_{\geq 0}}r^{d-2}\mathds{1}_{rv\in P_\eps\backslash P_{\eps'}}\dd\mathcal{H}^{1}(r)\dd\mathcal{H}^{d-2}(v)\\
& \lesssim \eps^{d-2}(\eps-\eps')\int_{S^{d-2}}1\dd\mathcal{H}^{d-2}(v)\approx \eps^{d-2}(\eps-\eps')
\end{align*}
which, together with \eqref{eq meas}, shows \eqref{C lip}.

\textbf{b)} $\tilde{D}_x^\eps$ is convex by definition, and any convex set has Lipschitz boundary.

\textbf{c)} Assume that $y\in \tilde{D}_{x}^{\eps}\cap\tilde{D}_{x'}^{\eps'}$, then since $\diam(\tilde{D}_x^\eps)\lesssim \eps$, we know that $|x-x'|\lesssim \eps'+\eps<<1$. By continuity and since the boundary is $C^2$, we have that $\scalar{\nu_x}{\nu_{x'}}>0$. By the convexity and boundedness of the sets, there is a unique $s>0$ such that $y+s\nu_x\in \de( \tilde{D}_x^\eps\cap \tilde{D}_{x'}^{\eps'})$. Since $\scalar{y+s\nu_x-x}{\nu_x}>\scalar{y-x}{\nu_x}>-\eps^2$ and similarly  $\scalar{y+s\nu_{x}-x'}{\nu_{x'}}>-(\eps')^2$, we see that $y+s\nu_x$ lies in neither of the $\tilde{E}$'s and must therefore lie in $\tilde{C}_{x}^{\eps}\cap\tilde{C}_{x'}^{\eps'}$, which implies that the intersection is not empty.



%




\textbf{d)} Convexity is trivial, since each $\tilde{C}_x^\eps$ is convex by definition and projections of convex sets are convex.

To see the other statement, we first note that here $\mathfrak{pr}$ is just the projection onto the first $d-1$ components.
For $(y^0,\varpi(y^0))=x\in U''$ close to $x_0=0$ it holds that \begin{align}
\tilde{C}_x^\eps=\left\{y\in B\,\big|\, y_d= \varpi(y')\text{ and } \mathrm{D}\varpi(y^0)(y'-y^0)-(y_d-\varpi(y_0))\geq -\eps^2\sqrt{1+|\mathrm{D}\varpi(y^0)|^2}\right\}\label{set vpi}
\end{align}
where we again wrote $y=(y',y_d)$ and because $(\mathrm{D}\varpi(y^0),-1)$ is the normal at $x$.

This condition holds in a neighborhood of order $\eps$ of $y^0$ by Taylor's theorem (uniformly in $x$). This shows the first inclusion in \eqref{morse prereq}.
On the other hand, by the uniform convexity, it also holds that $\varpi(y')-\varpi(y^0)-\mathrm{D}\varpi(y^0)(y'-y^0)\gtrsim |y'-y^0|^2$, uniformly in $x$. This shows (together with \eqref{dpi bd}) that the condition in \eqref{set vpi} can also only hold for $y'$ in a neighborhood of order $\eps$ of $y^0$. Together, we have obtained the second inclusion in \eqref{morse prereq}.

\textbf{e)} The upper bound is a trivial consequence of the fact that the normal is Lipschitz as the surface is $C^2$ by assumption and that $\tilde{C}_x^\eps$ is contained in a ball of radius $ C\eps$ as established in the proof of a).

Regarding the lower bound, we only show this for $x=x^0=0$, the general case follows by rotating and translating, and because all the involved estimates on $\varpi$ are uniform in $x^0$.
We can reuse that the projection of $\tilde{C}_x^{\rho\eps}$ onto $\R^{d-1}$ is $P_{\rho\eps}$ and contains a ball of radius $C_4^{-\frac{1}{2}}\rho\eps$ by \eqref{bd P},  and hence for all $y=(y',\varpi(y'))$ with $\varpi(y')\in (\rho^2\eps^2,\eps^2)$ we have $|y'|\gtrsim \rho\eps$ . This shows by uniform convexity that \begin{align*}
|\mathrm{D}\varpi(y')|\geq \frac{1}{|y'|}\scalar{\mathrm{D}\varpi(y')-\mathrm{D}\varpi(0)}{y'-0}\gtrsim |y'|\approx \rho\eps.
\end{align*}
Now every point in $\tilde{C}_x^\eps\backslash \tilde{C}_x^{\rho\eps}$ is of this form and therefore \begin{align*}
\nu_{(y',\varpi(y'))}=\frac{1}{\sqrt{1+|\mathrm{D}\varpi(y')|^2}}\binom{\mathrm{D}\varpi(y')}{-1}.\end{align*} 
Using the elementary estimate $\left|\frac{1}{\sqrt{1+|a|^2}}\binom{a}{-1}+e_d\right|\gtrsim \min(\frac{1}{2},|a|)$ we conclude.
\end{proof}

\subsection{Proof of the Lemmata \ref{char1}, \ref{dlvp} and \ref{char2}}\label{S53}
We prove all three Lemmata together: We first construct $\hatu$ and then show the bounds for $\hatu$ and $\bar{u}_x$ in parallel.

\textbf{Construction of $\hatu$:} First observe that $\hat{z}_x$ is $C^1$ by the assumption \ref{A2} and its definition \eqref{def hatz}, hence the characteristics of $\hatz$ exist in a neighborhood of the boundary  (cf.\ \cite[Chapter 3]{Evans}).  
We furthermore have
\begin{align}
- \scalar{\hat{z}_x(y)}{\nu_x}=\scalar{\mathrm{D}_\xi f^\infty(y,-\nu_x)}{-\nu_x}=f^\infty(y,-\nu_x)\gtrsim 1,\label{prod est}
\end{align}
thanks to \eqref{doal} and \eqref{gc finf}.

This implies that no characteristic can intersect $E_x^\eps$ in more than one point and, since by continuity, for small $\eps>0$ it also holds $\scalar{\hat{z}_x}{\nu_y}<0$ in a neighborhood of size $O(1)$ of $x$, that no characteristic can intersect $C_x^\eps$ twice.

Therefore, there is a well-defined map $\Psi:\overline{D_x^\eps}\rightarrow \overline{C_x^\eps}\times [0,1]$, depending on $x$ and $\eps$, which is a $C^1$-diffeomorphism onto its image, defined through its inverse fulfilling the ODE \begin{align}
\frac{\mathrm{d}}{\mathrm{d}t}\Psi^{-1}(y,t)=\hatz(\Psi^{-1}(y,t))\quad\text{and} \quad \Psi^{-1}(y,0)=y.\label{charode}
\end{align}
We extend $u_0$ to a function on $\overline{C_x^\eps}\times [0,1]$ which does not depend on the last variable. We denote the extension by $u_0^e$.

%

%


We then set \begin{align}
\hat{u}_x=u_0^e\circ \Psi.\label{def pb}
\end{align} 
This fulfills \eqref{charu1} and \eqref{charu2} trivially if $u_0$ is $C^1$, the general case follows from the fact that the Anzelotti pairing with $\hatz$ is just the pointwise product since $\hatz$ is $C^1$, once we have established that $\hatu\in BV$.\smallskip

\textbf{Estimates on $\Psi$:} We next have to show the estimate \eqref{est hatu1}, and in the case of Lemma \ref{char2}, the estimates \eqref{est baru1} and \eqref{est baru2}. In order to do so, we will first show that $\Psi$ is bi-Lipschitz uniformly in $x$ and $\eps$, which is quite standard, but we provide the proof here anyway for the sake of completeness.


We start with $\norm{\Psi}_{C^1}$, in order to bound it, it suffices to estimate the Lipschitz constants of the two components of $\Psi$ uniformly.

The (backwards) characteristic starting at $y\in D_x^\eps$ can be written as $\Psi^{-1}(\Psi_1(y),\Psi_2(y)-t)$, where the indices denote the components.
As $\hat{z}_x$ is $C^1$, uniformly in $x$ and $\eps$, we therefore have by Gronwall, applied to \eqref{charode}\begin{align}
\left|\Psi^{-1}(\Psi_1(y),\Psi_2(y)-t)-\Psi^{-1}(\Psi_1(y'),\Psi_2(y')-t)\right|\leq e^{\norm{\hat{z}_x}_{W^{1,\infty}}t}|y-y'|\label{gw1}
\end{align}
for $t\leq \min(\Psi_2(y),\Psi_2(y'))$, where $y,y'\in D_x^\eps$ are arbitrary. 

Now, if we take $t_y=\Psi_2(y)\leq \Psi_2(y')$, which is not restrictive by symmetry, then it holds that $\Psi^{-1}(\Psi_1(y),\Psi_2(y)-t_y)=\Psi_1(y)\in C_x^\eps$ and hence \eqref{gw1} gives \begin{align*}
\dist(\Psi^{-1}(\Psi_1(y'),\Psi_2(y')-t_y),C_x^\eps)\lesssim |y-y'|,
\end{align*}
where we could absorb the exponential factor into the implicit constant because $t_y\lesssim 1$.
By the estimate \eqref{prod est} and the fact that the curve $\Psi^{-1}(\Psi_1(y'),\Psi_2(y')-t)$ is the characteristic through $y'$, we have \begin{align*}
\frac{\mathrm{d}}{\mathrm{d}t}\dist\big(\Psi^{-1}(\Psi_1(y'),\Psi_2(y')-t),C_x^\eps\big)\leq -\min_{x'\in C_x^\eps}\scalar{\hat{z}_x}{\nu_{x'}}\lesssim -1,
\end{align*}
 this implies that the time $\Psi_2(y')$ at which the curve $\Psi^{-1}(\Psi_1(y'),\Psi_2(y')-t)$ hits $C_x^\eps$ is $\lesssim |y-y'|$ apart from $t_y$, yielding \begin{align*}
|t_y-\Psi_2(y')|=|\Psi_2(y)-\Psi_2(y')|\lesssim |y-y'|.
\end{align*}
 By the boundedness of $\hat{z}_x$, combined with \eqref{gw1}  \begin{align*}
\mel|\Psi_1(y)-\Psi_1(y')|=|\Psi^{-1}(\Psi_1(y),0)-\Psi^{-1}(\Psi_1(y'),0)|\\
&\leq \left|\Psi^{-1}(\Psi_1(y),\Psi_2(y)-t_y)-\Psi^{-1}(\Psi_1(y'),\Psi_2(y')-t_y)\right|+\norm{\hat{z}_x}_{L^\infty}\left|\Psi_2(y')-t_y\right|\\
&\lesssim |y-y'|,
\end{align*}
proving that $\norm{\Psi}_{C^1}\lesssim 1$, uniformly in $x$.

Similarly, it holds that \begin{align*}
&|\Psi^{-1}(y,t)-\Psi^{-1}(y,t')|\leq \norm{\hat{z}_x}_{L^\infty}|t-t'|\lesssim |t-t'|\\
&|\Psi^{-1}(y,t)-\Psi^{-1}(y',t)|\leq e^{\norm{\hat{z}_x}_{W^{1,\infty}}t}|y-y'|,
\end{align*}
showing that $\norm{\Psi^{-1}}_{C^1}\lesssim 1$ uniformly in $x$ and $\eps$.

The same argument can also be made for the projection in direction $\zeta_x$, which is used in Lemma \ref{char2}, where one uses \eqref{pos scalar} in place of \eqref{prod est}. We omit the details.\smallskip

%


\textbf{Proof of the bounds:}  We only show the $BV$-bounds for $\hat{u}_x$, the $BV$-bounds for $\bar{u}_x$ proceed in the same way, since one has the same estimates for the characteristics in both cases. The fact that the $\bar{u}_x\in L^2$ is a trivial consequence of the construction and the assumption that $u_0\in L^2$ in the vectorial setting. 

We first note that the image of $D_x^\eps$ under $\Psi$ actually lies in $C_x^\eps\times (0,C\eps^2)$ for some $C$ bounded uniformly in $x$ and $\eps$, since $D_x^\eps$ is contained in a strip of width $\lesssim \eps^2$ in direction $\nu_x$ by definition and since the characteristics have a speed $\gtrsim 1$ in the $\nu_x$-direction by the estimate \eqref{prod est}.

By e.g.\ the chain rule and smooth approximation (see e.g.\ \cite[Remark 3.18]{ambrosio2000functions} for more details), it holds $\hat{u}_x\in BV$ and that
 \begin{align*}
\mathrm{D}^a\hat{u}_x(y)=\mathrm{D}^au_0^e(\Psi(y)) \mathrm{D}\Psi(y)
\end{align*}
and
 \begin{align*}
(\mathrm{D}^s\hat{u}_x)(A)=\int_{\Psi(A)}|\det \mathrm{D}\Psi|^{-1} \mathrm{D}\Psi^T\dd\mathrm{D}^su_0^e,
\end{align*}
for any Borel set $A$.
Therefore, using that both the $C^1$-norms of $\Psi$ and $\Psi^{-1}$ are uniformly bounded by the previous step, we see that for any nondecreasing function $i_0:\R_{\geq 0}\rightarrow \R_{\geq 0}$ with $i_0(2t)\lesssim i_0(t)$ it holds that \begin{align}
\mel\int_{D_x^\eps} i_0(|\mathrm{D}^a\hat{u}_x(y)|)\dy\lesssim \int_{C_x^\eps\times (0,C\eps^2) }i_0(|\mathrm{D}^au_0^e(y)|)\dd(\mathcal{H}^{d-1}\times \mathcal{L}^1)(y)\nonumber\\
& \approx \eps^2\int_{C_x^\eps}i_0(|\mathrm{D}^au_0(y)|)\dd\mathcal{H}^{d-1}(y)\label{i est}
\end{align}
and \begin{align*}
|\mathrm{D}^s(\hat{u}_x)|(D_x^\eps)\lesssim \eps^2|\mathrm{D}^su_0|(C_x^\eps)
\end{align*}
since $u_0^e$ is a constant extension of $u_0$  onto a strip of $\lesssim \eps^2$.
 
In particular, this shows the estimate \eqref{est hatu1} by setting $i_0=|\cdot|$. By the de La Vall\'ee Poussin Lemma \cite[p.\ 3]{rao1991theory}\footnote{The property that $i(2t)\lesssim i(t)$ is not stated in the reference, one can however easily achieve it by replacing $i(t)$ with $\min_{\rho\in (0,1]}\rho^{-2}(i+1)(\rho t)$}, we can pick $i:\R_{\geq 0}\rightarrow \R_{\geq 0}$ non-decreasing, such that $i(2t)\lesssim i(t)$ and \begin{align*}
\lim_{t\rightarrow+\infty} \frac{i(t)}{t}=+\infty
\end{align*} 
and \begin{align*}
\int_{U'} i(|\mathrm{D}^au_0|(y))\dH(y)<\infty.
\end{align*}
Setting $i=i_0$ in \eqref{i est} then shows the Lemma.\hfill\qedsymbol

\subsection{Proof of the Lemmata \ref{cont z1} and \ref{cont z2}}\label{S54}
\textbf{Part a)} We first focus on the statement for $z$ in Lemma \ref{cont z1} a), and show Lemma \ref{cont z2} a) afterwards. 

We first partially integrate $z-\hatz$, using \eqref{gg form}, against the function $\scalar{y-x}{\nu_x}+\frac{1}{2}\kappa\eps^2$, which vanishes on $E_x^\eps$ by its definition.

This yields that \begin{equation}\begin{aligned}
\mel\int_{D_x^\eps} \scalar{z-\hat{z}_x}{\nu_x}\dy=-\int_{D_x^\eps} \div(z-\hat{z}_x)\left(\scalar{y-x}{\nu_x}+\frac{1}{2}\kappa\eps^2\right)\dy\\
&+\int_{C_x^\eps}[z-\hat{z}_x,\nu_y]\left(\scalar{y-x}{\nu_x}+\frac{1}{2}\kappa\eps^2\right)\dH(y).\label{pint lem}
\end{aligned}\end{equation}
We next realize that, thanks to $f^\infty$ being $C^1$ by \ref{A2}, and because $|\nu_x-\nu_y|\lesssim\eps$ in $C_x^\eps$, as it is a spherical cap of base radius $\leq\eps$, it holds that \begin{align*}
\left|\mathrm{D}_\xi f^\infty(y,-\nu_y)-\hat{z}_x\right|=
\left|\mathrm{D}_\xi f^\infty(y,-\nu_y)-\mathrm{D}_\xi f^\infty(y,-\nu_x)\right|\lesssim \eps.
\end{align*}
It holds that \begin{align}
|\scalar{y-x}{\nu_x}+\frac{1}{2}\kappa\eps^2|\leq \frac{1}{2}\kappa\eps^2\label{kappaeps}
\end{align} 
on $D_x^\eps$ and $C_x^\eps$, by definition (see \eqref{def D1}) and since the scalar product is maximized at $y=x$ by convexity. Now, using the assumed inequality \eqref{del ass}, we see that \begin{align*}
\mel\left|\int_{C_x^\eps}[z-\hat{z}_x,\nu_y]\left(\scalar{y-x}{\nu_x}+\frac{1}{2}\kappa\eps^2\right)\dH(y)\right|\\
&\leq \int_{C_x^\eps}\left|[\mathrm{D}_\xi f^\infty(y,-\nu_y)-\hat{z}_x,\nu_y]\left(\scalar{y-x}{\nu_x}+\frac{1}{2}\kappa\eps^2\right)\right|\dH(y)\\
&\quad+\int_{C_x^\eps}\mathds{1}_{[z,\nu_y](y)\neq \scalar{\mathrm{D}_\xi f^\infty(y,- \nu_y)}{\nu_y}}\left|[z-\hat{z}_x,\nu_y]\left(\scalar{y-x}{\nu_x}+\frac{1}{2}\kappa\eps^2\right)\right|\dH(y)\\
&\lesssim \eps^3\mathcal{H}^{d-1}(C_x^\eps)+(\norm{z}_{L^\infty(D_x^\eps)}+\norm{\hatz}_{L^\infty(D_x^\eps)})\eps^2\delta\mathcal{H}^{d-1}(C_x^\eps)\\
&\lesssim (\eps+\delta)\eps^{d+1},
\end{align*}
where we also used the bound \eqref{est C1} on the measure of $C_x^\eps$ and that $\norm{z}_{L^\infty(D_x^\eps)}+\norm{\hatz}_{L^\infty(D_x^\eps)}\lesssim 1$ by \eqref{ball subdiff}, \eqref{uni bd z} and \eqref{bound hatz}.

%

Combining \eqref{pint lem} and \eqref{kappaeps} to estimate the other integral, this shows that
 \begin{align*}
\mel\int_{D_x^\eps} \scalar{z-\hat{z}_x}{\nu_x}\dy\lesssim\eps^2\left(\int_{D_x^\eps}|\div  z|\dy+(\eps+\delta)\eps^{d-1}\right)+\norm{\div\hat{z}}_{L^\infty}\mathcal{L}^d(D_x^\eps)\eps^2\\
& \lesssim\eps^2\left(\int_{D_x^\eps}|\div  z|\dy+(\eps+\delta)\eps^{d-1}\right).
\end{align*}
where we also used that $\norm{\div\hat{z}_x}_{L^\infty}\lesssim \norm{f^\infty}_{C^1_yC^1_\xi}\lesssim 1$ by its Definition \eqref{def hatz} and the Assumption \ref{A2} and that $\mathcal{L}^d(D_x^\eps)\lesssim \eps^{d+1}$  by \eqref{est D1}.
%


We note that by the definition \eqref{def hatz} of $\hat{z}_x$, the pointwise constraint \eqref{cs4} on $z$ and the inequality \eqref{quant fen}, it holds that \begin{align*}
\scalar{z(y)-\hat{z}_x(y)}{\nu_x}=\scalar{\mathrm{D}f^\infty(y,-\nu_x)-z(y)}{-\nu_x}\gtrsim |z-\hat{z}_x|^2,
\end{align*}
yielding \eqref{del stat}.

The vector-valued case in Lemma \ref{cont z2} is very similar. We partially integrate against the weight $v\otimes(\scalar{y-x}{\nu_x}+\eps^2)$, which vanishes on $\tilde{E}_x^\eps$ by definition, and use the fact that the divergences vanish to see that \begin{align*}
\mel\int_{\tilde{D}_x^\eps} \scalar{z-\mu_1(v)\otimes \zeta_x}{v\otimes\nu_x}\dy=
\int_{\tilde{C}_x^\eps}\scalar{[z-\mu_1(v)\otimes \zeta_x,\nu_y]}{v}\left(\scalar{y-x}{\nu_x}+\eps^2\right)\dH(y).
\end{align*}
Since $\mu_2$ inherits the $C^1$-regularity of $f^\infty$, we also have \begin{align*}
|\mu_1(v)\otimes \zeta_y-\mu_1(v)\otimes \zeta_x|=\left|\mu_1(v)\otimes \mu_2(\nu_y)-\mu_1(v)\otimes\mu_2(\nu_x)\right|\lesssim \eps    
\end{align*}
on $\tilde{C}_x^\eps$.
Similarly as above, the assumed inequality \eqref{del ass2} then shows that \begin{align*}
\mel\left|\int_{\tilde{C}_x^\eps}\scalar{[z-\mu_1(v)\otimes\mu_2(\nu_x),\nu_y]}{v}\left(\scalar{y-x}{\nu_x}+\eps^2\right)\dH(y)\right|\\
&\leq \int_{\tilde{C}_x^\eps}\big|\scalar{[\mu_1(v)\otimes \zeta_y-\mu_1(v)\otimes \zeta_x,\nu_y]}{v}\big|\left|\scalar{y-x}{\nu_x}+\eps^2\right|\dH(y)\\
&\quad+\int_{\tilde{C}_x^\eps}\mathds{1}_{[z,\nu](y)\in \scalar{\zeta_y}{\nu_y}B_\delta(\mu_1(v))}|[z,\nu_y]-\mu_1(v)\scalar{\zeta_y}{\nu_y}||v||\scalar{y-x}{\nu_x}+\eps^2|\dH(y)
\\
&\quad+\int_{\tilde{C}_x^\eps}\mathds{1}_{[z,\nu](y)\notin \scalar{\zeta_y}{\nu_y}B_\delta(\mu_1(v))}\left|\scalar{[z-\mu_1(v)\otimes \zeta_x,\nu_y]}{v}\right|\left|\scalar{y-x}{\nu_x}+\eps^2\right|\dH(y)\\
&\lesssim \eps^3\mathcal{H}^{d-1}(\tilde{C}_x^\eps)+\delta\eps^2\mathcal{H}^{d-1}(\tilde{C}_x^\eps)+(\norm{z}_{L^\infty(\tilde{D}_x^\eps)}+1)\eps^2\delta\mathcal{H}^{d-1}(\tilde{C}_x^\eps)\\
&\lesssim (\eps+\delta)\eps^{d+1},
\end{align*}
where we have used the bound \eqref{est C2} on the measure of $\tilde{C}_x^\eps$, as well as the fact that $v$ is a unit vector and that $z$ and the $\mu$'s are $\lesssim 1$ by \eqref{uni bd z} and the assumption on $f$.
Therefore \begin{align}
\int_{\tilde{D}_x^\eps}\scalar{z-\mu_1(v)\otimes \zeta_x}{v\otimes \nu_x }\dy\lesssim (\eps+\delta)\eps^{d+1}
\end{align}
from which one concludes \eqref{del stat2} from Lemma \ref{fen lemma} as above.

\subsection{Proof of Lemma \ref{cont z1} b) and Lemma \ref{cont z2} b)}\label{S55}
Both proofs are quite similar, we start with the one of Lemma \ref{cont z1} b).

We make use of the pointwise characterisation of $z$ in Lemma \ref{sca char}. It holds that \begin{align}
 (z-\hatz,\mathrm{D}u)&=\scalar{z-\hatz}{\mathrm{D}^au}\mathcal{L}^d+ \left(f^\infty(\cdot,\frac{\mathrm{d}\mathrm{D}^su}{\mathrm{d}|\mathrm{D}^su|}(\cdot))-\scalar{\hatz}{\frac{\mathrm{d}\mathrm{D}^su}{\mathrm{d}|\mathrm{D}^su|}(\cdot)}\right)|\mathrm{D}^su|,\label{exp diff}
\end{align}
here we have used that the pairing is linear and that for $\hatz$ it is just the classical product because $\hatz$ is $C^1$ by \eqref{reg hatz}.

The prefactor of the singular part is non-negative and therefore the singular part may be discarded. Indeed this follows from the fact that $\hatz\in \ran(\de_\xi f^\infty(y,\cdot))$ by definition (see \eqref{def hatz}) and the property \eqref{doal}.

Furthermore, if $f(x,\cdot)$ is differentiable, which is e.g.\ the case for $|\mathrm{D}^au|>R$ for the $R$ from Assumption \ref{A2} or if $f=f^\infty$, then it holds that $
z(y)=\mathrm{D}_\xi f(y,\mathrm{D}^a u(y))$ by \eqref{cs2}.
We note that in the special case in which $f$ is positively $1$-homogeneous, the absolutely continuous part of \eqref{exp diff} is also non-negative as a consequence of \eqref{doal} and \eqref{cs2} and because then $f=f^\infty$.

Regarding the absolutely continuous part, we note that among all values in $\ran(\de_\xi f^\infty(y,\cdot))$ the one maximizing the scalar product with $\mathrm{D}^au(y)$ is $\mathrm{D}f^\infty(y,\mathrm{D}^au(y))$ by e.g.\ \eqref{doal} and hence, if $|\mathrm{D}^au|\geq R$ we can estimate,  \begin{align}
\mel\scalar{\mathrm{D}_\xi f(y,\mathrm{D}^au(y))-\hatz}{\mathrm{D}^au(y)}\geq \scalar{\mathrm{D}_\xi f(y,\mathrm{D}^au(y))-\mathrm{D}_\xi f^\infty(y,\mathrm{D}^au(y))}{\mathrm{D}^au(y)}\label{est w rec}\\
&=-\left\langle\int_{|\mathrm{D}^au(y)|}^{\infty}\scalar{\mathrm{D}_\xi^2f(y,s\frac{\mathrm{D}^au(y)}{|\mathrm{D}^au(y)|})\frac{\mathrm{D}^au(y)}{|\mathrm{D}^au(y)|}}{\frac{\mathrm{D}^au(y)}{|\mathrm{D}^au(y)|}}\dd s,\,\mathrm{D}^au(y)\right\rangle,\label{est df prod}
\end{align}
where we used \eqref{conv grad} and the fundamental theorem of calculus in the second step.
If we set \begin{align*}
\mathfrak{f}(t):=\sup_{y\in  \Omega,\, |\xi|\geq t}|\xi|\int_{|\xi|}^\infty|\scalar{\mathrm{D}^2f(y, s\frac{\xi}{|\xi|})\frac{\xi}{|\xi|}}{\frac{\xi}{|\xi|}}|\dd s,
\end{align*}
which is exactly the expression from \eqref{conv df} in the Assumption \ref{A2}, then \eqref{est df prod} implies that \begin{align*}
\scalar{\mathrm{D}_\xi f(y,\mathrm{D}^au(y))-\hatz}{\mathrm{D}^au(y)}\geq -\mathfrak{f}(|\mathrm{D}^au(y)|)
\end{align*}
for $|\mathrm{D}^au(y)|\geq R$.
To conclude, we split the integral of the absolutely continous part in \eqref{exp diff} into the parts where $|\mathrm{D}^au|\geq \max(R,|z-\hatz|^{-\frac{1}{2}})$ and where the opposite inequality holds, this then gives \begin{align*}
\mel\int_{D_x^\eps}|(z-\hatz,\mathrm{D}u)_-|\leq \int_{D_x^\eps}\mathds{1}_{|\mathrm{D}^au|\leq \max(R,|z-\hatz|^{-\frac{1}{2}})} |z-\hatz||\mathrm{D}^au|\dy+\int_{D_x^\eps} \mathfrak{f}(\max(R,|z-\hatz|^{-\frac{1}{2}}))\dy\\
&\leq\int_{D_x^\eps} R|z-\hatz|+|z-\hatz|^\frac{1}{2}+\mathfrak{f}(\max(R,|z-\hatz|^{-\frac{1}{2}}))\dy.
\end{align*}
We then take \begin{align*}k(t)=\min\big(C,Rt+t^\frac{1}{2}+\mathfrak{f}(\max(R,t^{-\frac{1}{2}}))\big),\end{align*}
 where $C$ is some large fixed constant such that $C>\norm{z}_{L^\infty}+\norm{\hatz}_{L^\infty}$ (which exists by \eqref{uni bd z} and \eqref{bound hatz}). Since $|z-\hatz|\leq C$, this cutoff does not affect the previous inequalities.
The function $k$ goes to $0$ at $0$ because $\lim_{t\rightarrow \infty} \mathfrak{f}(t)=0$ by the Assumption \ref{A2}.

\subsubsection{Proof of Lemma \ref{cont z2} b)}
As in \eqref{est w rec} above, we see from the pointwise characterisation in Lemma \ref{vec char} of $z$ that \begin{align}
\scalar{z(y)-\bar{z}_x}{\mathrm{D}^au(y)}\geq \scalar{\mathrm{D}_\xi f(\mathrm{D}^au(y))-\mathrm{D}_\xi f^\infty(\mathrm{D}^au(y))}{\mathrm{D}^au(y)}.\label{easy est}
\end{align}
By the same argument as earlier\begin{align*}
\bar{\mathfrak{f}}(t):=\max_{ |\xi|\geq t}|\xi|\int_{|\xi|}^\infty|\scalar{\mathrm{D}^2f( s\frac{\xi}{|\xi|})\frac{\xi}{|\xi|}}{\frac{\xi}{|\xi|}}|\dd s
\end{align*}
fulfills $\lim_{t\rightarrow +\infty}\bar{\mathfrak{f}}(t)=0$ by the Assumption \ref{B2} and \begin{align}\label{lim df}
\scalar{\mathrm{D}_\xi f(\xi)-\mathrm{D}_\xi f^\infty(\xi)}{\xi}\geq -\bar{\mathfrak{f}}(\xi).
\end{align}
%
%
%
%
We now claim that for every $t>0$ there is an $A(t)$ with $\lim_{t\searrow 0} A(t)=0$ such that whenever there are $\xi\in \R^{n\times d}$ and $v_1\in S^{n-1}$ and $v_2\in S^{d-1}$ fulfilling \begin{equation}\begin{aligned}
\left|\mu_1(v_1)\otimes \mu_2(v_2)-\mathrm{D}_\xi f(\xi)\right|\leq t,  \quad \text{then}\quad \scalar{\mathrm{D}_\xi f(\xi)-\mathrm{D}_\xi f^\infty(\xi)}{\xi}\geq -A(t).\label{A stat}
\end{aligned}\end{equation}
Indeed, if this were not the case, there would be sequences $\xi_m\in \R^{n\times d}$ and  $v_{1,m}\in S^{n-1}$ and $v_{2,m}\in S^{d-1}$, depending on $t$, with \begin{align*}
\mu_1(v_{1,m})\otimes \mu_2(v_{2,m})-\mathrm{D}_\xi f(\xi_m)\rightarrow 0.
\end{align*}
and 
\begin{align}\limsup_{m\rightarrow +\infty} \scalar{\mathrm{D}_\xi f(\xi_m)-\mathrm{D}_\xi f^\infty(\xi_m)}{\xi_m}<0.\label{bu1}
\end{align}
However, by \eqref{lim df}, this is only possible if $\xi_m$ is a bounded sequence and hence these sequences have a limit, at least along a subsequence. Furthermore, $|\xi_m|$ must have a positive lower bound by \eqref{bu1} and because the derivatives of $f$ and $f^\infty$ are uniformly bounded (see e.g.\ \eqref{ball subdiff}).

Since $\mu_1,\mu_2,\mathrm{D}_\xi f$ and $\mathrm{D}_\xi f^\infty$ are all continuous (away from $0$) by assumption, this yields that there are $v_{1,\infty},v_{2,\infty}$ and $\xi_{\infty}\neq 0$ with \begin{align*}
\mu_1(v_{1,\infty})\otimes \mu_2(v_{2,\infty })=\mathrm{D}_\xi f^\infty(v_{1,\infty}\otimes v_{2,\infty})=\mathrm{D}_\xi f(\xi_\infty)
\end{align*}
and \begin{align*}
\mathrm{D}_\xi f(\xi_\infty)\neq \mathrm{D}_\xi f^\infty(\xi_\infty),
\end{align*}
which is a contradiction to the last point in Assumption \ref{B2}.

The proof of the Lemma now proceeds by first noting that the singular part of the measure is non-negative for the same reason as above in the previous proof with the slight change that, because $\bar{z}_x$ is not $C^1$, one has to use that by \eqref{anz rep}  it holds that \begin{align*}
\frac{\mathrm{d}(\bar{z}_x,\mathrm{D}u)}{\mathrm{d}|\mathrm{D}^s u|}(y)\in\left\{z_0:\frac{\mathrm{d}\mathrm{D}^s u}{\mathrm{d}|\mathrm{D}^s u|}(y)\,\big|\, z_0\in \overline{\ran(\de_\xi f)}\right\}
\end{align*} 
(up to a zero set) which still yields the non-negativity of the singular part by \eqref{doal}.

%
%
We now estimate the integral over the absolutely continuous part, using \eqref{cv2}, as \begin{align*}
\mel\int_{\tilde{D}_x^\eps}\left|\scalar{z(y)-\bar{z}_x(y)}{\mathrm{D}^a u(y)}_-\right|\dy\geq \int_{\tilde{D}_x^\eps}\scalar{\mathrm{D}_\xi f(\mathrm{D}^au(y))-\mathrm{D}_\xi f^\infty(\mathrm{D}^au(y))}{\mathrm{D}^au(y)}\dy\\
&\geq-\int_{\tilde{D}_x^\eps}A(|z-\mu_1(v)\otimes \zeta_x|)\dy.
\end{align*}
Here we have used \eqref{easy est} in the first step, as well as \eqref{A stat} in the second step.

This shows the Lemma with $k(t)=A(t)$.\hfill\qedsymbol

\subsection{Proof of the trace Lemma \ref{trace lemma}}\label{S56}
We note first that for any $w\in BV(D_x^\eps)$ we can partially integrate $|w|$ against the function $\eps^{-2}(\scalar{y-x}{\nu_x}+\frac{1}{2}\kappa\eps^2)$, which vanishes on $E_x^\eps$ by definition (see \eqref{def E1}), hence
 \begin{align*}
\mel\eps^{-2}\int_{D_x^\eps}|w|\dy=\int_{C_x^\eps}\eps^{-2}\left(\scalar{y-x}{\nu_x}+\frac{1}{2}\kappa\eps^2\right)\scalar{\nu_x}{\nu_y}|w|(y)\dH(y)\\
&-\int_{D_x^\eps}\eps^{-2}\left(\scalar{y-x}{\nu_x}+\frac{1}{2}\kappa\eps^2\right)\nu_x\cdot\mathrm{D}|w|(y).
\end{align*} 
Since $D_x^\eps$ is contained in a strip of width $\frac{1}{2}\kappa\eps^2$ in direction $\nu_x$ by definition, the second integral on the right-hand side is $\lesssim |\mathrm{D}|w||(D_x^\eps)\leq|\mathrm{D}w|(D_x^\eps)$. In the first integral, on the  other hand, it holds that $\scalar{\nu_x}{\nu_y}\gtrsim 1$ by continuity of the normal, and in each $C_x^{\rho\eps}$ it holds that \begin{align*}
\eps^{-2}\left(\scalar{y-x}{\nu_x}+\frac{1}{2}\kappa\eps^2\right)\geq \frac{1}{2}\kappa(1-\rho^2)\gtrsim_{\rho} 1\end{align*}
by some straightforward geometric computations. Hence \begin{align*}
\int_{C_x^{\rho\eps}}|w|\dH\lesssim_\rho\eps^{-2}\int_{D_x^\eps}|w|\dy+|\mathrm{D}w|(D_x^\eps),
\end{align*}
which shows \eqref{tr1}.
%
%
\eqref{tr2} follows in the same way.


For \eqref{tr3} on the other hand, we estimate \begin{align*}
\int_{D_x^\eps}|w -h|\dy\leq\int_{D_x^\eps}|w-h(x)|+|h-h(x)|\dy,\label{h split}
\end{align*}
the first integrand is again dealt with using \eqref{tr2} since adding a constant does not change the total variation, yielding \begin{align}
\eps^{-2}\int_{D_x^\eps}|w-h(x)|\dy\lesssim \int_{C_x^\eps}|w-h(y)|\dH(y)+\int_{C_x^\eps}|h(x)-h(y)|\dH(y)+|\mathrm{D}w|(D_x^\eps), 
\end{align}
where the second term is estimated as \begin{align*}
\int_{C_x^\eps}|h(x)-h(y)|\dH(y)\leq \omega(\eps)\mathcal{H}^{d-1}(C_x^\eps)\lesssim \omega(2\eps)\eps^{d-1},
\end{align*}
since $C_x^\eps$ is contained in a ball of radius $\eps$ by definition and where we used the bound \eqref{est C1}.
The second summand in \eqref{h split} is estimated, using \eqref{est D1}, as \begin{align*}
\int_{D_x^\eps}|h-h(x)|\dy\leq \omega(2\eps)\mathcal{L}^{d}(D_x^\eps)\approx\omega(2\eps)\eps^{d+1},
\end{align*}
since $\diam(D_x^\eps)\leq 2\eps$ by definition.\hfill\qedsymbol

\section{Proof of Corollary \ref{cor sans l2}}\label{S6}
We use the following Lemma.

\begin{lemma}
Consider the relaxed problem \eqref{gen rel prob} for $g=h=\lambda=0$ and $f(x,\xi)=|\xi|$.
Let $u$ be a minimizer of the relaxed problem \eqref{gen rel prob} for some given $u_0\in L^1(\Omega)$.

 Then for every $b\in \R_{\geq 0}$, the function $T_b(u)$ is a minimizer of the relaxed problem \eqref{gen rel prob} with boundary datum $T_b(u_0)$, where $T_b$ is the truncation as defined in \eqref{def trunc}.
\end{lemma}

The corollary follows directly from this Lemma, since we know from Proposition \ref{ex min} that for a given $u_0\in L^1(\Omega)$ a minimizer of the relaxed problem exists and we know from Theorem \ref{T1} and the Lemma that $T_b(u)=T_b(u_0)$ on $\de \Omega$ for every $b$, yielding that in fact $u=u_0$ on $\de \Omega$.



\begin{proof}[Proof of the Lemma]
To show the claim, we use the characterisation of minimizers of the relaxed problem from \cite[Thm.\ 2.5]{mazon2014functions}, which states that $u\in BV(\Omega)$ is a minimizer of the relaxed problem, if and only if there is a  $z\in X_d(\Omega, \R^d)$ such that \begin{align}
&\norm{z}_{L^\infty(\Omega,\R^d)}\leq 1\label{gv 1}\\
&\div(z)=0\label{gv 2}& \text{in $\Omega$}\\
&(z,\mathrm{D}u)=|\mathrm{D}u|& \text{in $\Omega$}\label{gv 3}\\
&[z,\nu]\in \sgn(u_0-u) &\text{a.e.\ on $\de\Omega$,}\label{gv 4}
\end{align}
where $\sgn=\de |\cdot|$. It suffices to show that if \eqref{gv 1}-\eqref{gv 4} hold for $u$ and $u_0$, then they also hold for $T_b(u)$ and $T_b(u_0)$ with the same $z$. For \eqref{gv 1}, \eqref{gv 2}, and \eqref{gv 4} this is elementary to check.
For \eqref{gv 3}, we can use that by the coarea formula it holds that $|\mathrm{D}u|=|\mathrm{D}T_b(u)|+|\mathrm{D}(u-T_b(u))|$ (see e.g.\ \cite[Lemma 3]{andreu2001minimizing} for details), and hence we can estimate, using \eqref{gv 3} and \eqref{anz bas est} \begin{align*}
|\mathrm{D}u|=(z,\mathrm{D}u)=(z,\mathrm{D}T_b(u))+(z,\mathrm{D}(u-T_b(u)))\leq |\mathrm{D}T_b(u)|+|\mathrm{D}(u-T_b(u))|=|\mathrm{D}u|.
\end{align*}
Since the first and last terms are the same, equality must hold in every step, which in particular implies that $(z,\mathrm{D}T_b(u))=|\mathrm{D}T_b(u)|$, showing the lemma.
\end{proof}


\section{Proof of Theorem \ref{T3}}\label{S7}
The theorem is a consequence of the following Proposition.

\begin{proposition}\label{ex bad f}
There exists a norm $f_0:\R^{2\times 2}\rightarrow \R_{\geq 0}$, fulfilling all the assumptions in \ref{B2} except \eqref{r1 split} and such that additionally $f_0\in C^\infty(\R^{2\times 2}\backslash \{0\})$, and there is an $\eps>0$ such that whenever $0\leq |b|<\eps|a|$, we have that \begin{align}
\mathrm{D}_{\xi}f_0((ae_1+be_2)\otimes (ae_1+be_2))=a\frac{(ae_1+be_2)\otimes e_1}{f_0((ae_1+be_2)\otimes (ae_1+be_2))}.\label{bad grad}
\end{align}
\end{proposition}
\begin{proof}[Proof of the theorem using Proposition \ref{ex bad f}]
We use $f_0$ as the integrand, by the proposition, it has the desired regularity and convexity properties. Since it is homogeneous, we trivially have $f_0=f_0^\infty$.

We then define the boundary datum $u_0:\de B_1(0)\rightarrow \R^2$ as follows: We set \begin{align*}
u_0(x)=\mathds{1}_{x_1>0}\eta(x_2)x
\end{align*}
where $\eta$ is smooth, non-negative, supported in a small neighborhood of size $<<\eps$ of $0$ and not identically $0$.
Then, by the $0$-homogeneity of $\mathrm{D}_\xi f_0$ and \eqref{bad grad} it holds that \begin{equation}\begin{aligned}
\mel\mathrm{D}_{\xi}f_0(u_0(x)\otimes \nu_x)=\mathrm{D}_{\xi}f_0((x_1e_1+x_2e_2)\otimes (x_1e_1+x_2e_2))\\
&=x_1\frac{(x_1e_1+x_2e_2)\otimes e_1}{f_0((x_1e_1+x_2e_2)\otimes (x_1e_1+x_2e_2))},\label{z grad}
\end{aligned}\end{equation}
whenever $u_0\neq 0$ and where $x=(x_1,x_2)\in \de B_1(0)$. Since the boundary can be smoothly parametrized by $x_2$ in a neighborhood of $e_1$, there is a smooth function $\mathfrak{g}:\R\rightarrow \R^{2\times 2}$, such that \begin{align*}
\mathfrak{g}(x_2)=x_1\frac{(x_1e_1+x_2e_2)\otimes e_1}{f_0((x_1e_1+x_2e_2)\otimes (x_1e_1+x_2e_2))} \quad \text{ for  $(x_1,x_2)\in B_\eps(e_1)\cap \de B_1(0)$.}
\end{align*}
We now define a $z$, fulfilling the conditions in the characterisation of the subdifferential in Lemma \ref{vec char} for $v=0$ and $w=0$, which shows that $u=0$ is indeed a minimizer for this datum. We set \begin{align*}
z(x)=\bar{\eta}(x_2)\mathfrak{g}(x_2)
\end{align*}
where $\bar{\eta}$ is another smooth cutoff function, with values in $[0,1]$, which equals $1$ on the support of $\eta$ and is also supported in a neighborhood of size $<<\eps$ of $0$.

This $z$ is trivially divergence free, smooth and it fulfills \eqref{cv4} because $\mathfrak{g}$ takes values in $\ran(\de f_0)$ by definition and by the convexity of $\ran(\de f_0)$. It fulfills \eqref{cv2} because $u=0$ and because every element of $\ran(\de f_0)$ lies in the subdifferential at $0$ as explained in Section \ref{Sec conv ana}.
It fulfills the boundary condition \eqref{cv5} on the support of $u_0$ by \eqref{doal} and the construction of $z$, while on the set where $u_0=u=0$ the condition \eqref{cv5} is an empty statement.
 The condition \eqref{cv3} is trivial since $u=0$.
\end{proof}

\begin{proof}[Proof of the Proposition \ref{ex bad f}]
We denote the components of vectors $x\in \R^{2\times 2}$ by $(x_{11},x_{21},x_{12},x_{22})$ such that $x_{ij}$ corresponds to the direction $e_i\otimes e_j$. 
We first define a auxiliary function $\mathfrak{q}$ (depending on $\eps$) by restricting $2\frac{x_{21}^2x_{22}}{x_{11}}$ to the unit ball, multiplying it by a smooth non-negative cutoff function supported in a (symmetric) neighborhood of $\pm e_1\otimes e_1$ of size $ 20\eps$ which equals $1$ in a neighborhood of size $10\eps$, and then taking the $2$-homogeneous extension of this again.
This function has the following properties \begin{itemize}
\item $\mathfrak{q}(x_{11},x_{12},x_{21},x_{22})=2\frac{x_{21}^2x_{22}}{x_{11}}$  in a neighborhood of size $5\eps$ of $x=e_1\otimes e_1$ 
\item$ \mathfrak{q}(x_{11},x_{12},x_{21},x_{22})=0$ if  $40\eps |x_{11}|\leq \min(|x_{12}|,|x_{21}|,|x_{22}|)$
\item $\mathfrak{q}$ is $2$-homogeneous
\item $\mathfrak{q}$ is $C^\infty$ except at $0$ and \begin{align}
|\mathfrak{q}(x)|\lesssim |x|^2\eps^3\quad |\mathrm{D}^m\mathfrak{q}(x)|\lesssim_m \eps^{3-m}|x|^{2-m}\label{est q}
\end{align}
for $x\neq 0$ and all $m\in \N_{\geq 0}$.
\end{itemize}
%
%
We define \begin{align*}
Ax=x+x_{21}e_1\otimes e_2+x_{12}e_2\otimes e_1+x_{12}e_1\otimes e_2.
\end{align*}
The matrix $A$ is symmetric positive definite by direct calculation.

We then define \begin{align*}
f_{0,*}(x):=\sqrt{Ax\cdot x+\mathfrak{q}(x)}.
\end{align*}
We define $f_0$ as the dual norm of this, i.e., \begin{align*}
f_{0}(x):=\sup_{x^*:\, f_{0,*}(x^*)\leq 1}\scalar{x}{x^*},
\end{align*}
which both depends on $\eps>0$.

We now check the different properties.

\textbf{Step 1. $f_{0,*}$ is a well-defined norm for $\eps$ small enough.} We first note that, because $Ax\cdot x\gtrsim |x|^2$ and because of \eqref{est q}, we have $f_{0,*}(x)>0$ if $x\neq 0$ and $\eps$ is small enough. In particular, the square root is well-defined. Homogeneity is clear from the definition. 
%

Furthermore, $f_{0,*}^2$ is convex, thanks to \eqref{est q}, if $\eps$ is small enough and hence the set $\{x: f_{0,*}(x)\leq 1\}$ is convex, as it is a sublevel set of $f_{0,*}^2$, which implies the triangle inequality for $f_{0,*}$.

By classical duality theory, this shows that $f_{0}$ is a norm too. It further follows from the Lindenstrauss-Day theorem \cite[Section 1.e]{Lindenstrauss} and the fact that $f_{0,*}$ is smooth away from $0$ that $f_0$ is uniformly convex with a quadratic modulus, which in particular implies that $f_0^2$ is uniformly convex.

We also note that \eqref{almost hom} and the last condition in \ref{B2} trivially hold for this $f_0$ since it is $1$-homogeneous and, in particular, $f_0=f_0^\infty$ holds.

\textbf{Step 2. Regularity of $f_0$.}
It is classical that $\frac{1}{2}f_0^2$ is the Legendre transform of $\frac{1}{2}f_{0,*}^2$ see e.g.\ \cite[Example 3.27]{boyd2004convex}. We furthermore have that $\mathrm{D}^2(\frac{1}{2}f_{0,*}^2)=A+\frac{1}{2}\mathrm{D}^2\mathfrak{q}$ is an invertible matrix at every point thanks to \eqref{est q} for small enough $\eps>0$, which by e.g.\ the Hadamard global inverse function theorem \cite[Thm.\ 1]{ruzhansky2015global} shows that  $\mathrm{D}(\frac{1}{2}f_{0,*}^2)$ is invertible with an inverse which is $C^\infty$ away from $0$.

We can now use the formula \begin{align}
\mathrm{D}f^*(x)=(\mathrm{D}f)^{-1}(x)\label{grad dual}
\end{align}
which holds for any convex function $f$ with Legendre transform $f^*$, to conclude that $\mathrm{D}(\frac{1}{2}f_0^2)$ is $C^\infty$ away from $0$, showing that $f_0$ is $C^\infty$ away from $0$, since it is positive away from $0$ by Step 1.

\textbf{Step 3. \eqref{bad grad} holds.}
Using again \eqref{grad dual}, we see that   \begin{align*}
\mathrm{D}\left(\frac{1}{2}f_{0}^2\right)\big((ae_1+be_2)\otimes (ae_1+be_2)\big)=\mathrm{D}\left(\frac{1}{2}f_{0,*}^2\right)^{-1}\big((ae_1+be_2)\otimes (ae_1+be_2)\big).
\end{align*}
On the other hand whenever the condition $0\leq |b|<\eps a$ holds, we calculate \begin{align*}
\mel\mathrm{D}\left(\frac{1}{2}f_{0,*}^2\right)\big((a^2e_1+abe_2)\otimes e_1\big)=aA\left((ae_1+be_2)\otimes e_1\right)+\frac{a}{2}\mathrm{D}\mathfrak{q}((ae_1+be_2)\otimes e_1)\\
&=a^2e_1\otimes e_1+ab(e_2\otimes e_1+e_1\otimes e_2)+b^2e_2\otimes e_2,
\end{align*}
where we used that around such vectors it holds that $\mathfrak{q}(x)=2\frac{x_{21}^2x_{22}}{x_{11}}$ by definition. Together both equations show that \begin{align*}
\mathrm{D}\left(\frac{1}{2}f_{0}^2\right)\big((ae_1+be_2)\otimes (ae_1+be_2)\big)=(a^2e_1+abe_2)\otimes e_1
\end{align*}
and the chain rule reveals that \begin{align*}
\mel\mathrm{D}f_0\big((ae_1+be_2)\otimes (ae_1+be_2) \big)=\frac{\mathrm{D}(\frac{1}{2}f_{0}^2)((ae_1+be_2)\otimes (ae_1+be_2))}{f_0((ae_1+be_2)\otimes (ae_1+be_2))}\\
&=\frac{(a^2e_1+abe_2)\otimes e_1}{f_0((ae_1+be_2)\otimes (ae_1+be_2))}.
\end{align*}

\end{proof}

\appendix
\section{ Proof of Lemma \ref{fen lemma}}\label{appendix}

\renewcommand{\theequation}{a.\arabic{equation}}

We define the dual of $f^\infty(y,\cdot)$ as \begin{align}
f_\star(y,\xi^*)=\sup_{\xi:\, f^\infty(y,\xi)\leq 1} \scalar{\xi^*}{\xi}.\label{def pol}
\end{align}
By classical convex analysis (see e.g.\ \cite[Lemma 4.2]{han1999plasticity}), this can equivalently be written as \begin{align}
f_\star(y,\xi^*)=\inf \left\{t\,\big|\, \xi^*=\frac{1}{t}\mathrm{D}_\xi f^\infty(y,\xi) \text{ for some $\xi$}\right\}.\label{desc d ball}
\end{align}
Furthermore, by applying the bipolar theorem \cite[Thm.\ 4.1.5]{borwein2006convex} to the set $\{f^\infty(x,\cdot)\leq 1\}$ we see that, \begin{align}
f^\infty(y,\xi)=\sup_{\xi^*: f_\star(y,\xi^*)\leq 1}\scalar{\xi^*}{\xi}.\label{bipo}
\end{align}
Finally, by the growth bound on $f$ in \ref{A2} and the induced bound  \eqref{gb rec} for $f^\infty$, we have
\begin{align}
f^\infty(y,\xi)\approx |\xi|\label{l bd fin}
\end{align}
and
 \begin{align}
f_\star(y,\xi^*)\approx |\xi^*|\label{bd pol}
\end{align}
uniformly in $y$.
Then, even though these quantities are not norms as they are not necessarily even, the analogue of the statement ``uniform smoothness implies uniform convexity of the dual norm'' does still hold here, more precisely, for any given $t\in (0,1)$ we can estimate \begin{align*}
\mel\sup\left\{ f_\star(y,\xi_1^*+\xi_2^*)+tf_\star(y,\xi_1^*-\xi_2^*)\,\big|\, f_\star(y,\xi_1^*),f_\star(y,\xi_2^*)\leq 1\right\}\\
&=\sup\left\{ \scalar{\xi_1}{\xi_1^*+\xi_2^*}+t\scalar{\xi_2}{\xi_1^*-\xi_2^*}\,\big|\, f_\star(y,\xi_1^*),f_\star(y,\xi_2^*),f^\infty(y,\xi_1),f^\infty(y,\xi_2)\leq 1\right\}\\
&=\sup\left\{ \scalar{\xi_1+t\xi_2}{\xi_1^*}+\scalar{\xi_1-t\xi_2}{\xi_2^*}\,\big|\, f_\star(y,\xi_1^*),f_\star(y,\xi_2^*),f^\infty(y,\xi_1),f^\infty(y,\xi_2)\leq 1\right\}\\
&=\sup\left\{ f^\infty(y,\xi_1+t\xi_2)+f^\infty(y,\xi_1-t\xi_2)\,\big|\, f^\infty(y,\xi_1),f^\infty(y,\xi_2)\leq 1\right\},
\end{align*}
here we have used \eqref{bipo} in the final step. We claim that this supremum is $\leq 2+Ct^2$ for some $C$ not depending on $y$ or $t$.

To estimate the supremum we distinguish the three cases $f^\infty(y,\xi_1)\geq \frac{1}{2}$ and $f^\infty(y,\xi_1)\leq \frac{1}{2}$ and $t\geq \frac{1}{3}$. In the latter case, we have a trivial upper bound of $2+Ct^2$, simply by picking $C$ large enough.
In the second case, it trivially holds that \begin{align*}
f^\infty(y,\xi_1+t\xi_2)+f^\infty(y,\xi_1-t\xi_2)\leq \frac{5}{3}
\end{align*}
 thanks to the convexity and homogeneity if $t<\frac{1}{3}$. In the case $f^\infty(y,\xi_1)\geq \frac{1}{2}$ and $t\leq \frac{1}{3}$ we may use the lower bound \eqref{l bd fin} for $f^\infty$ to conclude that the line between $\xi_1+t\xi_2$ and $\xi_1-t\xi_2$ lies outside of some $B_{\frac{1}{C}}(0)$, with a $C$ not depending on $y$. This enables us to estimate the sum from above with $2+Ct^2$, thanks to $f^\infty$ being $C^2$ uniformly in $y$ away from $0$ by assumption and because $\xi_2\lesssim 1$ by assumption and the bound \eqref{l bd fin}.

Either way, we have obtained that \begin{align*}
\sup\left\{ f_\star(y,\xi_1^*+\xi_2^*)+tf_\star(y,\xi_1^*-\xi_2^*)\,\big|\, f_\star(y,\xi_1^*),f_\star(y,\xi_2^*)\leq 1\right\}\leq 2+Ct^2.\label{uni conv1}
\end{align*}
If we pick $t=\frac{1}{2C}f_\star(y,\xi_1^*-\xi_2^*)$, then we see that \begin{align}
f_\star(y,\frac{1}{2}\left(\xi_1^*+\xi_2^*\right))\leq 1-\frac{1}{8C}f_\star(y,\xi_1^*-\xi_2^*)^2.
\end{align}
for all $\xi_1,\xi_2$ with $f_\star(y,\xi_1^*),f_\star(y,\xi_2^*)\leq 1$. Coming back to the estimate \eqref{quant fen} in the lemma, we see that $f_*(y,\mathrm{D}_\xi f^\infty(y,v))=1$, because of \eqref{desc d ball} and \eqref{bipo}.
Therefore, using the definition \eqref{def pol} of $f_*$, we can estimate  \begin{align*}
\mel\scalar{\mathrm{D}_\xi f^\infty(y,v)-v^*}{v}=2\scalar{\mathrm{D}_\xi f^\infty(y,v)-\frac{v^*+\mathrm{D}_\xi f^\infty(y,v)}{2}}{v}\\
&\geq 2f^\infty(y,v)\left(1-f_\star(y,\frac{v^*+\mathrm{D}_\xi f^\infty(y,v)}{2})\right)\\
&\gtrsim f_\star(y,\mathrm{D}_\xi f^\infty(y,v)-v^*)^2\\
&\gtrsim |\mathrm{D}_\xi f^\infty(y,v)-v^*|^2.
\end{align*}
Here we have used \eqref{uni conv1} and that $f^\infty$ and $f_\star$ are equivalent to the modulus of their argument up to a constant by \eqref{l bd fin}-\eqref{bd pol} in the last two steps.

The vectorial case proceeds by exactly the same argument.\hfill\qedsymbol

\subsection*{Acknowledgment}
The author has received funding from the European Research Council (ERC) under the European Union’s Horizon 2020 research and innovation programme through the grant agreement 862342.

\bibliographystyle{abbrv}
  \bibliography{LeastGradient}

\begin{thebibliography}{10}

\bibitem{acerbi1994new}
E.~Acerbi and G.~Dal~Maso.
\newblock New lower semicontinuity results for polyconvex integrals.
\newblock {\em Calculus of Variations and Partial Differential Equations},
  2(3):329--371, 1994.

\bibitem{ambrosio2000functions}
L.~Ambrosio, N.~Fusco, and D.~Pallara.
\newblock {\em Functions of bounded variation and free discontinuity problems}.
\newblock Oxford university press, 2000.

\bibitem{andreu2001minimizing}
F.~Andreu, C.~Ballester, V.~Caselles, and J.~M. Maz{\'o}n.
\newblock Minimizing total variation flow.
\newblock {\em Differential and Integral Equations}, 14:321--360, 2001.

\bibitem{andreu2001dirichlet}
F.~Andreu, C.~Ballester, V.~Caselles, and J.~M. Maz{\'o}n.
\newblock {The Dirichlet problem for the Total Variation Flow}.
\newblock {\em Journal of Functional Analysis}, 180(2):347--403, 2001.

\bibitem{andreu2005cauchy}
F.~Andreu, V.~Caselles, and J.~M. Maz{\'o}n.
\newblock {The Cauchy problem for a strongly degenerate quasilinear equation}.
\newblock {\em Journal of the European Mathematical Society}, 7(3):361--393,
  2005.

\bibitem{andreu2004parabolic}
F.~Andreu-Vaillo, V.~Caselles, and J.~M. Maz{\'o}n.
\newblock {\em Parabolic quasilinear equations minimizing linear growth
  functionals}, volume 223.
\newblock Springer Science \& Business Media, 2004.

\bibitem{anzellotti1983pairings}
G.~Anzellotti.
\newblock Pairings between measures and bounded functions and compensated
  compactness.
\newblock {\em {Annali di Matematica Pura ed Applicata}}, 135(1):293--318,
  1983.

\bibitem{anzellotti1983traces}
G.~Anzellotti.
\newblock Traces of bounded vector fields and the divergence theorem.
\newblock {\em Unpublished preprint}, 1983.

\bibitem{bahouri1994equations}
H.~Bahouri and J.-Y. Chemin.
\newblock {\'E}quations de transport relatives {\'a} des champs de vecteurs
  non-lipschitziens et m{\'e}canique des fluides.
\newblock {\em {Archive for Rational Mechanics and Analysis}}, 127(2):159--181,
  1994.

\bibitem{baldo1991non}
S.~Baldo and L.~Modica.
\newblock Non uniqueness of minimal graphs.
\newblock {\em Indiana University Mathematics Journal}, 40(3):975--983, 1991.

\bibitem{bauschke}
H.~H. Bauschke and P.~L. Combettes.
\newblock {\em Convex analysis and monotone operator theory in Hilbert spaces}.
\newblock Springer, 2011.

\bibitem{beck2018globally}
L.~Beck, M.~Bul{\'\i}{\v{c}}ek, and E.~Maringov{\'a}.
\newblock {Globally Lipschitz minimizers for variational problems with linear
  growth}.
\newblock {\em ESAIM: Control, Optimisation and Calculus of Variations},
  24(4):1395--1413, 2018.

\bibitem{beck2013dirichlet}
L.~Beck and T.~Schmidt.
\newblock {On the Dirichlet problem for variational integrals in BV}.
\newblock {\em Journal f{\"u}r die reine und angewandte Mathematik (Crelles
  Journal)}, 674:113--194, 2013.

\bibitem{beck2015interior}
L.~Beck and T.~Schmidt.
\newblock {Interior gradient regularity for BV minimizers of singular
  variational problems}.
\newblock {\em Nonlinear Analysis: Theory, Methods \& Applications},
  120:86--106, 2015.

\bibitem{bernstein1910surfaces}
S.~Bernstein.
\newblock Sur les surfaces d{\'e}finies au moyen de leur courbure moyenne ou
  totale.
\newblock {\em Annales scientifiques de l'{\'E}cole Normale Sup{\'e}rieure},
  27:233--256, 1910.

\bibitem{bildhauer2003convex}
M.~Bildhauer.
\newblock Convex variational problems.
\newblock {\em Lecture Notes in Mathematics}, 2003.

\bibitem{blomgren1998color}
P.~Blomgren and T.~F. Chan.
\newblock {Color TV: Total variation methods for Restoration of vector-valued
  images}.
\newblock {\em {IEEE Transactions on Image Processing}}, 7(3):304--309, 1998.

\bibitem{borwein2006convex}
J.~Borwein and A.~Lewis.
\newblock {\em Convex Analysis and Nonlinear Optimization: Theory and
  Examples}.
\newblock Springer, 2006.

\bibitem{boyd2004convex}
S.~P. Boyd and L.~Vandenberghe.
\newblock {\em Convex optimization}.
\newblock Cambridge university press, 2004.

\bibitem{brezis2019remarks}
H.~Brezis.
\newblock {Remarks on some minimization problems associated with BV norms.}
\newblock {\em Discrete \& Continuous Dynamical Systems: Series A}, 39(12),
  2019.

\bibitem{bulivcek2015existence}
M.~Bul{\'\i}{\v{c}}ek, J.~M{\'a}lek, K.~R. Rajagopal, and J.~R. Walton.
\newblock {Existence of solutions for the anti-plane stress for a new class of
  “strain-limiting” elastic bodies}.
\newblock {\em Calculus of Variations and Partial Differential Equations},
  54(2):2115--2147, 2015.

\bibitem{caselles2011regularity}
V.~Caselles, A.~Chambolle, and M.~Novaga.
\newblock Regularity for solutions of the total variation denoising problem.
\newblock {\em Revista Matem{\'a}tica Iberoamericana}, 27(1):233--252, 2011.

\bibitem{chan2005variational}
T.~F. Chan and J.~Shen.
\newblock Variational image inpainting.
\newblock {\em Communications on Pure and Applied Mathematics}, 58(5):579--619,
  2005.

\bibitem{corsato2016dirichlet}
C.~Corsato, C.~De~Coster, and P.~Omari.
\newblock {The Dirichlet problem for a prescribed anisotropic mean curvature
  equation: existence, uniqueness and regularity of solutions}.
\newblock {\em Journal of Differential Equations}, 260(5):4572--4618, 2016.

\bibitem{crasta2019anzellotti}
G.~Crasta and V.~De~Cicco.
\newblock {Anzellotti's pairing theory and the Gauss--Green theorem}.
\newblock {\em Advances in Mathematics}, 343:935--970, 2019.

\bibitem{devore1984maximal}
R.~A. DeVore and R.~C. Sharpley.
\newblock {\em Maximal functions measuring smoothness}, volume 293.
\newblock American Mathematical Soc., 1984.

\bibitem{di2012hitchhikers}
E.~Di~Nezza, G.~Palatucci, and E.~Valdinoci.
\newblock {Hitchhiker's guide to the fractional Sobolev spaces}.
\newblock {\em Bulletin des sciences math{\'e}matiques}, 136(5):521--573, 2012.

\bibitem{dweik2019p}
S.~Dweik and F.~Santambrogio.
\newblock {$L^p$ bounds for boundary-to-boundary transport densities, and
  $W^{1, p}$ bounds for the BV least gradient problem in 2D}.
\newblock {\em Calculus of Variations and Partial Differential Equations},
  58(1):31, 2019.

\bibitem{Evans}
L.~C. Evans.
\newblock {\em Partial differential equations}, volume~19.
\newblock American mathematical society, 2022.

\bibitem{figueiredo2018nehari}
G.~M. Figueiredo and M.~T. Pimenta.
\newblock {Nehari method for locally Lipschitz functionals with examples in
  problems in the space of bounded variation functions}.
\newblock {\em Nonlinear Differential Equations and Applications NoDEA},
  25(5):47, 2018.

\bibitem{fonseca2007modern}
I.~Fonseca and G.~Leoni.
\newblock {\em {Modern Methods in the Calculus of Variations: $L^p$ spaces}}.
\newblock Springer, 2007.

\bibitem{fussangel2024singular}
L.~Fussangel, B.~Priyasad, and P.~Stephan.
\newblock On the singular set of $\mathrm{BV}$ minimizers for non-autonomous
  functionals.
\newblock {\em arXiv preprint arXiv:2412.14997, To appear in Communications in
  Contemporary Mathematics}, 2024.

\bibitem{Giaquinta1979}
M.~Giaquinta, G.~Modica, and J.~Souček.
\newblock {Functionals with linear growth in the calculus of variations. I.}
\newblock {\em Commentationes Mathematicae Universitatis Carolinae},
  020(1):143--156, 1979.

\bibitem{giusti1976boundary}
E.~Giusti.
\newblock Boundary value problems for non-parametric surfaces of prescribed
  mean curvature.
\newblock {\em Annali della Scuola Normale Superiore di Pisa-Classe di
  Scienze}, 3(3):501--548, 1976.

\bibitem{giusti1984minimal}
E.~Giusti.
\newblock {\em Minimal surfaces and functions of bounded variation}.
\newblock Springer, 1984.

\bibitem{gmeineder2019partial}
F.~Gmeineder and J.~Kristensen.
\newblock {Partial regularity for BV minimizers}.
\newblock {\em Archive for Rational Mechanics and Analysis}, 232(3):1429--1473,
  2019.

\bibitem{gorny2018planar}
W.~G{\'o}rny.
\newblock Planar least gradient problem: existence, regularity and anisotropic
  case.
\newblock {\em Calculus of Variations and Partial Differential Equations},
  57(4):98, 2018.

\bibitem{gorny2021existence}
W.~G{\'o}rny.
\newblock Existence of minimisers in the least gradient problem for general
  boundary data.
\newblock {\em Indiana University Mathematics Journal}, 70(3):1003--1037, 2021.

\bibitem{gorny2024functions}
W.~G{\'o}rny and J.~M. Maz{\'o}n.
\newblock {\em Functions of Least Gradient}.
\newblock Springer, 2024.

\bibitem{gorny2022duality}
W.~G{\'o}rny and J.~M. Maz{\'o}n.
\newblock A duality-based approach to gradient flows of linear growth
  functionals.
\newblock {\em Publicacions matematiques}, 69(2):341--365, 2025.

\bibitem{gorny2017special}
W.~G{\'o}rny, P.~Rybka, and A.~Sabra.
\newblock Special cases of the planar least gradient problem.
\newblock {\em Nonlinear Analysis: Theory, Methods \& Applications},
  151:66--95, 2017.

\bibitem{hakkarainen2015stability}
H.~Hakkarainen, R.~Korte, P.~Lahti, and N.~Shanmugalingam.
\newblock Stability and continuity of functions of least gradient.
\newblock {\em Analysis and Geometry in Metric Spaces}, 3(1):123--139, 2015.

\bibitem{han1999plasticity}
W.~Han and B.~D. Reddy.
\newblock {\em {Plasticity: Mathematical Theory and Numerical Analysis}}.
\newblock Springer, 2013.

\bibitem{jerrard2018existence}
R.~L. Jerrard, A.~Moradifam, and A.~I. Nachman.
\newblock Existence and uniqueness of minimizers of general least gradient
  problems.
\newblock {\em Journal f{\"u}r die reine und angewandte Mathematik (Crelles
  Journal)}, 734:71--97, 2018.

\bibitem{kawohl2007dirichlet}
B.~Kawohl and F.~Schuricht.
\newblock {Dirichlet problems for the 1-Laplace operator, including the
  eigenvalue problem}.
\newblock {\em Communications in Contemporary Mathematics}, 9(04):515--543,
  2007.

\bibitem{kohn1986constrained}
R.~V. Kohn and G.~Strang.
\newblock The constrained least gradient problem.
\newblock In {\em Non-Classical Continuum Mechanics. Proceedings of the London
  Mathematical Society Symposium}, pages 226--243, Durham, July 1986.

\bibitem{kreuml2019fractional}
A.~Kreuml and O.~Mordhorst.
\newblock {Fractional Sobolev norms and BV functions on manifolds}.
\newblock {\em Nonlinear Analysis}, 187:450--466, 2019.

\bibitem{leonardi2015overview}
G.~P. Leonardi.
\newblock {An overview on the Cheeger problem}.
\newblock {\em {New Trends in Shape Optimization}}, pages 117--139, 2015.

\bibitem{Lindenstrauss}
J.~Lindenstrauss and L.~Tzafriri.
\newblock {\em Classical Banach spaces II: function spaces}, volume~97.
\newblock Springer Science \& Business Media, 2013.

\bibitem{lledos2025study}
B.~Lledos.
\newblock Study of solutions for a minimization problem with linear growth
  under bounded slope condition.
\newblock {\em Hal preprint hal-04993725}, 2025.

\bibitem{marcellini2006nonlinear}
P.~Marcellini and G.~Papi.
\newblock Nonlinear elliptic systems with general growth.
\newblock {\em Journal of Differential Equations}, 221(2):412--443, 2006.

\bibitem{mazon2016nonlocal}
J.~M. Maz{\'o}n, M.~P{\'e}rez-Llanos, J.~D. Rossi, and J.~Toledo.
\newblock {A nonlocal 1-Laplacian problem and Median Values}.
\newblock {\em Publicacions matematiques}, pages 27--53, 2016.

\bibitem{mazon2014functions}
J.~M. Maz{\'o}n, J.~D. Rossi, and S.~S. De~Le{\'o}n.
\newblock Functions of least gradient and 1-harmonic functions.
\newblock {\em Indiana University Mathematics Journal}, pages 1067--1084, 2014.

\bibitem{mercier2018continuity}
G.~Mercier.
\newblock {Continuity results for TV-minimizers}.
\newblock {\em Indiana University Mathematics Journal}, 67:1499--1545, 2018.

\bibitem{Meyer}
D.~Meyer.
\newblock Total $\mathbb{A}$-variation-type flows for general integrands.
\newblock {\em arXiv preprint arXiv:2310.15283}, 2023.

\bibitem{miranda1971principio}
M.~Miranda.
\newblock Un principio di massimo forte per le frontiere minimali e una sua
  applicazione alla risoluzione del problema al contorno per l'equazione delle
  superfici di area minima.
\newblock {\em Rendiconti del Seminario Matematico della Universit{\`a} di
  Padova}, 45:355--366, 1971.

\bibitem{miranda1974dirichlet}
M.~Miranda.
\newblock {Dirichlet problem with $L^1$ data for the non-homogeneous minimal
  surface equation}.
\newblock {\em Indiana University Mathematics Journal}, 24(3):227--241, 1974.

\bibitem{moradifam2018existence}
A.~Moradifam.
\newblock Existence and structure of minimizers of least gradient problems.
\newblock {\em Indiana University Mathematics Journal}, pages 1025--1037, 2018.

\bibitem{muller2016density}
J.-S. M{\"u}ller.
\newblock {A density result for Sobolev functions and functions of higher order
  bounded variation with additional integrability constraints}.
\newblock {\em Annales Fennici Mathematici}, 41(2):789--801, 2016.

\bibitem{parks1977explicit}
H.~Parks.
\newblock Explicit determination of area minimizing hypersurfaces.
\newblock {\em {Duke Mathematical Journal}}, 44:519--534, 1977.

\bibitem{rao1991theory}
M.~M. Rao and Z.~D. Ren.
\newblock {Theory of Orlicz spaces}.
\newblock 1991.

\bibitem{rudin1992nonlinear}
L.~I. Rudin, S.~Osher, and E.~Fatemi.
\newblock Nonlinear total variation based noise removal algorithms.
\newblock {\em {Physica D: Nonlinear phenomena}}, 60(1-4):259--268, 1992.

\bibitem{ruzhansky2015global}
M.~Ruzhansky and M.~Sugimoto.
\newblock On global inversion of homogeneous maps.
\newblock {\em Bulletin of Mathematical Sciences}, 5(1):13--18, 2015.

\bibitem{rybka2025dirichlet}
P.~Rybka.
\newblock {On the Dirichlet Problem for the One-Dimensional
  Rudin--Osher--Fatemi Functional}.
\newblock {\em Mathematical Methods in the Applied Sciences}, 0:1--13, 2025.

\bibitem{santi1972sul}
E.~Santi.
\newblock Sul problema al contorno per l’equazione delle superfici di area
  minima su domini limitati qualunque.
\newblock {\em Annali dell’Universit{\`a} di Ferrara}, 17(1):13--26, 1972.

\bibitem{serrin1968dirichlet}
J.~Serrin and H.~Jenkins.
\newblock {The Dirichlet problem for the minimal surface equation in higher
  dimensions.}
\newblock {\em Journal f{\"u}r die reine und angewandte Mathematik},
  229:170--187, 1968.

\bibitem{spradlin2014not}
G.~S. Spradlin and A.~Tamasan.
\newblock Not all traces on the circle come from functions of least gradient in
  the disk.
\newblock {\em Indiana University Mathematics Journal}, pages 1819--1837, 2014.

\bibitem{sternberg1992existence}
P.~Sternberg, G.~Williams, and W.~Ziemer.
\newblock Existence, uniqueness, and regularity for functions of least
  gradient.
\newblock {\em Journal f{\"u}r die reine und angewandte Mathematik},
  430:35--60, 1992.

\bibitem{zuniga2019continuity}
A.~Zuniga.
\newblock Continuity of minimizers to weighted least gradient problems.
\newblock {\em Nonlinear Analysis}, 178:86--109, 2019.

\end{thebibliography}


\end{document}